\documentclass[12pt,notitlepage]{amsart}
\usepackage{latexsym,amsfonts,amssymb,amsmath,amsthm}
\usepackage{graphicx}
\usepackage{xcolor}
\usepackage[normalem]{ulem}
\usepackage{graphicx}
\usepackage{color}
\usepackage{datetime2}
\usepackage{mathrsfs}
\usepackage{amsmath, amsthm, amsfonts, amssymb, cite}
\usepackage{wrapfig}
\usepackage{stackrel}
\usepackage{color}
\usepackage{float}
\usepackage{enumitem}
\usepackage{mathtools}
\usepackage{multicol}
\let\turc\c
\usepackage{etoolbox}
\robustify\turc %new thing to get the c in Topacogullari
\renewcommand{\c}{\mathfrak{c}}
\newcommand{\bpm}{\begin{pmatrix}}
\newcommand{\epm}{\end{pmatrix}}

\usepackage[inner=1.0in,outer=1.0in,bottom=1.0in, top=1.0in]{geometry}

\newcommand{\mz}{\ensuremath{\mathbb Z}}
\newcommand{\mr}{\ensuremath{\mathbb R}}

\newcommand{\Q}{\ensuremath{\mathbb Q}}
\newcommand{\mymod}{\ensuremath{\negthickspace \negmedspace \pmod}}
\newcommand{\shortmod}{\ensuremath{\negthickspace \negthickspace \negthickspace \pmod}}

\newcommand{\intR}{\int_{-\infty}^{\infty}}

\newcommand{\sumstar}{\sideset{}{^*}\sum}

\newcommand{\cond}{\mathrm{cond}}
\newcommand{\redu}{\mathrm{red}}

\newcommand{\hit}{\mathrm{ht}}

\DeclareMathOperator{\sgn}{sgn}

\theoremstyle{plain}		
	\newtheorem{mytheo}{Theorem} [section]
	
	\newtheorem{myprop}[mytheo]{Proposition}
	\newtheorem{mycoro}[mytheo]{Corollary}
     \newtheorem{mylemma}[mytheo]{Lemma}

\theoremstyle{definition}
	\newtheorem{defiplaintext}[mytheo]{Definition}

\theoremstyle{remark}

\numberwithin{equation}{section}
\numberwithin{figure}{section}

\begin{document}

\author{Matthew P. Young}
 \address{Department of Mathematics \\
 	  Texas A\&M University \\
 	  College Station \\
 	  TX 77843-3368 \\
 		U.S.A.}

 \email{myoung@math.tamu.edu}
 \thanks{This material is based upon work supported by the National Science Foundation under agreement No. DMS-2001306.  Any opinions, findings and conclusions or recommendations expressed in this material are those of the authors and do not necessarily reflect the views of the National Science Foundation.  }

\begin{abstract} 
We prove an essentially optimal large sieve inequality for self-dual Eisenstein series of varying levels.  This bound can alternatively be interpreted as a large sieve inequality for rationals ordered by height.  The method of proof is recursive, and has some elements in common with Heath-Brown's quadratic large sieve, and the asymptotic large sieve of Conrey, Iwaniec, and Soundararajan.
\end{abstract}

 \title{The large sieve for self-dual Eisenstein series of varying levels
%  \\
% Compiled: \DTMnow
}

\dedicatory{To Henryk Iwaniec, with admiration and gratitude, on the occasion of his 75th birthday}
\maketitle
\section{Introduction}
\subsection{Setting up the problem}
A general large sieve inequality is an upper bound on the operator norm of an arithmetically-defined matrix $\Lambda = (\lambda_{m,n})$, with $\lambda_{m,n} \in \mathbb{C}$.  Define the norm of $\Lambda$, denoted $\|\Lambda \|$, by
\begin{equation*}
\|\Lambda \| =  \max_{|{\bf \alpha}| = 1} \sum_{m} \Big| \sum_n \alpha_n \lambda_{m,n} \Big|^2,
\qquad {\bf \alpha} = (\alpha_n).
\end{equation*}
The duality principle implies that $\| \Lambda \| = \|\Lambda^t \|$, where
$\Lambda^t = (\lambda_{n,m})$.
%
%\begin{equation}
%\|\Lambda^*\| =  \max_{|{\bf \beta}| = 1} \sum_{n} \Big| \sum_m \lambda_{m,n} \beta_m \Big|^2,
%\qquad {\bf \beta} = (\beta_m).
%\end{equation}

A particularly interesting choice of $\lambda_{m,n}$ is $\lambda_f(n)$, where $f$ ranges over a family $\mathcal{F}$ of automorphic forms or $L$-functions, $n$ ranges over an interval of positive integers, say $N/2 < n \leq N$, and $\lambda_f(n)$ is the $n$-th Dirichlet series coefficient of the $L$-function $L(f,s)$.  In this case, we write $\Delta(\mathcal{F}, N)$ for the norm of this large sieve matrix, namely
\begin{equation}
 \Delta(\mathcal{F}, N) = \max_{|{\bf \alpha}| = 1} \sum_{f \in \mathcal{F}} \Big| \sum_{N/2 < n \leq N} \alpha_n \lambda_f(n)  \Big|^2.
\end{equation}
The dual norm $\Delta^*(\mathcal{F}, N)$ is given by
\begin{equation}
 \Delta^*(\mathcal{F}, N) = \max_{|{\bf \beta}| = 1} 
 \sum_{N/2 < n \leq N}
  \Big| \sum_{f \in \mathcal{F}} \beta_f \lambda_f(n)  \Big|^2.
\end{equation}

The classical multiplicative large sieve inequality concerns the case where $\lambda_{f}(n) = \chi(n)$, and where the family runs over primitive Dirichlet characters $\chi$ modulo $q$, with $q \leq Q$.  Applications include the Bombieri-Vinogradov theorem, estimates for moments of $L$-functions, zero density estimates, and a variety of sieving problems. See \cite{Montgomery} for details.

There are many works on large sieve inequalities for other families.  For instance, Deshouillers and Iwaniec \cite{DeshouillersIwaniec}  obtained a sharp bound for cusp forms on $GL_2$, which in turn has been a powerful tool in studying statistical properties of the Riemann zeta function on the critical line.  Heath-Brown \cite{HB} showed an essentially optimal upper bound on the sparse sub-family of quadratic Dirichlet characters.  Many state of the art works on quadratic twists of modular forms, with elliptic curves being of particular interest, have relied on Heath-Brown's bound.

In this paper, we are interested in the following family $\mathcal{F}$. 
For any Dirichlet character $\psi$ modulo $r$ and real number $t$, define
\begin{equation*}
 \lambda_{\psi, t}(a,b) =  \psi(a) \overline{\psi}(b) (a/b)^{it}.
\end{equation*}
Here $\sum_{ab=n} \lambda_{\psi, t}(a,b) =: \lambda_{\psi,t}(n)$ is the $n$-th Hecke eigenvalue of a self-dual Eisenstein series on $\Gamma_0(r^2)$, and when $\psi$ is primitive, the Eisenstein series is a newform.
%Moreover, $\sum_{a,b \geq 1} \lambda_{\psi, t}(a,b)(ab)^{-s} = L(s+it, \psi) L(s-it, \overline{\psi})$, a product of Dirichlet $L$-functions.
Let $k$ be a positive integer, and let $\theta$ run over all Dirichlet characters modulo $k$.  Let $Q \geq 1$ be a real number, and for each $Q/2 < q \leq Q$ with $(q,k) =1$, let $\chi$ run over primitive Dirichlet characters modulo $q$.  Finally, let $T \geq 1$ be a real number, and let $|t| \leq T$.  Then define $\mathcal{F}$ to consist of the characters $\chi \theta$, with corresponding data $\lambda_{\chi \theta, t}(a,b)$, with $N/2 < ab \leq N$ and $(a,b) = 1$.  We write
\begin{equation}
\label{eq:DeltaDef}
 \Delta(Q, k, T, N)
 = \max_{|{\bf \alpha}| = 1} 
 \int_{T/2 \leq t \leq T}
\sum_{\substack{Q/2 < q \leq Q \\ (q,k) = 1}}
\thinspace
\sumstar_{\chi \shortmod{q}}
\thinspace
\sum_{\theta \shortmod{k}}
\Big|
\sum_{\substack{N/2 < ab \leq N \\ (a,b) = 1}} \alpha_{a,b} \lambda_{\chi \theta,t}(a,b) 
\Big|^2 dt,
\end{equation}
which agrees with $\Delta(\mathcal{F}, N)$ for this family $\mathcal{F}$.  
The dual norm $\Delta^*(Q,k,T,N)$ is given by
\begin{equation}
\label{eq:Delta*Def}
\Delta^*(Q,k,T,N) = \max_{|{\bf \beta}| = 1}
\sum_{\substack{N/2< ab \leq N \\ (a,b) = 1}} 
\Big|
\int_{T/2 \leq t \leq T}
\sum_{\substack{Q/2 < q \leq Q \\ (q,k) = 1}}
\thinspace
\sumstar_{\chi \shortmod{q}}
\thinspace
\sum_{\theta \shortmod{k}}
\beta_{\chi, \theta, t}
 %\chi \theta(a) \overline{\chi \theta}(b) (a/b)^{it} 
\lambda_{\chi \theta, t}(a,b)
 dt
\Big|^2.
\end{equation}

As a `trivial' bound, which we mainly state for a frame of reference, one may deduce from the classical large sieve inequality the bound
\begin{equation}
\label{eq:DeltaTrivialBound}
\Delta(Q, k, T, N) \ll ( Q^2 k T \sqrt{N} + N \log{N}).
\end{equation}
Deducing the estimate \eqref{eq:DeltaTrivialBound} uses the idea of the Dirichlet hyperbola method, by summing over $a \leq \sqrt{N}$ trivially, and applying the classical large sieve to the sum over $b \ll N/a$.  

The condition $(a,b) = 1$ may be easily overlooked, yet it is vital.
The above sketch shows that the trivial bound \eqref{eq:DeltaTrivialBound} holds even without the condition $(a,b) = 1$.  In fact, if the condition $(a,b) = 1$ were to be omitted in \eqref{eq:DeltaDef}, then the term of size $Q^2 k T \sqrt{N}$ in \eqref{eq:DeltaTrivialBound} would not be removable, because one could choose $\alpha_{a,b}$ to be the indicator function that $a=b$ in \eqref{eq:DeltaDef}.
For this, note $\lambda_{\chi, t}(a,a)  = 1$ for $a$ coprime to the modulus of $\chi$.
Therefore, any substantial improvement over this trivial bound must use the condition $(a,b) = 1$.  The restriction $(a,b) =1$ is similar in spirit to the (necessary) square-free restriction when studying quadratic characters, as in \cite{HB}; for more on this point, see Section \ref{section:recursivethmFEsketch}.  We also observe that choosing $\alpha_{a,b} = \alpha_{ab}$ to depend only on the product $ab$ would give rise to sums of the form $\sum_n \alpha_n \lambda_{\psi,t}(n)$ appearing in \eqref{eq:DeltaDef}.  Then considering $n=p^2$ would lead to a large term as discussed above.

\subsection{Main results, and discussion}
%Our main result proves an essentially optimal upper bound on $\Delta$.
\begin{mytheo}
\label{thm:mainthm}
We have
\begin{equation}
\label{eq:mainthm}
 \Delta(Q,k,T,N) \ll_{\varepsilon} (QkTN)^{\varepsilon}(Q^2 k T + N).
\end{equation}
\end{mytheo}
This estimate is optimal (up to the $\varepsilon$-aspect) by general principles (see \cite[Chapter 7]{IK}).  We may interpret this as a spectral large sieve inequality for the family of trivial nebentypus newform Eisenstein series on $\Gamma_0(q^2 k^2)$, with varying level $q$.  
Theorem \ref{thm:mainthm} appears to be the first sharp large sieve inequality for a $GL_2$ family with varying levels.  The classical large sieve inequality can be interpreted as a $GL_1$ large sieve inequality, while Heath-Brown's celebrated quadratic large sieve can be viewed as an estimate for the sub-family of self-dual $GL_1$ forms.  
The $GL_2$ families of varying nebentypus do not seem to have strong orthogonality properties, as shown by Iwaniec and Li \cite{IwaniecLi}.

We also have an additive character variant on Theorem \ref{thm:mainthm}.
\begin{mytheo}
\label{thm:mainthmadditive}
Define a norm
\begin{equation*}
\Delta_{\text{add.}}(Q,N)
 = \max_{|{\bf \alpha}| = 1} 
\sum_{\substack{Q/2 < q \leq Q \\ (q,k) = 1}}
\thinspace
\sumstar_{t \shortmod{q}}
\Big|
\sum_{\substack{N/2 < ab \leq N \\ (a,b) = 1 \\ (ab,q) = 1}} \alpha_{a,b} e_q(t a \overline{b}) 
\Big|^2 dt.
\end{equation*}
Then
$\Delta_{\text{add.}}(Q,N) \ll (Q^2 + N)^{1+\varepsilon}$.
\end{mytheo}
Theorem \ref{thm:mainthmadditive} follows quickly from Theorem \ref{thm:mainthm}, by the method in \cite[Section 7.5]{IK}.  
We have omitted the $T$ and $k$ aspects solely to simplify the expressions; hybrid bounds analogous to \eqref{eq:mainthm} hold for the additive characters as well.

We may interpet Theorem \ref{thm:mainthm} as a large sieve inequality for \emph{rationals}, which we now explain.  
Let $v_p$ be the usual $p$-adic valuation.
For $q \geq 1$, let 
$\mathbb{Q}_{(q)} = \{ x \in \mathbb{Q} : v_p(x) \geq 0 \text{ for all } p \mid q \}$, which is a ring.  Indeed, with
the multiplicative set $S$ defined by
 $S = \{ n \in \mz : (n,q) = 1\}$, then $\mathbb{Q}_{(q)} = S^{-1} \mathbb{Z}$, the localization of $\mathbb{Z}$ by $S$.  
 There exists a natural reduction map $\redu_q : \mathbb{Q}_q \rightarrow \mz/q\mz$.  The reduction map may be restricted to $\mathbb{Q}_{(q)}^{\times} = \{ x \in \mathbb{Q} : v_p(x) = 0 \text{ for all } p \mid q \}$, which is a multiplicative subgroup of $\mathbb{Q}^{\times}$.  
If $\chi$ is a Dirichlet character modulo $q$, and $n \in \mathbb{Q}_{(q)}^{\times}$, then  define $\chi(n)$ by $\chi(\redu_q(n))$.  That is, if $n = a/b \in \mathbb{Q}_{(q)}^{\times}$, then $\chi(n) = \chi(a \overline{b})$.
For $n = a/b \in \mathbb{Q}^{\times}$ in lowest terms, define $\hit(n) = |ab|$, which is a cousin of a height function.
Note that $|\{ n \in \mathbb{Q}^{\times} : \hit(n) \leq X \}| =  X^{1+o(1)}$.
%$|\{ n \in \mathbb{Q}^{\times} : \hit(n) \leq X \}| \sim \frac{4}{\zeta(2)} X^{}$.
\begin{mytheo}
\label{thm:largesieverationalversion}
We have
\begin{equation}
\sum_{q \leq Q}
\thinspace
\sumstar_{\chi \shortmod{q}}
\Big|
\sum_{\substack{n \in \mathbb{Q}_{(q)}^{\times} \\ \hit(n) \leq N}}
\alpha_n \chi(n)
\Big|^2
\ll (Q^2 + N)^{1+\varepsilon}
\sum_{\substack{n \in \mathbb{Q}^{\times} \\ \hit(n) \leq N}}
|\alpha_n|^2. 
\end{equation}
\end{mytheo}
This is simply a restatement of Theorem \ref{thm:mainthm} in this notation, with $k=1$ and the omission of $T$.  These specializations are not necessary, and are only in place to de-clutter the statement.

From Theorem \ref{thm:largesieverationalversion} one can also easily derive results about rationals ordered by the more standard height function.  For $n=a/b \in \mathbb{Q}^{\times}$ in lowest terms, let $\mathrm{Ht}(n) = \max( |a|, |b|)$.  Note that $\hit(n) \leq \mathrm{Ht}(n)^2$, from which we immediately deduce:
\begin{mycoro}
\label{coro:largesieverationalversionStandardHeight}
We have
\begin{equation*}
\sum_{q \leq Q}
\thinspace
\sumstar_{\chi \shortmod{q}}
\Big|
\sum_{\substack{n \in \mathbb{Q}_{(q)}^{\times} \\ \mathrm{Ht}(n) \leq N}}
\alpha_n \chi(n)
\Big|^2
\ll (Q^2 + N^2)^{1+\varepsilon}
\sum_{\substack{n \in \mathbb{Q}^{\times} \\ \mathrm{Ht}(n) \leq N}}
|\alpha_n|^2. 
\end{equation*}
\end{mycoro}
This is sharp, since $|\{ n \in \mathbb{Q}^{\times} : \mathrm{Ht}(n) \leq X \}| = X^{2+o(1)}$.  
%$|\{ n \in \mathbb{Q}^{\times} : \mathrm{Ht}(n) \leq X \}| \sim \frac{4}{\zeta(2)} X^{2}$.  
Since Theorem \ref{thm:largesieverationalversion} easily implies Corollary \ref{coro:largesieverationalversionStandardHeight}, but not vice-versa, this supports our usage of $\hit$ in place of $\mathrm{Ht}$.

For $n \in \mathbb{Q}^{\times}$, one may define $\alpha_n = e(n)$, or $\alpha_{a/b} = e_b(\overline{a})$, etc.   These examples illustrating Theorem \ref{thm:largesieverationalversion} are somewhat similar to the quantities studied in \cite{DFIBilinearKloosterman}.

The proof of Theorem \ref{thm:mainthm} attacks the problem from both sides, via $\Delta$ and $\Delta^*$.  In this sense, the proof has new features not seen in previous large sieve inequality bounds.  Very briefly, the strategy of proof is as follows.  If $N \gg Q^2 kT$, then 
we study the dual norm $\Delta^*$ and apply the functional equation of Dirichlet $L$-functions.
The dual side is effective in this range of parameters because the functional equation will shorten the lengths of summation.  On the other hand, if $N \ll Q^2 kT$, then we more directly study the family average.
The main tool on this side is the divisor-switching method used by Conrey-Iwaniec-Soundararajan on the asymptotic large sieve \cite{ConreyIwaniecSound} (see also \cite[p.210]{Hooley}).
On both sides, we derive a recursive bound which relates the norm to itself, but with different (smaller) parameters.

When $N \approx Q^2 kT$, then both methods are essentially circular.  The key to breaking out of this deadlock is to use monotonicity, lengthening one of the sums. 
The use of the functional equation and monotonicity were both crucial tools in Heath-Brown's quadratic large sieve.  A major difference between our method and Heath-Brown's is that in the quadratic case, the norm was almost self-dual by quadratic reciprocity.  This property completely fails in our situation.  

We now discuss the two main workhorse results used to prove Theorem \ref{thm:mainthm}, both of which require defining some variants on $\Delta$.
Let
\begin{equation}
\label{eq:Delta''def}
 \Delta'(Q,k,T,N) = 
 \max_{\substack{X, R, U, C \in \mr_{\geq 1}, \ell \in \mz_{>0} \\  X R^2 \ell U \leq Q^2 k T \\ X \leq C }} 
  X
 \Delta\Big(R, \ell, U, \frac{N}{C} \Big).
\end{equation}
%and
%\begin{equation}
%\label{eq:Delta'defAlt}
%\Delta'(Q,k,T,N) = \max_{R, V, U \in \mathbb{Z}_{\geq 1}}
%\max_{1 \leq Y \leq RV^2} Y 
% \Delta(\frac{Q}{RVU}, kRU, T,  \frac{N}{Y}) .
%\end{equation}
Note trivially $\Delta(Q, k, T, N)
\leq \Delta'(Q,k,T,N)$, by taking $X=1$, $R=Q$, $\ell = k$, $U=T$, $C=1$.  Theorem \ref{thm:mainthm} will show these norms are essentially the same order of magnitude.  On a first pass, the reader is encouraged to think of $\Delta'(Q,k,T,N)$ as $\Delta(Q,k,T,N)$ itself.
Another notational convenience is to write
\begin{equation}
\label{eq:Deltabardef}
\overline{\Delta}(Q, k, T, N) = 
\max_{\substack{Q \leq R \leq Q (Q^2 kTN)^{\varepsilon} \\ T \leq U \leq T (Q^2 kTN)^{\varepsilon} \\ N \leq M \leq N (Q^2 kTN)^{\varepsilon}}} \Delta(R, k, U, M),
\end{equation}
and similarly for other norms, such as $\overline{\Delta'}$.
In practice, the choices of $\varepsilon$ will be either unimportant, or apparent from the context, and no confusion should arise from suppressing them on the left hand side of \eqref{eq:Deltabardef}.

\begin{mytheo}[Recursive functional equation]
\label{thm:recursivethmFE}
Suppose $N \gg Q^2 k T (QkT)^{-\varepsilon}$.  Then
\begin{equation}
\Delta(Q, k, T, N) \ll (QkTN)^{\varepsilon} \Big[N + \frac{N}{Q^2 kT}  \overline{\Delta'}\Big(Q, k, T, \frac{Q^4 k^2 T^2}{N}\Big) \Big].
\end{equation}
\end{mytheo}

We also derive a recursive bound on $\Delta$ by
the family average approach. 
\begin{mytheo}[Recursive family average]
\label{thm:familyavgthm}
Suppose
$Q^2 k T \gg  N (QkT)^{-\varepsilon}$.  Then
\begin{equation}
\Delta(Q, k, T, N) \ll 
(QkTN)^{\varepsilon} \Big[
Q^2 kT + 
\frac{Q^2 k T}{N} \overline{\Delta'}\Big(\frac{N}{kQT}, k, T, N\Big)
\Big]
.
\end{equation}
\end{mytheo}

The proofs of Theorems \ref{thm:recursivethmFE} and \ref{thm:familyavgthm}, appearing in Sections \ref{section:FE} and \ref{section:FamilyAvg}, respectively, are logically independent, and can be read in either order.  
Although very different in the fine details, the two proofs have important structural similarities.  
Because of the logical independence of these two sections, and due to the strong analogies, we have deliberately chosen to `refresh' notation when passing from Section \ref{section:FE} to Section \ref{section:FamilyAvg}.  Even more, we have structured the proofs in a similar way, and chosen notation to help draw the reader's attention to analogous quantities in the two proofs.

Our main interest in Theorem \ref{thm:mainthm} is with $k=T=1$.  However, the recursive nature of the proof and the appearance of the generalized norm $\Delta'$ in Theorems \ref{thm:recursivethmFE} and \ref{thm:familyavgthm} forces us to consider more general values of $k$ and $T$.

\subsection{Applications}
The classical large sieve has a wealth of important applications, and we consider some variants for the new rational large sieve (Theorem \ref{thm:mainthm}).  
The literature in analytic number theory on sieving problems for the rational numbers is relatively sparse.
The authors of \cite{EEHK, Zywina} give versions of Gallagher's larger sieve for rationals, and deduce some impressive algebraic applications.  More applications could be of great interest.

Consider the following sieving problem.  Let $\mathcal{N} = \{n \in \mathbb{Q}_{>0} : \hit(n) \leq N \}$.  Let $\mathcal{P}$ be a finite set of prime numbers.  For each $p \in \mathcal{P}$, let $\Omega_p \subset \mathbb{Z}/p\mathbb{Z}$.  Define the sifted set
\begin{equation*}
\mathcal{S}(\mathcal{N}, \mathcal{P}, \Omega)
=
\{ n \in \mathcal{N} : 
\text{for all } p \text{ with } v_p(n) = 0, \thinspace \redu_p(n) \not \in \Omega_p  
\}.
\end{equation*}
Note that if $v_p(n) \neq 0$, then $\redu_p(n)$ may not be defined.
Let $\omega(p) = |\Omega_p|$, and suppose that $\omega(p) < p$ for all $p \in \mathcal{P}$.  
Let $h(p) = \frac{\omega(p)}{p - \omega(p)}$ for $p \in \mathcal{P}$, and $h(p) = 0$ for $p \not \in \mathcal{P}$, and extend $h$ multiplicatively on the squarefree integers.  
Define $H = \sum_{q \leq Q} h(q)$.  
\begin{myprop}
With the above notation, we have
\begin{equation*}
|\mathcal{S}(\mathcal{N}, \mathcal{P}, \Omega)|
\ll \frac{(N+Q^2)^{1+\varepsilon}}{H}.
\end{equation*}
\end{myprop}
One can prove this following the method of \cite[Theorem 7.14]{IK}.  Alternatively, see \cite[Proposition 2.3]{Kowalski} for a proof in much greater generality.
For a nontrivial result, one needs $H \gg N^{\varepsilon}$, which is more restrictive than in the classical arithmetic large sieve.

A standard application of the classical large sieve is to let $\Omega_p$ consist of $\frac{p-1}{2}$ residue classes chosen arbitrarily, for all $p \leq Q$.  Then $H \gg Q$, and taking $Q = \sqrt{N}$ gives that $|\mathcal{S}(\mathcal{N}, \mathcal{P}, \Omega)| \ll N^{1/2+\varepsilon}$.  
%The particular case where $\Omega_p$ consists of the quadratic non-residues modulo $p$ is of particular interest in the classical large sieve, but this is not a novel illustration of our new large sieve for \emph{rationals}, since $(\frac{a \overline{b}}{p}) = (\frac{ab}{p})$, which means that this example could be handled directly with the classical large sieve.

We also present a Barban-Davenport-Halberstam type theorem.
Suppose that $\alpha_n$ is a sequence supported on $\mathbb{Q}_{>0}$, with $\hit(n) \leq X$.  We assume a weak Siegel-Walfisz type condition for the sequence, as follows.  Define
\begin{equation*}
 S(X, \chi) = \sum_{\hit(n) \leq X} \alpha_n \chi(n).
\end{equation*}
For $\chi = \chi' \chi_0$ with $\chi'$ of conductor $r > 1$, and $\chi_0$ trivial modulo $s$, we assume
\begin{equation}
\label{eq:SWcondition}
 |S(X, \chi)| \ll_{B, k} |\alpha| \frac{X^{1/2} \tau_k(s)}{(\log X)^{B}}, 
\end{equation}
for some $k$-fold divisor function $\tau_k$, and all $r \leq (\log X)^B$.

\begin{myprop}
\label{prop:BDH}
 Suppose that $\alpha$ satisfies the S-W condition \eqref{eq:SWcondition}, for any $B > 0$.  Then
 \begin{equation*}
  \sum_{q \leq Q} \thinspace \sumstar_{a \shortmod{q}}
  \Big| \sum_{\substack{\hit(n) \leq X \\ n \equiv a \shortmod{q}}} \alpha_n
  - \frac{1}{\varphi(q)} \sum_{\substack{\hit(n) \leq X \\ (n,q) = 1}} \alpha_n
  \Big|^2 \ll \frac{X |\alpha|^2}{(\log X)^A},
 \end{equation*}
 for any $A > 0$,
provided $Q \ll X^{1-\varepsilon}$.  
\end{myprop}
We prove Proposition \ref{prop:BDH} in Section \ref{section:BDH}.

\subsection{Proof sketches}
Here we present some overly-simplified outlines of the proofs.  In this section we freely drop factors of size $(Q^2 k T N)^{\varepsilon}$, as if they were $1$.
%\subsubsection{Theorem \ref{thm:mainthm}}
%The deduction of Theorem \ref{thm:mainthm} from Theorems \ref{thm:recursivethmFE} and \ref{thm:familyavgthm} is similar to Heath-Brown's quadratic large sieve \cite{HB}.  
\subsubsection{Theorem \ref{thm:recursivethmFE}}
\label{section:recursivethmFEsketch}
For simplicity, we omit the $t$-aspect, and write $\Delta(Q,k,N)$ for the norm.
For a bump function $w$ supported on $[1/2, 2]$, 
consider
\begin{equation*}
 S = \sum_{(a,b) = 1} w\Big(\frac{ab}{N}\Big) | T(a,b) |^2,
 \quad 
 \text{where}
 \quad
 T(a,b) = \sum_{q,\chi, \theta} \beta_{\chi, \theta} \chi \theta(a \overline{b}).
\end{equation*}
The condition $(a,b) = 1$ is necessary but difficult to use.  
In comparison to the quadratic large sieve, this condition is analogous to the restriction to fundamental discriminants.  Inspired by this similarity, and following \cite{HB}, let $1 \leq Y <N/10$ to be chosen later, and note $S \leq S_{>Y}$ where
\begin{equation*}
 S_{>Y} = \sum_{\frac{ab}{(a,b)^2} > Y} w\Big(\frac{ab}{N}\Big) |T(a,b)|^2.
\end{equation*}
We then write $S_{>Y} = S_{\infty} - S_{\leq Y}$, where $S_{\leq Y}$ has $\frac{ab}{(a,b)^2} \leq Y$, and $S_{\infty}$ has $a$ and $b$ unconstrained.  These two sums are treated in completely different ways.  For $S_{\leq Y}$, let $g = (a,b)$ and change variables $a \rightarrow ga$ and $b \rightarrow gb$.  Ignoring coprimality issues, then $T(ga,gb) \approx T(a,b)$, and so
\begin{equation*}
 S_{\leq Y} \approx \sum_{\substack{ab \leq Y \\ (a,b) = 1}} \sum_g w\Big(\frac{g^2 ab}{N}\Big)
 |T(a,b)|^2
 =  \int_{(2)} \widetilde{w}(s) \zeta(2s) \sum_{\substack{ab \leq Y \\ (a,b) = 1}} 
 \Big(\frac{N}{ab}\Big)^s
 |T(a,b)|^2 \frac{ds}{2 \pi i}.
\end{equation*}
Next shift contours to the line $\varepsilon$, passing a pole at $s=1/2$.  The contribution to $S_{\leq Y}$ from the new contour is essentially $\ll N^{\varepsilon} \Delta(Q,k, Y) |\beta|^2$.  The pole at $s=1/2$ gives
\begin{equation}
\label{eq:sketchpolartermFE}
 \tfrac12 \widetilde{w}(1/2) 
 \sum_{\substack{ab \leq Y \\ (a,b) = 1}} \Big(\frac{N}{ab}\Big)^{1/2} |T(a,b)|^2.
\end{equation}
This term \eqref{eq:sketchpolartermFE} is \emph{not} satisfactorily bounded on its own.  Indeed, even if we accept Theorem \ref{thm:mainthm}, then by breaking up into dyadic segments $M/2 < ab \leq M$, with $1 \leq M \leq Y$, we can at best bound \eqref{eq:sketchpolartermFE} by
\begin{equation*}
 \max_{1 \leq M \leq Y} \Big(\frac{N}{M}\Big)^{1/2} (Q^2k + M)^{} |\beta|^2
\ll (Q^2 k \sqrt{N} + N^{1/2} Y^{1/2} )|\beta|^2
.
\end{equation*}
The former term of size $Q^2 k \sqrt{N}$ is the culprit, and matches with \eqref{eq:DeltaTrivialBound}.  Luckily, and crucially, the term \eqref{eq:sketchpolartermFE} will partially cancel with another term from $S_{\infty}$.  This cancellation property also appeared in \cite{HB}.

Next consider $S_{\infty}$.  Opening $|T(a,b)|$ and applying the Mellin inversion formula gives
\begin{equation*}
 S_{\infty} = 
  \sum_{q_1, q_2, \chi_1, \chi_2, \theta_1, \theta_2} \beta_1 \overline{\beta_2}
 \int_{(2)} \widetilde{w}(s) N^s 
 L(s , \Phi) L(s , \overline{\Phi}) \frac{ds}{2 \pi i},
\end{equation*}
where $\Phi = \chi_1 \overline{\chi_2} \theta_1 \overline{\theta_2}$.  Unfortunately, $\Phi$ may not be primitive, and this complicates the application of the functional equation.  For this sketch, we consider the two extremes, where either $\Phi$ is primitive of conductor $q_1 q_2 k$, or where $\Phi$ is trivial.  The trivial case is easy to control, since this means $\chi_1 = \chi_2$ (whence $q_1 = q_2$) and $\theta_1 = \theta_2$.  This gives rise to a diagonal term of acceptable size $O(N |\beta|^2)$.  For the primitive characters, we shift contours to the line $-1$, change variables $s \rightarrow 1-s$, and apply the functional equation.  This gives (roughly)
\begin{equation*}
  \sum_{q_1, q_2, \chi_1, \chi_2, \theta_1, \theta_2} \beta_1 \overline{\beta_2}
 \int_{(2)} \widetilde{w}(1-s)
 \frac{(q_1 q_2 k)^{2s-1}}{N^{s-1}} \frac{\gamma(s)}{\gamma(1-s)}
 L(s, \Phi) L(s, \overline{\Phi}) \frac{ds}{2 \pi i},
\end{equation*}
where $\gamma(s)$ is the product of gamma factors in the completed $L$-function of $L(s,\Phi) L(s, \overline{\Phi})$.  
Next re-open the Dirichlet series and rearrange, giving
\begin{equation*}
 \sum_{a,b}
 \int_{(2)} \widetilde{w}(1-s) \frac{\gamma(s)}{\gamma(1-s)}
 \sum_{q_1, q_2, \chi_1, \chi_2, \theta_1, \theta_2} \beta_1 \overline{\beta_2}
 \frac{(q_1 q_2 k)^{2s-1}}{(ab)^s N^{s-1}} 
\chi_1 \overline{\chi_2} \theta_1 \overline{\theta_2} (a \overline{b})
 \frac{ds}{2 \pi i}.
\end{equation*}
Letting $g = (a,b)$, replacing $a$ by $ga$ and $b$ by $gb$, and summing over $g$, we obtain
\begin{equation*}
 \sum_{\substack{ab \leq \frac{Q^4 k^2}{N} \\ (a,b) = 1}}
 \int_{(2)} \widetilde{w}(1-s) \frac{\gamma(s)}{\gamma(1-s)} \zeta(2s)
 \sum_{q_1, q_2, \chi_1, \chi_2, \theta_1, \theta_2} \beta_1 \overline{\beta_2}
 \frac{(q_1 q_2 k)^{2s-1}}{(ab)^s N^{s-1}} 
\chi_1 \overline{\chi_2} \theta_1 \overline{\theta_2} (a \overline{b})
 \frac{ds}{2 \pi i},
\end{equation*}
as the sum can be truncated at $ab \leq \frac{Q^4 k^2}{N}$ (by shifting the contour far to the right).
Next we shift contours back to the line $\varepsilon$, crossing a pole at $s=1/2$.  This polar term has a nice simplification, and takes the same form as \eqref{eq:sketchpolartermFE}, but with $ab$ truncated at $\frac{Q^4 k^2}{N}$ instead of $Y$.  Taking $Y = \frac{Q^4 k^2}{N}$ then causes these two polar terms to cancel!  The contribution on the line $\varepsilon$ essentially becomes bounded by
%\begin{equation*}
 $\frac{N}{Q^2 k} \Delta(Q, k,  \frac{Q^4 k^2}{N})$,
%\end{equation*}
in line with Theorem \ref{thm:recursivethmFE}.

\subsubsection{Theorem \ref{thm:familyavgthm}}
For simplicity take $k=1$ and omit $t$, and write $\Delta(Q,N)$ for the norm.  For a bump function $w$, let
\begin{equation*}
 S =  \sum_q w(q/Q) \sumstar_{\chi \shortmod{q}} |T(\chi)|^2,
 \qquad
 T(\chi) = 
 \sum_{\substack{N/2 < ab \leq N \\ (a,b) = 1}} \alpha_{a,b} \chi(a \overline{b}).
\end{equation*}
The condition that $\chi$ is primitive is necessary but difficult to use.  In analogy with the proof of Theorem \ref{thm:recursivethmFE}, let $Y < Q/10$, and define
\begin{equation*}
 S_{>Y} =  \sum_{q} w(q/Q) \sumstar_{\substack{\chi \shortmod{q} \\ \cond(\chi) > Y}} |T(\chi)|^2.
\end{equation*}
Then $S \leq S_{>Y}$, by positivity.  Again, write $S = S_{\infty} - S_{\leq Y}$ where 
 $S_{\leq Y}$ has characters modulo $q$ with $\cond(\chi) \leq Y$ and
$S_{\infty}$ has $\chi$ ranging over all characters of modulus $q$.  

For $S_{\leq Y}$, replace $q$ by $q q_0$ and $\chi$ by $\chi \chi_0$ where (the new) $\chi$ has conductor $q$, and $\chi_0$ is trivial.  Ignoring coprimality, we have $T(\chi \chi_0, t) \approx T(\chi, t)$.  Applying Mellin inversion, and summing over $q_0$ to form a zeta function, we obtain
\begin{equation*}
 S_{\leq Y} \approx 
 \sum_{q \leq Y} 
\int_{(2)} \widetilde{w}(s) \Big(\frac{Q}{q}\Big)^s \zeta(s) 
 \sumstar_{\substack{\chi \shortmod{q}}} |T(\chi)|^2
 \frac{ds}{2 \pi i}.
\end{equation*}
We shift contours to the line $\varepsilon$, passing a pole at $s=1$ only.  This polar term takes the form
\begin{equation}
\label{eq:sketchpolartermFamilyAvg}
 Q \widetilde{w}(1)  \sum_{q \leq Y} 
 q^{-1}
 \sumstar_{\substack{\chi \shortmod{q}}} |T(\chi)|^2.
\end{equation}
On the new line $\varepsilon$, we essentially obtain an expression of size $\Delta(Y,  N) |\beta|^2$.
This polar term is the analog of \eqref{eq:sketchpolartermFE}, and as before, it is not satisfactorily bounded on its own.  Indeed, Theorem \ref{thm:mainthm} would imply that at best it is bounded by 
$$
Q \max_{R \leq Y} R^{-1} (R^2 + N) |\alpha|^2 = (QY + QN) | \alpha|^2. 
$$
Here the term $QN$ is the culprit, and as before, we will cancel this polar term with one arising within $S_{\infty}$.

Now consider $S_{\infty}$.  Opening the square and applying orthogonality of characters gives
\begin{equation*}
S_{\infty} \approx Q \sum_{q} w_1(q/Q) \sum_{\substack{(a_1, b_1) = (a_2, b_2) = 1 \\ a_1 b_2 \equiv a_2 b_1 \shortmod{q}}}
\alpha_{a_1, b_1} \overline{\alpha_{a_2, b_2}},
\end{equation*}
where $w_1(x) = x w(x)$.  The range of possible values of $\gcd(a_1 b_2, a_2 b_1)$ causes some arithmetical difficulties.  For this sketch, we consider the two extreme cases, where either they are coprime, or where $a_1 b_2 = a_2 b_1$, which we call the diagonal case.  Since $(a_1, b_1) = (a_2, b_2) = 1$, the diagonal reduces to $a_1 = a_2$ and $b_1 = b_2$, giving a term of size $O(Q^2  |\alpha|^2)$, which is acceptable.

We now focus on the case $(a_1 b_2, a_2 b_1) = 1$.  Write $a_1 b_2 = a_2 b_1 + qr$, which we now interpret as $a_1 b_2 \equiv a_2 b_1 \pmod{r}$, with $q = \frac{a_1 b_2 - a_2 b_1}{r}$.  
Note typically $r \ll N/Q$, so this reduces the modulus when $Q^2 \gg N$.
This leads to
\begin{equation*}
S_{\infty} \approx Q \sum_{r }  \sum_{\substack{(a_1, b_1) =  (a_2, b_2) = 1 \\ a_1 b_2 \equiv a_2 b_1 \shortmod{r}}}
w_1\Big(\frac{a_1 b_2 - a_2 b_1}{Qr}\Big)
\alpha_{a_1, b_1} \overline{\alpha_{a_2, b_2}}.
\end{equation*}
Next we detect the congruence with
characters modulo $r$, as in \cite{ConreyIwaniecSound}, giving
\begin{equation*}
S_{\infty} \approx Q \sum_{r  } \sum_{\chi \shortmod{r}} r^{-1}  \sum_{\substack{(a_1, b_1) =   (a_2, b_2) = 1 }}
w_1\Big(\frac{a_1 b_2 - a_2 b_1}{Qr}\Big)
\alpha_{a_1, b_1} \overline{\alpha_{a_2, b_2}}
\chi(a_1 b_2 \overline{a_2 b_1})
.
\end{equation*}
Since the characters are not primitive, replace $\chi$ by $\chi \chi_0$ and $r$ by $r r_0$ where the new $\chi$ has conductor $r$, and $\chi_0$ is trivial modulo $r_0$.  Applying Mellin inversion, and evaluating the $r_0$-sum in terms of a zeta function, we obtain that $S_{\infty}$ is roughly
\begin{equation*}
Q \int_{(1)} \widetilde{w_1}(-s)
\sum_{r \leq N/Q} \thinspace \sumstar_{\chi \shortmod{r}} r^{-1}  \sum_{\substack{(a_1, b_1) = 1 \\ (a_2, b_2) = 1 }}
\Big(\frac{a_1 b_2 - a_2 b_1}{Qr}\Big)^s \zeta(s+1)
\alpha_{a_1, b_1} \overline{\alpha_{a_2, b_2}}
\chi(a_1 b_2 \overline{a_2 b_1})
\frac{ds}{2 \pi i}
.
\end{equation*}
Next we shift contours to the line $-1+\varepsilon$, passing a pole at $s=0$ only.  Note that $\widetilde{w_1}(0) = \widetilde{w}(1)$.    This polar term nicely simplifies, and takes the same form as \eqref{eq:sketchpolartermFamilyAvg}, but with $r$ truncated at $N/Q$ instead of $Y$.  Taking $Y = N/Q$ causes the two polar terms to cancel.  
Next consider the integral along the line $-1+\varepsilon$.  
The variables $a_i, b_i$ are not separated, but one might hope that this is only a technical issue solvable with integral transform techniques (indeed, see Lemma \ref{lemma:separationofvariablesDeterminant}).  We might then expect that the contribution from the new line of integration to be bounded by $\frac{Q^2}{N} \Delta(N/Q,  N) |\alpha|^2$, which is consistent with Theorem \ref{thm:familyavgthm}.

The wealth of extra parameters in the definition of $\Delta'$ in \eqref{eq:Delta''def} are there to account for the overlooked conditions (both arithmetical and archimedean).

\subsubsection{Reflections}
The similarities between the proofs are remarkable, even if the fine details are different.  We also observe that the divisor-switching method used in the proof of Theorem \ref{thm:familyavgthm} is analogous to the functional equation of the Dirichlet $L$-functions used for Theorem \ref{thm:recursivethmFE}.  At the cost of some exaggeration, one might call the divisor switch itself a functional equation.  In support of this, consider the family of functions $\tau_{s}(n) = \sum_{ab=n} (a/b)^s$, which does indeed satisfy the functional equation $\tau_{-s}(n) = \tau_{s}(n)$, by the divisor switch.  Moreover, these coefficients $\tau_{s}(n)$ appear as Fourier coefficients of the level $1$ Eisenstein series, and the functional equation of the Eisenstein series is entwined with the functional equation of its Fourier coefficients.

\subsubsection{Theorem \ref{thm:mainthm}}
Theorem \ref{thm:mainthm} is deduced from Theorems \ref{thm:recursivethmFE} and \ref{thm:familyavgthm} in Section \ref{section:mainthmdeduction}.  The proof uses that the norm $\Delta$ is monotonic, and applies the two self-referential theorems in a recursive manner.  
In retrospect, some of these ideas have similarities to elements used in \cite{BombieriIwaniecZeta, BombieriIwaniecExponentialSums}.

\subsection{Notation and conventions}
Let $\Gamma_{\mr}(s) = \pi^{-s/2} \Gamma(s/2)$.  If $\chi$ is a Dirichlet character and $a/b \in \mathbb{Q}$ in lowest terms, we may interchangeably write
\begin{equation}
 \chi(a) \overline{\chi}(b) = \chi(a \overline{b}) = \chi(a/b).
\end{equation}
We
use the notation $A \lesssim B$ as a synonym for
\begin{equation}
\label{eq:AlesssimBdef}
 A \leq C(\varepsilon) (Q^2 k TN)^{\varepsilon} B.
\end{equation}
%The context should make it clear if \eqref{eq:AlesssimBdef} is meant to hold for all $\varepsilon > 0$, or if there is a specific $\varepsilon$ and $C(\varepsilon)$ in place.

\subsection{Acknowledgments}
I thank Henryk Iwaniec and Emmanuel Kowalski for valuable comments.
I am also grateful to the referee for a careful reading which uncovered several inaccuracies.

\section{Deduction of Theorem \ref{thm:mainthm}}
\label{section:mainthmdeduction}
In this section, we use Theorems \ref{thm:recursivethmFE} and \ref{thm:familyavgthm} to prove Theorem \ref{thm:mainthm}.

\subsection{Monotonicity}
As in the quadratic large sieve \cite{HB}, it is vital that the norm $\Delta(Q, k, T, N)$ is essentially monotonic in the $N$- and $Q$-components.
%  In fact, the norm is essentially monotonic in
%each component.  
The proofs differ a bit depending on the case, but the overall theme is similar, and based on an idea of Forti and Viola \cite{FortiViola}.

\begin{mylemma}
\label{lemma:mononicityNaspect}
 Suppose $P \gg \log{QN}$ with a large (but absolute) implied constant.  Then there exists a prime $p \in [P, 2P]$ so that
 \begin{equation*}
  \Delta(Q,k,T, N) \leq 8 \Delta(Q,k,T, Np).
 \end{equation*}
\end{mylemma}
\begin{proof}
Since $k$ and $T$ are frozen, we suppress them from the discussion, writing $\Delta(Q,N)$ in place of $\Delta(Q,k,T,N)$.  
Let $\gamma_{a,b}$ be complex numbers supported on $N/2 < ab \leq N$, and $(a,b) = 1$.  Let $P \geq 1$ be a parameter to be chosen, and let 
$P^*$ denote the number of primes $p \in [P, 2P]$.  The prime number theorem implies $P^* \sim \frac{P}{\log{P}}$.
Now we have
\begin{align*}
\sum_{q, \chi} 
\Big|
\sum_{\substack{(a,b) = 1 \\ N/2 < ab \leq N}} \gamma_{a,b} \chi(a) \overline{\chi}(b) \Big|^2
=
\sum_{q, \chi} \frac{1}{P^*} \sum_{P \leq p \leq 2P}
\Big|
\sum_{\substack{(a,b) = 1 \\ N/2 < ab \leq N}} \gamma_{a,b} \chi(a) \overline{\chi}(b) \Big|^2
\\
=
\sum_{q, \chi} \frac{1}{P^*} 
\Big(\sum_{\substack{P \leq p \leq 2P \\ p \nmid q}}
+
\sum_{\substack{P \leq p \leq 2P \\ p | q}}
\Big)
\Big|
\sum_{\substack{(a,b) = 1 \\ N/2 < ab \leq N}} \gamma_{a,b} \chi(a) \overline{\chi}(b) \Big|^2.
%\label{eq:monotonicityEquation}
\end{align*}
For the terms with $p|q$, we simply use $\frac{1}{P^*} \sum_{\substack{P \leq p \leq 2P \\ p | q}} 1 \leq \frac{\log{Q}}{P^* \log{P}}$.  Taking $P \gg \log{Q}$ large enough so that $P^* \log{P} \geq 2 \log{Q}$, and rearranging, we obtain
\begin{equation*}
 \Delta(Q,N) \leq 
\max_{\gamma \neq 0} 
 \frac{2}{|\gamma|^2} 
 \sum_{q, \chi} \frac{1}{P^*} 
\sum_{\substack{P \leq p \leq 2P \\ p \nmid q}}
\Big|
\sum_{\substack{(a,b) = 1 \\ N/2 < ab \leq N}} \gamma_{a,b} \chi(a) \overline{\chi}(b) \Big|^2.
\end{equation*}
Next we separate the values of $a$ and $b$ to make two sub-sums corresponding to $(p,ab) = 1$ and $p|ab$.  This gives
\begin{equation*}
\Delta(Q,N) \leq  
\max_{\gamma \neq 0} 
 \frac{4}{|\gamma|^2} 
 \sum_{q, \chi} \frac{1}{P^*} 
\sum_{\substack{P \leq p \leq 2P \\ p \nmid q}}
\Big(
\Big|
\sum_{\substack{(ab,p) = 1}} \Big|^2
+
\Big|\sum_{\substack{ p|ab}} \Big|^2
\Big).
\end{equation*}
We bound the terms with $p|ab$ similarly to the treatment of $p|q$, giving
\begin{multline*}
\max_{\gamma \neq 0} 
 \frac{4}{|\gamma|^2} 
 \sum_{q, \chi} \frac{1}{P^*} 
\sum_{\substack{P \leq p \leq 2P \\ p \nmid q}} 
\Big|\sum_{p|ab} \Big|^2 
\leq 
\max_{\gamma \neq 0}
\frac{4}{|\gamma|^2 P^*} 
\sum_{\substack{P \leq p \leq 2P}} 
\Delta(Q,N) \sum_{p|ab} |\gamma_{a,b}|^2
\\
\leq \frac{4 \log{N}}{P^* \log P} \Delta(Q,N).
\end{multline*}
We choose $P \gg \log{N}$ large enough so that $\frac{4 \log{N}}{P^* \log{P}} \leq \frac12$, whence
\begin{equation*}
\Delta(Q,N) \leq 
\max_{\gamma \neq 0} 
 \frac{8}{|\gamma|^2} 
 \sum_{q, \chi} \frac{1}{P^*} 
\sum_{\substack{P \leq p \leq 2P \\ p \nmid q}}
\Big|
\sum_{\substack{(a,b) = 1 \\ (ab,p) = 1}} \gamma_{a,b} \chi(a \overline{b}) \Big|^2.
\end{equation*}

Now we freely multiply by $|\chi(p)|^2$, which has absolute value $1$ since $p \nmid q$.
In addition, we change variables $A = ap$, let $\delta_{A,b} = \gamma_{A/p, b}$, make note that $Np/2 < Ab \leq Np$, $| \delta | = | \gamma|$, and $(A,b) = 1$.  Thus
\begin{equation*}
\Delta(Q,N) \leq \frac{8}{P^*} 
\sum_{\substack{P \leq p \leq 2P }}
\Delta(Q, Np) 
\leq 8 \max_{P \leq p \leq 2P} \Delta(Q, Np). \qedhere
\end{equation*}
\end{proof}

\begin{mylemma}
\label{lemma:mononicityQaspect}
 Suppose $P \gg \log{NQ}$ with a large (but absolute) implied constant.  Then there exists a prime $p \in [P, 2P]$ so that
 \begin{equation*}
  \Delta(Q,k,T, N) \leq 8 \Delta(Qp,k,T, N).
 \end{equation*}
\end{mylemma}
\begin{proof}
Since $k$ and $T$ are frozen, we suppress them in the notation.  
%For this proof, it is convenient to work on the dual side, expressing $\Delta^*$ as the norm of
%\begin{equation}
% \sum_{\substack{(a,b) = 1 \\ N/2 < ab \leq N}} 
% \Big| \sum_{q, \chi} \beta_{\chi} \chi(a) \overline{\chi}(b) \Big|^2.
%\end{equation}
Let $P \geq 10$ to be chosen, and let $P^{**} = \sum_{P \leq p \leq 2P} \sumstar_{\psi \mymod{p}} 1$, so $P^{**} \asymp \frac{P^2}{\log{P}}$.  We have
\begin{align*}
\sum_{\substack{(a,b) = 1 \\ N/2 < ab \leq N}} 
 \Big| \sum_{q, \chi} \beta_{\chi} \chi(a) \overline{\chi}(b) \Big|^2
 = 
 \sum_{\substack{(a,b) = 1 \\ N/2 < ab \leq N}} 
 \frac{1}{P^{**}} 
 \sum_{P \leq p \leq 2P} \thinspace \sumstar_{\psi \shortmod{p}}
 \Big| \sum_{q, \chi} \beta_{\chi} \chi(a) \overline{\chi}(b) \Big|^2
 \\
  = 
 \sum_{\substack{(a,b) = 1 \\ N/2 < ab \leq N}} 
 \frac{1}{P^{**}} 
 \Big(\sum_{\substack{p, \psi \\ (p,ab) = 1}}
 + 
 \sum_{\substack{p, \psi \\ p|ab}}
 \Big)
 \Big| \sum_{q, \chi} \beta_{\chi} \chi(a) \overline{\chi}(b) \Big|^2.
\end{align*}
For the terms with $p|ab$, we simply use 
$\frac{1}{P^{**}} \sum_{\substack{p, \psi \\ p | ab}} 1 \leq \frac{2P \log{N}}{P^{**} \log{P}}$, and choose $P \gg \log{N}$ large enough so that 
$\frac{2P \log{N}}{P^{**} \log{P}} \leq \frac12$.
For the terms with $p \nmid ab$, we freely multiply by $|\psi(a) \overline{\psi}(b)|^2$, which is $1$ for such primes.  This gives
\begin{equation*}
 \Delta(Q,N) \leq \max_{\beta \neq 0} \frac{2}{|\beta|^2} 
 \sum_{\substack{(a,b) = 1 \\ N/2 < ab \leq  N}} 
 \frac{1}{P^{**}} 
\sum_{\substack{p, \psi}}  
 \Big| \sum_{q, \chi} \beta_{\chi} \chi \psi(a) \overline{\chi \psi}(b) \Big|^2.
\end{equation*}
Next we separate the values of $q$ to make two sub-sums corresponding to $(p,q) = 1$ and $p|q$.  This gives
\begin{equation*}
 \Delta(Q,N) \leq 
 \max_{\beta \neq 0} \frac{4}{|\beta|^2} 
 \sum_{\substack{(a,b) = 1 \\ N/2 < ab \leq  N}} 
 \frac{1}{P^{**}} 
\sum_{\substack{p, \psi}}  
\Big(
 \Big| \sum_{\substack{q, \chi \\ (q,p) = 1}}  \Big|^2
 + 
 \Big| \sum_{\substack{q, \chi \\ p|q}}  \Big|^2
\Big).
 \end{equation*}
We upper bound the terms with $p|q$, giving
\begin{equation*}
\sum_{\substack{(a,b) = 1 \\ N/2 < ab \leq  N}} 
 \frac{4}{P^{**}} 
\sum_{\substack{p, \psi}}  
\Big| \sum_{\substack{q, \chi \\ p|q}}  \Big|^2
\leq
 \frac{4}{P^{**}} 
\sum_{\substack{p, \psi}}  \Delta(Q, N) \sum_{\substack{ q, \chi \\ p|q}} |\beta_{\chi}|^2
\leq 
\frac{4P \log{Q} }{P^{**} \log{P}}
\Delta(Q,N) |\beta|^2.
\end{equation*}
We choose $P \gg \log{Q}$ large enough so that $\frac{4 P \log{Q}}{P^{**} \log{P}} \leq \frac12$, whence
\begin{equation*}
 \Delta(Q,N) \leq 
  \max_{\beta \neq 0} \frac{8}{|\beta|^2} 
 \sum_{\substack{(a,b) = 1 \\ N/2 < ab \leq  N}} 
 \frac{1}{P^{**}} 
\sum_{\substack{p, \psi}}  
\Big| \sum_{\substack{q, \chi \\ (q,p) = 1}} \beta_{\chi} \chi \psi(a \overline{b}) \Big|^2.
\end{equation*}
Now $\chi \psi$ is a character of conductor $pq$, with $pQ/2 \leq pq \leq pQ$, so we obtain
\begin{equation*}
 \Delta(Q,N) \leq  \frac{8}{P^{**}} 
\sum_{\substack{p, \psi}}  \Delta(pQ, N) \leq 8 \max_{P \leq p \leq 2P} \Delta(pQ,N). \qedhere
\end{equation*}
\end{proof}
Remark.  The norm $\Delta$ is also monotonic in the $k$ and $T$-aspects, but this property is not needed in this work, so we do not give proofs.

\subsection{Relations between norms}
To simplify the recursive steps in the proof of Theorem \ref{thm:mainthm}, it is convenient to have the following relations.
\begin{mylemma}
\label{lemma:normtransferNExponent}
Suppose that there exists $e>1$ such that 
%\begin{equation}
$$\Delta(Q,k,T,N) \lesssim Q^2 k T + N^e,$$
%\end{equation}
for all $Q$, $k$, $T$, $N$.  Then we have for all $Q$, $k$, $T$, $N$ that
%\begin{equation}
$$\overline{\Delta'}(Q,k,T,N) \lesssim Q^2 k T + N^e.$$
%\end{equation}
\end{mylemma}
\begin{mylemma}
\label{lemma:normtransferFamilyExponent}
Suppose that there exists $e>1$ such that 
$$\Delta(Q,k,T,N) \lesssim (Q^2 k T)^e + N,$$
for all $Q$, $k$, $T$, $N$.  Then we have for all $Q$, $k$, $T$, $N$ that
$$\overline{\Delta'}(Q,k,T,N) \lesssim (Q^2 k T)^e + N.$$
\end{mylemma}
\begin{proof}
The proofs of both lemmas follow from the definitions \eqref{eq:Delta''def} and \eqref{eq:Deltabardef}.
\end{proof}

\subsection{The recursions}
\begin{myprop}
\label{prop:recursionNToFamilyexponents}
Suppose that there exists $e > 1$ such that
\begin{equation}
\label{eq:DeltaBoundAssumptionEndgame2}
\Delta(Q,k,T,N) \lesssim Q^2 k T + N^e,
\end{equation}
for all $Q,k,T,N$.  Then with $e' = 2 - \frac{1}{e}$, we have for all $Q, k, T, N$ that
\begin{equation*}
\Delta(Q,k,T,N) \lesssim (Q^2 k T)^{e'} + N.
\end{equation*}
\end{myprop}
\begin{proof}
 Let $F = Q^2 k T$, which is the size of the family.  By monotonicity (Lemma \ref{lemma:mononicityNaspect}), we have $\Delta(Q,k,T,N) \ll \Delta(Q,k,T, N_1)$ for $N_1 \gg N \log(FN)$.  Let $N_1 \asymp N\log{N} + F^{\alpha}$ for some $\alpha > 1$, so that $F \ll N_1$.  By Theorem \ref{thm:recursivethmFE}, 
\begin{equation*}
 \Delta(Q,k,T,N) \ll \Delta(Q,k,T,N_1) \lesssim 
 N_1 + \frac{N_1}{F} \overline{\Delta'}\Big(Q, k, T, \frac{F^2}{N_1}\Big).
\end{equation*}
By Lemma \ref{lemma:normtransferNExponent}, we can use the assumption \eqref{eq:DeltaBoundAssumptionEndgame2} to obtain
\begin{equation*}
 \Delta(Q,k,T,N) \lesssim
 N_1 + \frac{N_1}{F} \Big(F + \Big(\frac{F^2}{N_1}\Big)^e \Big)
 \ll N_1 + \frac{F^{2e-1}}{N_1^{e-1}}
 \lesssim N + F^{\alpha} + F^{2e-1 - \alpha(e-1)}
 .
\end{equation*}
We choose $\alpha$ optimally so that $\alpha = 2e -1 - \alpha(e-1)$, which simplifies as $\alpha = 2 - \frac{1}{e}$.  Since $e > 1$ by assumption, this means $\alpha > 1$, and completes the proof.
\end{proof}

We also have a complementary version:

\begin{myprop}
\label{prop:recursionFamilyToNexponents}
Suppose that there exists $e > 1$ such that
\begin{equation}
\label{eq:DeltaBoundAssumptionEndgame1}
\Delta(Q,k,T,N) \lesssim (Q^2 k T)^e + N,
\end{equation}
for all $Q,k,T,N$.  Then with $e' = 2 - \frac{1}{e}$ we have for all $Q, k, T, N$
\begin{equation*}
\Delta(Q,k,T,N) \lesssim Q^2 k T + N^{e'}.
\end{equation*}
\end{myprop}
\begin{proof}
Let $F = Q^2 k T$.  
By monotonicity (Lemma \ref{lemma:mononicityQaspect}), we have $\Delta(Q,k,T,N) \ll \Delta(Q_1,k,T, N)$ for $Q_1 \gg Q\log(FN)$.  We take $F_1 : = Q_1^2 k T \asymp Q^2 k T \log^2(FN) + N^{\alpha}$ for some $\alpha > 1$, so that $N \ll Q_1^2 k T$.  By Theorem \ref{thm:familyavgthm}, we have
\begin{equation*}
 \Delta(Q,k,T,N) \ll \Delta(Q_1,k,T,N) \lesssim 
 F_1 +  \frac{F_1}{N}  \overline{\Delta'}\Big(\frac{N}{k Q_1 T}, k, T, N \Big).
\end{equation*}
By Lemma \ref{lemma:normtransferFamilyExponent}, we can use the assumption \eqref{eq:DeltaBoundAssumptionEndgame1} to obtain
\begin{equation*}
 \Delta(Q,k,T,N) \lesssim
 F_1 + \frac{F_1}{N} \Big(\Big(\frac{N^2  }{F_1}\Big)^e  + N\Big)
 \ll F_1 + \frac{N^{2e-1}}{F_1^{e-1}}
 \lesssim F + N^{\alpha} + N^{2e-1 - \alpha(e-1)}
 .
\end{equation*}
Choosing $\alpha = 2 - \frac{1}{e}$ completes the proof.
\end{proof}

\subsection{Proof of Theorem \ref{thm:mainthm}}
Using the trivial bound \eqref{eq:DeltaTrivialBound}, we have
\begin{equation*}
 \Delta(Q,k,T,N) \lesssim Q^2 k T \sqrt{N} + N \leq (\sqrt{N} + Q^2 k T)^2 \ll N + (Q^2 k T)^2,
\end{equation*}
which is \eqref{eq:DeltaBoundAssumptionEndgame1} with exponent $e= e_0 = 2$.  Applying Proposition \ref{prop:recursionFamilyToNexponents} gives \eqref{eq:DeltaBoundAssumptionEndgame2} with $e_1 = 2 - \frac{1}{e_0} = 3/2$.  Continuing this process, we obtain a sequence of exponents $e_i$, with $e_{i+1} = 2 - \frac{1}{e_i}$, for which either \eqref{eq:DeltaBoundAssumptionEndgame1} or \eqref{eq:DeltaBoundAssumptionEndgame2} holds (in an alternating fashion).  It is easy to check that the $e_i$ are monotonically decreasing, with limit $1$, whence Theorem \ref{thm:mainthm} holds.

\section{Proof of Proposition \ref{prop:BDH}}
\label{section:BDH}
The following proof is based on \cite[Section 17.2]{IK}.
Decomposing with Dirichlet characters and applying orthogonality gives
\begin{equation}
\label{eq:BDHfirsteq}
  \sum_{q \leq Q} \thinspace \sumstar_{a \shortmod{q}}
  \Big| \sum_{\substack{\hit(n) \leq X \\ n \equiv a \shortmod{q}}} \alpha_n
  - \frac{1}{\varphi(q)} \sum_{\substack{\hit(n) \leq X \\ (n,q) = 1}} \alpha_n
  \Big|^2 
  =
 \sum_{q \leq Q} 
 \sum_{\substack{\chi \shortmod{q} \\ \chi \neq \chi_0}} 
 \frac{1}{\varphi(q)} | S(X, \chi)|^2.
\end{equation}
Write $q = q_0 q'$ and $\chi = \chi_0 \chi'$, where $\chi$ has conductor $q'$.  Then \eqref{eq:BDHfirsteq} is at most
\begin{equation*}
 \sum_{\substack{q_0 q' \leq Q \\ q' > 1}} 
 \thinspace
 \sumstar_{\substack{\chi' \shortmod{q'} }} 
 \frac{1}{\varphi(q_0) \varphi(q')} | S(X, \chi' \chi_0)|^2.
\end{equation*}
We break up this sum according to $q' \leq Q_0 = (\log X)^B$ and $q' > Q_0$.  For $q' \leq Q_0$, we apply the S-W condition \eqref{eq:SWcondition}, giving a bound of the form
\begin{equation*}
 \sum_{q_0 \leq Q} \frac{\tau_k(q_0)^2}{\varphi(q_0)} \sum_{1 < q' \leq Q_0}
 \frac{X | \alpha|^2}{(\log X)^{2B}}
 \ll (\log Q)^{(k+1)^2} \frac{X |\alpha|^2}{(\log X)^B}.
\end{equation*}

The terms with $Q_0 < q' \leq Q/q_0$ are bounded by
\begin{equation*}
 \ll \sum_{q_0 \leq Q} \frac{1}{\varphi(q_0)} 
 \sum_{\substack{Q_0 \leq R \leq Q/q_0 \\ \text{dyadic}}} 
 R^{-1+\varepsilon} \Delta(R, X) \sum_{\substack{\hit(n) \leq X \\ (n,q_0) = 1}} |\alpha_n|^2.
\end{equation*}
For $R \leq (XQ)^{1/10}$, we use the ``$\varepsilon$-free'' bound $\Delta(R, X) \ll (R^{4} + X \log{X})$ (see \eqref{eq:DeltaTrivialBound}), while for $R > (XQ)^{1/10}$, we use Theorem \ref{thm:mainthm}.  In total, we obtain the following bound for the terms with $q' > Q_0$:
\begin{equation*}
 |\alpha|^2 \sum_{q_0 \leq Q} \frac{1}{\varphi(q_0)} \Big(Q^{1+\varepsilon} + \frac{X}{(\log X)^{B(1-\varepsilon) - 1}}\Big)
 \ll \Big(Q^{1+\varepsilon} + \frac{X}{(\log X)^{B(1-\varepsilon) - 2}}\Big) |\alpha|^2.
\end{equation*}
Choosing $B(1-\varepsilon) - 2 > A$ completes the proof of Proposition \ref{prop:BDH}.

\section{Proof of Theorem \ref{thm:recursivethmFE}}
\label{section:FE}
\subsection{Miscellany}
We begin with some miscellaneous results that will be useful later.
\begin{defiplaintext}[A partition of unity]
\label{defi:dyadicpartition}
Let $T \geq 1$, $\varepsilon > 0$.  Choose smooth and even functions $\omega_0$ and $\omega_{T'}(r) = \omega(r/T')$ so that for all $|r| \ll T$ we have
\begin{equation}
\label{eq:dyadicpartitionomega}
\omega_0(r) + \sum_{T' \text{ dyadic}} \omega_{T'}(r)
= 1,
\end{equation}
where
$\omega_0(r)$ is supported on $r \ll T^{\varepsilon}$, $\omega$ is supported on $[1,2] \cup [-2, -1]$, and $T'$ runs over 
$O(\log T)$ real numbers with
$T^{\varepsilon} \ll T' \ll T$.  
\end{defiplaintext}
It is convenient to re-write the left hand side of \eqref{eq:dyadicpartitionomega} as $\sum_{T'} \omega_{T'}$, where $T'$ runs over the dyadic numbers from Definition \ref{defi:dyadicpartition}, along with an additional value $T'=1$ giving rise to $\omega_0$.

\begin{mylemma}
\label{lemma:archimedeanCoset}
 Let $w$ be an integrable function supported on $[U, 2U]$, with $1 \leq U \leq 2 T$.  Suppose $\beta_t \in L^2(\mr)$, supported on $[T/2, T]$.  Then
 \begin{multline}
\intR \intR \beta_{t_1} \overline{\beta_{t_2}} w(t_1 - t_2) dt_1 dt_2
\\
= 
\sum_{\substack{ 0 \leq j_1, j_2 \leq 10 T/U \\
|j_1 - j_2| \leq 1}}
\int_{U }^{2U} 
\int_{U}^{2U} \beta_{T-U +U j_1 + v_1} \overline{\beta_{T-U + U j_2 + v_2}}
w(U (j_1-j_2) + v_1 - v_2) dv_1 dv_2.
 \end{multline}
\end{mylemma}
\begin{proof}
 We cover the interval $[T/2, T]$ without overlaps by smaller intervals $[T/2, T/2 + U]$, $[T/2+U, T/2+ 2U]$, $\dots$, giving
\begin{equation}
\intR \intR \beta_{t_1} \overline{\beta_{t_2}} w(t_1 - t_2) dt_1 dt_2
= 
\sum_{\substack{ 0 \leq j_1, j_2 \leq 10 T/U }}
\int_{T/2 + U j_1}^{T/2 +   U j_1 + U} \beta_{t_1}
\int_{T/2 + U j_2}^{T/2 +  U j_2 + U} \overline{\beta_{t_2}}
w(t_1 - t_2) dt_1 dt_2.
 \end{equation} 
Next change variables $t_i = T/2 -U  + Uj_i + v_i$ for $i=1,2$, where $U \leq v_i \leq 2U$.   Note that the integrand vanishes unless $|j_1 - j_2| \leq 1$.  The result follows.
 \end{proof}

\begin{mylemma}[Archimedean separation of variables]
\label{lemma:archimedeanseparationofvariables}
For $s= \sigma + iy$ with $\sigma > 0$ fixed, $|r| \leq T$, and $|y| \leq |r|^{1/2}$, let
\begin{equation}
\label{eq:fdef}
\gamma(r) = \gamma_{s}(r) = \frac{\Gamma_{\mr}(\sigma+iy+ir) \Gamma_{\mr}(\sigma +iy - ir)}{\Gamma_{\mr}(1-\sigma-iy+ir) \Gamma_{\mr}(1-\sigma -iy - ir)}.
\end{equation}
Let $\omega$ and $\omega_0$ be as in Definition \ref{defi:dyadicpartition}. Then for $T'$ satisfying $1 + |s|^2 \ll T' \leq T$, there exists a function $\eta = \eta_{T'}$ satisfying
\begin{equation}
\label{eq:etabound}
\eta_{T'}(u) \ll (T')^{2\sigma} (1+|u| T')^{-A},
\qquad \text{and}
\qquad
\intR |\eta_{T'}(u)| du \ll (T')^{2\sigma -1},
\end{equation}
so that
\begin{equation}
\label{eq:gammawFourierIntegralDef}
\gamma(r) \omega_{T'}(r) = \intR \eta_{T'}(u) e(ur) du.
\end{equation}
If $|s| \ll T^{\varepsilon}$ and $T' = 1$ (that is, $\omega_{T'} = \omega_0$), then 
\eqref{eq:gammawFourierIntegralDef} holds with 
\begin{equation}
\label{eq:etaboundVariant}
 \eta_{1}(u) \ll T^{\varepsilon} \Big(1 + \frac{|u|}{T^{\varepsilon}}\Big)^{-A}.
\end{equation}
\end{mylemma}
\begin{proof}
A tedious but straightforward calculation with Stirling's approximation gives
\begin{equation*}
\gamma(r) = \Big(\frac{|r|}{2}\Big)^{2s-1} (c_0 + \frac{c_1}{r^2} + \dots),
\end{equation*}
where the $c_i$ are some polynomials in $s$, of degree at most $2i+1$.  This provides an asymptotic expansion as $r \rightarrow \infty$ provided $s \ll |r|^{1/2}$, say.  From this, one may derive
\begin{equation}
\label{eq:fderivatives}
\gamma^{(j)}(r) \ll |r|^{2\sigma -1 - j}, \quad \text{for} \quad  |r| \gg |s|^2 + 1.
\end{equation}
By Fourier inversion, we have
\begin{equation*}
\gamma(r)\omega (r/T') = \intR \eta_{T'}(u) e(ur) du,
\qquad
\eta_{T'}(u) = \intR \gamma(r) \omega(r/T') e(-ur) dr.
\end{equation*}
Integration by parts, aided with \eqref{eq:fderivatives}, gives \eqref{eq:etabound}.
% \begin{equation}
% \label{eq:gT'bound}
% \eta_{T'}(u) \ll_{A} (T')^{2\sigma} (1+|u|T')^{-A}. 
% \end{equation}
For $T'=1$ and $|s| \ll T^{\varepsilon}$, then the asymptotic Stirling formula does not hold, yet we can claim a crude but uniform upper bound of the form
%\begin{equation}
$\gamma^{(j)}(r) \ll (T^{\varepsilon})^j$, which suffices to give \eqref{eq:etaboundVariant}.
%\end{equation}
\end{proof}

\begin{mycoro}
\label{coro:archimedeanseparationofvariables}
 Let $\gamma = \gamma_s$ be as in \eqref{eq:fdef}, and suppose $b_t \in L^2(\mr)$, supported on $[T/2, T]$.  Suppose $s \ll T^{o(1)}$.  
 Suppose $\omega_{T'}$ is as in Definition \ref{defi:dyadicpartition} for some $1 \ll T' \ll T$.
 Then
\begin{multline}
 \intR \intR \beta_{t_1} \overline{\beta_{t_2}} \gamma(t_1 - t_2) \omega_{T'}(t_1 - t_2) dt_1 dt_2
 = 
\sum_{\substack{ 0 \leq j_1, j_2 \leq 10 T/T' \\
|j_1 - j_2| \leq 1}}
\intR \eta_{T'}(u) e(u T'(j_1 - j_2))
\\
\times
\Big(
\int_{T'}^{2T'} \beta_{T/2 - T' +  T' j_1 + v_1}  e( v_1 u) dv_1 \Big)
\Big( \int_{T' }^{ 2T' } \overline{\beta_{T/2 - T' + T' j_2 + v_2}} e(-v_1 u) dv_2 \Big)
du,
\end{multline}
with $\eta_{T'}$ as in Lemma \ref{lemma:archimedeanseparationofvariables}.
% \begin{equation}
% \label{eq:dyadicpartitionv1}
% \int_{t_1, t_2} b_{t_1} \overline{b_{t_2}} \gamma(t_1 - t_2) dt_1 dt_2
% = 
% \sum_{\substack{T' \text{ dyadic} \\ T' \ll T}} 
% \intR \eta_{T'}(u) 
% \Big|\int_{t} b_{t} e(ut) dt\Big|^2 du,
% \end{equation}
% where
% $\eta_{T'}$ satisfies 
% \eqref{eq:etabound} for $T' \gg T^{\varepsilon}$ and \eqref{eq:etaboundVariant} for $T'=1$.  
% An alternative formula is
% \begin{multline}
% \label{eq:dyadicpartitionv2}
% \int_{t_1, t_2} b_{t_1} \overline{b_{t_2}} \gamma(t_1 - t_2) dt_1 dt_2
% = \int_{t_1, t_2} b_{t_1} \overline{b_{t_2}} \gamma(t_1 - t_2) w_0(t_1 - t_2) dt_1 dt_2
% \\
% +
% \sum_{\substack{T' \text{ dyadic} \\ T^{o(1)} \ll T' \ll T}} 
% \int_{t_1, t_2} b_{t_1} \overline{b_{t_2}} \gamma(t_1 - t_2) w_{T'}(t_1 - t_2) dt_1 dt_2.
% \end{multline}
\end{mycoro}
\begin{proof}
This follows from Lemma \ref{lemma:archimedeanCoset} followed by \eqref{eq:gammawFourierIntegralDef}.
%Applying \eqref{eq:dyadicpartitionomega} gives \eqref{eq:dyadicpartitionv2}; for \eqref{eq:dyadicpartitionv1}, we simply apply Lemma \ref{lemma:archimedeanseparationofvariables}.
\end{proof}

\subsection{Preparation}
\label{section:breakdown}
Here we begin the proof of Theorem \ref{thm:recursivethmFE}.
Choose a nonnegative smooth weight function $w$, with $w(x) \geq 1$ for $1/2 \leq x \leq 1$, and $w(x) = 0$ for $x < 1/4$ and for $x \geq 2$. From \eqref{eq:Delta*Def}, we have
$\Delta^*(Q,k,T,N) \leq \max_{|{\mathbf \beta}|=1} S$, where
\begin{equation}
\label{eq:Sdef}
S = \sum_{\substack{ (a,b) = 1}} w(ab/N)
\Big|
\int_{T/2 \leq t \leq T}
\sum_{\substack{Q/2 < q \leq Q \\ (q,k) = 1}}
\thinspace
\sumstar_{\chi \shortmod{q}}
\thinspace
\sum_{\theta \shortmod{k}}
\beta_{\chi, \theta, t}
% \chi \theta(a) \overline{\chi \theta}(b) (a/b)^{it}
\lambda_{\chi \theta, t}(a,b)
\Big|^2.
\end{equation}
We will assume that $\beta_{\chi, \theta, t}$ is supported on 
\begin{equation}
\label{eq:betasupport}
\cond(\chi) = q, \qquad Q/2 < q \leq Q, \qquad (q,k) = 1,  \qquad \theta \mymod{k}, \qquad T/2 \leq t \leq T,
\end{equation}
and that an otherwise un-labeled integral/sum over $t$,$q$,$\chi$,$\theta$ is implied to run over this domain.
In particular, we will often suppress these conditions and recall them only when needed.
To prove Theorem \ref{thm:recursivethmFE}, it suffices to prove the bound for $\chi$ and $\theta$ of fixed parities, so for convenience we also assume that this condition is enforced by the support of $\beta_{\chi, \theta, t}$.
%Another notational convenience is that we may write $\chi(a) \overline{\chi}(b)$ as $\chi(a \overline{b})$, which simply means that we interpet $\chi(\overline{b})$ as $0$ if $b$ is not coprime to the modulus of $\chi$.

Let $1 \leq Y \leq \frac{N}{100}$ be a parameter to be chosen later.  Then $S \leq S_{>Y}$, where
\begin{equation}
\label{eq:SgrtYdef}
 S_{>Y}  = \sum_{\substack{\frac{ab}{(a,b)^2} > Y}} w(ab/N)
\Big|
\int_{T/2 \leq t \leq T}
\sum_{\substack{Q/2 < q \leq Q \\ (q,k) = 1}}
\thinspace
\sumstar_{\chi \shortmod{q}}
\thinspace
\sum_{\theta \shortmod{k}}
\beta_{\chi, \theta, t}
% \chi \theta(a) \overline{\chi \theta}(b) (a/b)^{it}
\lambda_{\chi \theta, t}(a,b)
\Big|^2,
\end{equation}
by positivity, since if $(a,b) =1$, then the condition $ab > Y$ is redundant to the support of $w(ab/N)$.  By simple inclusion-exclusion, we have
\begin{equation*}
 S_{>Y} = S_{\leq \infty} - S_{\leq Y},
\end{equation*}
where for $* \in \{Y, \infty \}$, $S_{\leq *}$ corresponds to the sum over $\frac{ab}{(a,b)^2} \leq *$.  We will often write $S_{\infty}$ as an alias for $S_{\leq \infty}$.  
%The two sums $S_{\infty}$ and $S_{\leq Y}$ will be treated with different methods.  For $S_{\infty}$, we will apply the functional equation of Dirichlet $L$-functions, and for $S_{\leq Y}$ we use more direct moves.
 
One of the main issues with applying the functional equation is that, after opening the square, we obtain a character of the form $\chi_1 \overline{\chi_2} \theta_1 \overline{\theta_2}$ which may be imprimitive.
In order to facilitate the problem of controlling the conductor, we will apply some combinatorial-type decompositions.  These preparatory results are bookended by Lemmas \ref{lemma:detectingprimitivecharacters} and \ref{lemma:chiseparation}.
\begin{mylemma}
[Detecting primitivity]
\label{lemma:detectingprimitivecharacters}
 Let $q \geq 1$ be an integer.  There exist complex numbers $c_{\ell} = c_{\ell}(q)$ supported on a finite set of integers with the following two properties:
 \begin{itemize}
  \item For each $\psi \pmod{q}$, the sum $\sum_{\ell} c_{\ell} \psi(\ell)$ is $1$ if $\psi$ is primitive, and is $0$ if $\psi$ is imprimitive.
  \item We have $
  %\sum_{\ell} |c_{\ell}|^2 \leq 
  \sum_{\ell} |c_{\ell}| \leq \tau(q)$, where $\tau(q)$ denotes the number of divisors of $q$.
 \end{itemize}
\end{mylemma}
\begin{proof}
 Suppose $\psi$ has conductor $q^*$.  Consider the expression
 \begin{equation*}
  \sum_{d | q} \mu(d) \Big(\frac{1}{d} \sum_{y \shortmod{d}}  \psi(1+\frac{q}{d}y) \Big).
 \end{equation*}
The inner sum inside the parentheses is $1$ if $q^*$ divides $q/d$ (equivalently, $d$ divides $q/q^*$), and $0$ otherwise.  Hence the above sum evaluates as $\sum_{d|q/q^*} \mu(d)$, which by M\"obius inversion is the indicator function that $q^* = q$, i.e., that $\psi$ is primitive.
To finish the proof, we can let $c_{\ell}$ be supported on $1 \leq \ell \leq q+1$, and let
\begin{equation}
c_{\ell} = \sum_{d|q} \frac{\mu(d)}{d} \sum_{\substack{1 \leq y \leq d \\ 1 + \frac{q}{d} y = \ell}} 1 
= \sum_{e | (q, \ell - 1)} \frac{\mu(q/e)}{q/e},
\end{equation}
so that $\sum_{\ell} |c_{\ell}| \leq \tau(q)$.
\end{proof}

Suppose $q,r \geq 1$ are integers with $r|q$. %and $q|r^{\infty}$.  
Let $G_q$ (resp. $G_r$) be the group of Dirichlet characters modulo $q$ (resp. $r$).  
By a slight abuse of notation, we can view $G_r$ as a subgroup of $G_q$, by multiplying every element of $G_r$ by the trivial character modulo $q$.  
%Note that if $q$ and $r$ share all their prime factors, then $G_r$ is a subgroup of $G_q$ without this abuse.  
\begin{mylemma}
\label{lemma:Fchi1chi2version1}
Let $q$, $r$, $G_q$, and $G_r$ be as above.
Let $F(\chi_1, \chi_2)$ be a function defined on pairs of Dirichlet characters modulo $q$.  Then
\begin{equation*}
 \sum_{\substack{\chi_1, \chi_2 \shortmod{q} \\ \chi_1 \overline{\chi_2} \text{ modulus } r}} F(\chi_1, \chi_2)
 =
 \sum_{\gamma \in G_q/G_r}
 \sum_{\substack{\psi_1, \psi_2 \shortmod{r}} } F(\gamma \psi_1, \gamma \psi_2).
\end{equation*}
% \begin{equation}
%  \sum_{\substack{\chi_1, \chi_2 \shortmod{q} \\ \chi_1 \overline{\chi_2} \text{ conductor } r}} F(\chi_1, \chi_2)
%  =
%  \sum_{\gamma \in G_q/G_r}
%  \sum_{\substack{\psi_1, \psi_2 \shortmod{r} \\ \psi_1 \overline{\psi_2} \text{ conductor } r} } F(\gamma \psi_1, \gamma \psi_2).
% \end{equation}
\end{mylemma}
Remark. Lemma \ref{lemma:Fchi1chi2version1} is analogous to Lemma \ref{lemma:archimedeanCoset}.
\begin{proof}
The condition that $\chi_1 \overline{\chi_2}$ has modulus $r$ means that $\chi_1 \overline{\chi_2} \in G_r$.  Now say $G_q = \cup_{\gamma} \gamma G_r$, where $\gamma$ runs over $G_q/G_r$.  By basic group theory,
% $\chi_1 \overline{\chi_2} \in G_r$ and $\chi_1 \in \gamma G_r$ are together equivalent to $\chi_1, \chi_2 \in \gamma G_r$.  Hence, 
we can write uniquely $\chi_1 = \gamma \psi_1$ and $\chi_2 = \gamma \psi_2$ with $\gamma \in G_q/G_r$ and $\psi_1, \psi_2 \in G_r$.   
\end{proof}

\begin{mycoro}[Separation of variables]
\label{coro:sepvarsfixedconductor}
 Let notation be as in Lemma \ref{lemma:Fchi1chi2version1}.  Then
 \begin{equation*}
 \sum_{\substack{\chi_1, \chi_2 \shortmod{q} \\ \chi_1 \overline{\chi_2} \text{ conductor } r}} F(\chi_1, \chi_2)
 =
 \sum_{\ell} c_{\ell}(r) 
 \sum_{\gamma \in G_q/G_r}
 \sum_{\substack{\psi_1, \psi_2 \shortmod{r}} } (\psi_1 \overline{\psi_2})(\ell) F(\gamma \psi_1, \gamma \psi_2).
\end{equation*}
\end{mycoro}
\begin{proof}
 We first apply Lemma \ref{lemma:Fchi1chi2version1} to detect that $\chi_1 \overline{\chi_2}$ has modulus $r$, and then use Lemma \ref{lemma:detectingprimitivecharacters} to detect that $\psi_1 \overline{\psi}_2$ is primitive.   
\end{proof}

\begin{defiplaintext}
\label{defi:kfactorization}
Let $k \geq 1$ be an integer.  
Define the set $D_k$ to consist of tuples
 ${\bf k} = (k_0, k_1, k', \delta)$, where $k_0$, $k_1$, $k'$ run over divisors of $k$ with $k_0 k_1 k' = k$, $(k_0, k') = 1$, and $k_1 | (k')^{\infty}$, and where $\delta$ runs over coset representatives of $G_k/G_{k'}$.
\end{defiplaintext}
\begin{mylemma}
\label{lemma:thetaseparation}
 Let $k \geq 1$ be an integer, and let $b_{\theta}$ be any sequence of complex numbers indexed by Dirichlet characters $\theta$ modulo $k$.  Then we have a decomposition of the form
 \begin{equation}
 \label{eq:thetadoublesumdecompositionv1}
  \Big| \sum_{\theta \shortmod{k}} b_{\theta} \Big|^2
  =
  \sum_{{\bf k} \in D_k}
  \sum_{\ell} c_{\ell}(k')  
  \Big| \sum_{\theta' \shortmod{k'}} b_{\delta \theta'} \theta'(\ell) \Big|^2,
 \end{equation}
which can alternatively be written as
 \begin{equation}
 \label{eq:thetadoublesumdecompositionv2}
 \Big| \sum_{\theta \shortmod{k}} b_{\theta} \Big|^2
 =
 %\sum_{\substack{k_0 k_1 k' = k \\ k_1 | (k')^{\infty} \\ (k_0, k') = 1}}
 % \sum_{\delta \in G_k/G_{k'}}
 \sum_{{\bf k} \in D_k}
  \sum_{\substack{\theta_1', \theta_2' \shortmod{k'} \\ \cond(\theta_1' \overline{\theta_2'}) = k'}}
  b_{\delta \theta_1'} \overline{b_{\delta \theta_2'}}.
 \end{equation}
 \end{mylemma}
\begin{proof}
 Begin by opening the square, obtaining a double sum $\sum_{\theta_1, \theta_2 \mymod{k}} b_{\theta_1} \overline{b_{\theta_2}}$.  
 Parameterizing the sum according to the conductor (say $k'$) of $\theta_1 \overline{\theta_2}$, we obtain
 \begin{equation*}
% \label{eq:bthetaksum12}
  \Big| \sum_{\theta \shortmod{k}} b_{\theta} \Big|^2
  =
  \sum_{k'|k} \sum_{\substack{\theta_1, \theta_2 \shortmod{k} \\ \cond(\theta_1 \overline{\theta_2}) = k'}} b_{\theta_1} \overline{b_{\theta_2}}.
 \end{equation*}
 Next we apply Corollary \ref{coro:sepvarsfixedconductor} with
  $F(\theta_1, \theta_2) = b_{\theta_1} \overline{b_{\theta_2}}$,  
giving
\begin{equation*}
  \Big| \sum_{\theta \shortmod{k}} b_{\theta} \Big|^2
  =
  \sum_{k'|k} 
  \sum_{\ell} c_{\ell}(k')
  \sum_{\delta \in G_{k}/G_{k'}}
  \sum_{\substack{\theta_1', \theta_2' \shortmod{k'} }} 
  (\theta_1' \overline{\theta_2'})(\ell)
  b_{\delta \theta_1'} \overline{b_{\delta \theta_2'}}.
 \end{equation*}
 With a further factorization $k_0 k_1 = \frac{k}{k'}$ with $(k_0, k') = 1$ and $k_0 | (k')^{\infty}$, we obtain \eqref{eq:thetadoublesumdecompositionv1}.
 The variant \eqref{eq:thetadoublesumdecompositionv2} is similar. 
\end{proof}

% Note that
% \begin{equation}
% \label{eq:theta1theta2barstarformula}
% \theta_1 \overline{\theta_2} 
% = 
% \underbrace{(\theta_1 \overline{\theta_2})^*}_{k'} 
% \underbrace{\chi_{0, k_0 k_1}}_{k_0 k_1} = \underbrace{(\theta_1 \overline{\theta_2})^*}_{k'} \underbrace{\chi_{0, k_0}}_{k_0},
% \quad
% \text{and}
% \quad
% (\theta_1 \overline{\theta_2})^* = \theta_1' \overline{\theta_2'},
% \end{equation}
% where $\chi_{0, q}$ denotes the principal character modulo $q$.

We also need more elaborate versions of Definition \ref{defi:kfactorization} and Lemma \ref{lemma:thetaseparation} to handle $\chi$ of varying modulus.
\begin{defiplaintext}
\label{defi:chiifactorization}
For $i=1,2$, suppose $\chi_i$ is primitive of conductor $q_i$.  Factor
\begin{equation}
%\label{eq:qifactorization}
\label{eq:chiifactorization}
q_i = q_i' q_i^{+} q_i^{-} r, 
\qquad
\text{and}
\qquad \chi_i = \chi_i' \chi_i^{+} \chi_i^{-} \chi_i^{(r)},
\end{equation}
where $\chi_i'$ has conductor $q_i'$, $\chi_i^{+}$ has conductor $q_i^{+}$, and so on, and
the factorization is 
% determined as follows.  The terms appearing on the right hand side of \eqref{eq:chiifactorization} are 
defined in terms of local information as follows. 
\begin{enumerate}[label=(\roman*)]
\item The primes making up $q_1'$ are those that divide $q_1$ but do not divide $q_2$, and likewise the primes in $q_2'$ are those that divide $q_2$ but not $q_1$.  
\item
The factors $q_1^{+}$ and $q_2^{-}$ are characterized by $1 \leq v_p(q_2^{-}) < v_p(q_1^{+})$ for all $p|q_1^{+}$.  Similarly, 
$q_2^{+}$ and $q_1^{-}$ are characterized by $1 \leq v_p(q_1^{-}) < v_p(q_2^{+})$ for all $p|q_2^{+}$.
\item The remaining factor $r$ corresponds to the primes where $v_p(q_1) = v_p(q_2)$.  
\end{enumerate}
\end{defiplaintext}
Definition \ref{defi:chiifactorization} is motivated by the fact that
\begin{equation}
 \label{eq:chi1chi2barformulaProto}
 \chi_1 \overline{\chi_2}
 =
 \underbrace{\chi_1'}_{q_1'} 
\underbrace{\overline{\chi_2'}}_{q_2'} 
\underbrace{(\chi_1^{+} \overline{\chi_2^{-}})}_{q_1^{+}}
\underbrace{(\chi_1^{-} \overline{\chi_2^{+}})}_{q_2^{+}}
\chi_1^{(r)} \overline{\chi_2^{(r)}},
\end{equation}
which has conductor $q_1' q_2' q_1^{+} q_2^{+} \cond(\chi_1^{(r)} \overline{\chi_2^{(r)}})$.

Let $b_{\chi}$ be any sequence of complex numbers indexed by primitive Dirichlet characters $\chi$ modulo $q$, with $q$ varying over a finite set of positive integers.
Consider the sum $|\sum_{q,\chi} b_{\chi}|^2$.  
Opening the square gives a sum of the form $\sum_{q_1, q_2, \chi_1, \chi_2} b_{\chi_1} \overline{b_{\chi_2}}$.  Definition \ref{defi:chiifactorization} shows that the parameters $q_i'$, $q_i^{+}$, etc., are uniquely determined.  We can then arrange the sum according to the values of these parameters, giving
\begin{equation}
\label{eq:bchisumsquaredv2}
\Big| \sum_{q, \chi} b_{\chi} \Big|^2 =
\sum_{\substack{q_1^{+}, q_1^{-}, q_2^{+}, q_2^{-}, r \\ (\text{Def. } \ref{defi:chiifactorization}) }} 
\Big(\sum_{\substack{q_1', \chi_1', \chi_1^{+}, \chi_1^{-}, \chi_1^{(r)} \\ (\text{Def. } \ref{defi:chiifactorization})}} b_{\chi_1' \chi_1^{+} \chi_1^{-} \chi_1^{(r)}} \Big)
\Big(\sum_{\substack{q_2', \chi_2', \chi_2^{+}, \chi_2^{-}, \chi_2^{(r)} \\ (\text{Def. } \ref{defi:chiifactorization}) }} \overline{b_{\chi_2' \chi_2^{+} \chi_2^{-} \chi_2^{(r)}}} \Big),
\end{equation} 
where the reference to (Def.\ \ref{defi:chiifactorization}) in the summation conditions indicates the conditions translated into appropriate summation form.  

We further develop the sums over $\chi_1^{(r)}$ and $\chi_2^{(r)}$, using \eqref{eq:thetadoublesumdecompositionv2}.  Specifically, write
\begin{equation}
 r = r_0 r_1 r',
\end{equation}
where $\chi_1^{(r)} \overline{\chi_2^{(r)}}$ has conductor $r'$, $(r_0, r') = 1$, and $r_1 | (r')^{\infty}$.  We then write $\chi_i^{(r)} = \gamma \psi_i$, where $\gamma$ runs over $G_{r}/G_{r'}$ and $\psi_i$ run over characters modulo $r'$.  The property that $\chi_1^{(r)} \overline{\chi_2^{(r)}}$ has conductor $r'$ is equivalent to $\psi_1 \overline{\psi_2}$ is primitive (of modulus $r'$).  Applying this to \eqref{eq:bchisumsquaredv2}, we obtain
that $\sum_{q,\chi} |b_{\chi}|^2$ equals
\begin{equation}
\label{eq:bchisumsquaredv3}
\sum_{\substack{q_1^{+}, q_1^{-}, q_2^{+}, q_2^{-}, r \\
(r_0, r_1, r', \gamma) \in D_r \\ (\text{Def.\ } \ref{defi:chiifactorization}) }} 
\Big(\sum_{\substack{q_1', \chi_1', \chi_1^{+}, \chi_1^{-}, \psi_1 \\ (\text{Def.\ } \ref{defi:chiifactorization})}} b_{\chi_1' \chi_1^{+} \chi_1^{-} \gamma \psi_1} \Big)
\Big(\sum_{\substack{q_2', \chi_2', \chi_2^{+}, \chi_2^{-}, \psi_2 \\ (\text{Def.\ } \ref{defi:chiifactorization}) }} \overline{b_{\chi_2' \chi_2^{+} \chi_2^{-} \gamma \psi_2}} \Big)
\delta(\cond(\psi_1 \overline{\psi_2}) = r').
\end{equation}
Now let
\begin{equation*}
{\bf q} = (q_1^{+}, q_1^{-}, q_2^{+}, q_2^{-}, r_0, r_1, r', \gamma)
\end{equation*}
where the integers $q_i^{\pm}$ satisfy Def.\ \ref{defi:chiifactorization}(ii), $r$ is coprime to the $q_i^{\pm}$, and $(r_0, r_1, r', \gamma) \in D_r$ (as in Def.\ \ref{defi:kfactorization}).  The two sums in parentheses in \eqref{eq:bchisumsquaredv3} have only the following 
conditions
between \emph{each other}: $q_1'$ and $q_2'$ are coprime, and the conductor of $\psi_1 \overline{\psi_2}$ is $r'$.  
%Any other implied conditions from Def.\ \ref{defi:chiifactorization} are in terms of ${\bf q}$ and an inner variable.
We have thus derived the following.
\begin{mylemma}
\label{lemma:chiseparation}
Let $b_{\chi}$ be any sequence of complex numbers indexed by primitive Dirichlet character $\chi$ modulo $q$, with $q$ varying over a finite set of positive integers.  Then
\begin{equation}
\Big| \sum_{q, \chi} b_{\chi} \Big|^2 = 
\sum_{{\bf q}} 
\sum_{\substack{q_i', \chi_i', \chi_i^{+}, \chi_i^{-}, \psi_i \\ 
(q_1', q_2') = 1, \thinspace
\psi_1 \overline{\psi_2} \text{ prim.}
\\
(\text{Def.\ } \ref{defi:chiifactorization})}} 
b_{\chi_1' \chi_1^{+} \chi_1^{-} \gamma \psi_1} 
%(\sum_{\substack{q_2', \chi_2', \chi_2^{+}, \chi_2^{-}, \psi_2 \\ (\text{Def.\ } \ref{defi:chiifactorization}) }} 
\overline{b_{\chi_2' \chi_2^{+} \chi_2^{-} \gamma \psi_2}}.
\end{equation}
\end{mylemma}
In reference to \eqref{eq:chi1chi2barformulaProto}, 
now $\chi_1^{(r)} \overline{\chi_2^{(r)}} = \psi_1 \overline{\psi_2} |\gamma|^2$, which has conductor $r'$, so $\chi_1 \overline{\chi_2}$ has conductor $q_1' q_2' q_1^{+} q_2^{+} r'$.

We are now ready to apply the preceding decompositions to  $S_{\leq *}$ (see \eqref{eq:SgrtYdef} to infer the definition).  Specifically, we apply Lemma \ref{lemma:thetaseparation} (in the form \eqref{eq:thetadoublesumdecompositionv2}) and Lemma \ref{lemma:chiseparation}, giving
\begin{equation}
\label{eq:SdefInTermsofSkq}
S_{\leq *} = \sum_{\substack{{\bf k} \\ (\text{Def } \ref{defi:kfactorization})}}
\sum_{\substack{{\bf q} \\ (\text{Def } \ref{defi:chiifactorization})}}
S_{\leq *}({\bf k}, {\bf q}),
\end{equation}
where
\begin{equation}
\label{eq:S*decomposed}
S_{\leq *}({\bf k}, {\bf q})
= 
  \sum_{\substack{\theta_1', \theta_2' \shortmod{k'} \\ \theta_1' \overline{\theta_2'} \text{ prim.}}}
   \sum_{\substack{q_i',
   \chi_i', \chi_i^{+}, \chi_i^{-}, \psi_i
   \\
   (q_1', q_2') = 1, \thinspace
\psi_1 \overline{\psi_2} \text{ prim.} \\
   (\text{Def } \ref{defi:chiifactorization})
   }} 
 \int_{t_1, t_2} 
 \beta_1 \overline{\beta_2}
% \beta_{\chi_1' \chi_1^{+} \chi_1^{-} \gamma \psi_1, \delta \theta_1', t_1} \overline{\beta}_{\chi_2' \chi_2^{+} \chi_2^{-} \gamma \psi_2, \delta \theta_2', t_2}
 \sum_{\substack{\frac{ab}{(a,b)^2} \leq * \\ (ab, k_0 r_0) = 1}} w\Big(\frac{ab}{N}\Big) 
\Phi(a \overline{b})
 dt_1 dt_2
\end{equation}
with
\begin{equation}
\beta_i = \beta_{\chi_i' \chi_i^{+} \chi_i^{-} \gamma \psi_i, \delta \theta_i', t_i},
\end{equation}
and where $\Phi = \Phi_1 \overline{\Phi_2}$, with
\begin{equation*}
\Phi_i(m) = (\chi_i' \chi_i^{+} \chi_i^{-} \psi_i  \theta_i')(m) m^{it_i}.
\end{equation*}
We remind the reader that there are additional conditions encoded in the support of the coefficients, as recorded in \eqref{eq:betasupport}, which will be recalled as needed.
Observe that the finite part of $\Phi$ (i.e., omitting $m^{it_1 -it_2}$) is primitive of modulus $q_1' q_2' q_1^{+} q_2^{+} r' k'$.  It is convenient to record here for later purposes that for $i=1,2$,
\begin{equation}
\label{eq:betainorm}
\sum_{{\bf k}, {\bf q}} |\beta_i|^2 := 
\sum_{{\bf k}, {\bf q}}
\int_{t_i}
\sum_{q_i', \chi_i', \chi_i^{+}, \chi_i^{-}, \psi_i, \theta_i'} 
 | \beta_{\chi_i' \chi_i^{+} \chi_i^{-} \gamma \psi_i, \delta \theta_i', t_i}|^2 dt_i
 \ll (kQ)^{\varepsilon} | \beta|^2.
\end{equation}

At this point our treatments of $S_{\leq *}$ for $* = Y$ and $*=\infty$ diverge.
\subsection{Elementary side}
In this section we develop $S_{\leq Y}({\bf k}, {\bf q})$.
\begin{myprop}
\label{prop:SgrtYstatement}
We have $S_{\leq Y}({\bf k}, {\bf q}) = S_{\leq Y}^{(0)}({\bf k}, {\bf q}) + S_{\leq Y}'({\bf k}, {\bf q})$, where
$S_{\leq Y}^{(0)}({\bf k}, {\bf q})$ is given by \eqref{eq:SkqgSievedSeparatedVariablesPolarTerm} below, 
and where %$|S_{\leq Y}'({\bf k}, {\bf q})|$ satisfies
\begin{equation}
\label{eq:SgrtYbound}
 |S_{\leq Y}'({\bf k}, {\bf q})|
 \lesssim
 \prod_{i=1}^{2} 
 \overline{\Delta}\Big(\frac{Q}{q_i^{+} q_i^{-} r' r_0 r_1}, q_i^{+} q_i^{-} r' k',T, Y\Big)^{1/2} |\beta_i|.
\end{equation}
\end{myprop}
\begin{proof}
Let $g = (a,b)$, and change variables $a \rightarrow g a$ and $b \rightarrow g b$, getting
\begin{multline*}
S_{\leq Y}({\bf k}, {\bf q})
= 
\sum_{(g, k_0 r_0) = 1}
  \sum_{\substack{\theta_1', \theta_2' \shortmod{k'} \\ \theta_1' \overline{\theta_2'} \text{ prim.}}}
   \sum_{\substack{q_i',
   \chi_i', \chi_i^{+}, \chi_i^{-}, \psi_i
   \\
   (q_1', q_2') = 1, \thinspace
\psi_1 \overline{\psi_2} \text{ prim.} \\
   (\text{Def } \ref{defi:chiifactorization})
   }} 
   \\
 \int_{t_1, t_2} 
 \beta_1 \overline{\beta_2}
% \beta_{\chi_1' \chi_1^{+} \chi_1^{-} \gamma \psi_1, \delta \theta_1', t_1} \overline{\beta}_{\chi_2' \chi_2^{+} \chi_2^{-} \gamma \psi_2, \delta \theta_2', t_2}
 \sum_{\substack{ab \leq Y \\ (a,b) = 1 \\ (ab, k_0 r_0) = 1}} w\Big(\frac{g^2 ab}{N}\Big) 
\Phi(a \overline{b} g \overline{g})
 dt_1 dt_2.
\end{multline*}
Next we apply the Mellin inversion formula and evaluate the $g$-sum as a Dirichlet $L$-function of principal character to modulus 
$q_1' q_2' q_1^{+} q_2^{+} k' r' k_0 r_0 $.  We further write 
\begin{equation}
\label{eq:LfunctionTrivialCharacter}
L(2s, \chi_{0, q_1' q_2' q_1^{+} q_2^{+} r' k' r_0 k_0}) = \zeta(2s)  \rho_{q_1'} \rho_{q_2'} \rho_{q_1^{+}} \rho_{q_2^{+}}  \rho_{r' r_0} \rho_{k' k_0},
\end{equation}
where $\rho_n = \rho_{n}(s) = \prod_{p|n} (1-p^{-2s})$.  This gives
\begin{multline*}
S_{\leq Y}({\bf k}, {\bf q})
= 
  \sum_{\substack{\theta_1', \theta_2' \shortmod{k'} \\ \theta_1' \overline{\theta_2'} \text{ prim.}}}
   \sum_{\substack{q_i',
   \chi_i', \chi_i^{+}, \chi_i^{-}, \psi_i
   \\
   (q_1', q_2') = 1, \thinspace
\psi_1 \overline{\psi_2} \text{ prim.} \\
   (\text{Def } \ref{defi:chiifactorization})
   }} 
 \\
 \int_{(2)} \rho_{r' r_0} \rho_{k' k_0}
 \int_{t_1, t_2} 
 \beta_1' \overline{\beta_2'}
% \beta_{\chi_1' \chi_1^{+} \chi_1^{-} \gamma \psi_1, \delta \theta_1', t_1} \overline{\beta}_{\chi_2' \chi_2^{+} \chi_2^{-} \gamma \psi_2, \delta \theta_2', t_2}
\sum_{\substack{ab \leq Y \\ (a,b) = 1 \\ (ab, k_0 r_0) = 1}} 
 \Big(\frac{N}{ab}\Big)^s \frac{\widetilde{w}(s)}{2 \pi i}
\zeta(2s) 
 \Phi(a \overline{b})
 dt_1 dt_2 ds,
\end{multline*}
with $\beta_1' = \beta_1 \rho_{q_1'} \rho_{q_1^{+}}$ and $\overline{\beta_2'} = \overline{\beta_2} \rho_{q_2'} \rho_{q_2^{+}}$.

Next we use Lemma \ref{lemma:detectingprimitivecharacters} to detect the condition that $\theta_1' \overline{\theta_2'}$ is primitive,
and again to detect that $\psi_1 \overline{\psi_2}$ is primitive (modulo $r'$).  We additionally use M\"{o}bius inversion to detect $(q_1', q_2') = 1$, via $\sum_{g'|(q_1', q_2')} \mu(g')$.
Altogether, this gives
\begin{multline}
\label{eq:SkqgSievedSeparatedVariables}
S_{\leq Y}({\bf k}, {\bf q})
= 
\sum_{g'} \mu(g')
\sum_{\ell_1, \ell_2} c_{\ell_1}(k') c_{\ell_2}(r')
\\
 \int_{(2)} N^s \frac{\widetilde{w}(s)}{2 \pi i} 
 \zeta(2s)
 \rho_{r' r_0} \rho_{k' k_0}
\sum_{\substack{(a,b) = 1 \\ ab \leq Y  \\ (ab, k_0 r_0) = 1}} 
\mathcal{A}_1 \overline{\mathcal{A}_2} \frac{ds}{(ab)^s},
\end{multline}
where
\begin{equation*}
\mathcal{A}_1 = \int_{t_1}
\sum_{\substack{q_1',
   \chi_1', \chi_1^{+}, \chi_1^{-}, \psi_1, \theta_1'
   \\
   q_1' \equiv 0 \shortmod{g'} \\
   (\text{Def } \ref{defi:chiifactorization})
   }} 
   \beta_1 \rho_{q_1'} \rho_{q_1^{+}}
   \theta_1'(\ell_1) \psi_1(\ell_2)
   \Phi_1( a \overline{b}) dt_1,
\end{equation*}
and $\mathcal{A}_2$ is similarly-defined. %(the only reason we were unable to write a uniform formula for $\mathcal{A}_i$ is the presence of $\rho_{q_i'} \rho_{q_i^{+}}$ which is not conjugated for $i=2$).

Now we shift the $s$-contour of integration to $\mathrm{Re}(s) = \varepsilon$, crossing a pole at $s=1/2$ only.  
Write $S_{\leq Y}({\bf k}, {\bf q}) = S_{\leq Y}^{(0)}({\bf k}, {\bf q}) + S_{\leq Y}'({\bf k}, {\bf q})$, where
$S_{\leq Y}^{(0)}$ denotes the polar term, and $S_{\leq Y}'$ denotes the new line of integration.  
Note that $\mathcal{A}_i \vert_{s=1/2} = \mathcal{A}_i^{(0)}$, where
\begin{equation}
\label{eq:Ai0def}
\mathcal{A}_i^{(0)} = \int_{t_i}
\sum_{\substack{q_i',
   \chi_i', \chi_i^{+}, \chi_i^{-}, \psi_i, \theta_i'
   \\
   q_i' \equiv 0 \shortmod{g'} \\
   (\text{Def } \ref{defi:chiifactorization})
   }} 
   \beta_i 
   \frac{\phi(q_i' q_i^{+})}{q_i' q_i^{+}}
   \theta_i'(\ell_1) \psi_i(\ell_2)
   \Phi_i(  a \overline{b}) dt_i,
\end{equation}
since $\rho_n(1/2) = \frac{\phi(n)}{n}$.
Therefore, using $(k' k_0, r' r_0) = 1$ for a slight simplification (recalling \eqref{eq:betasupport}), we have
\begin{equation}
\label{eq:SkqgSievedSeparatedVariablesPolarTerm}
S_{\leq Y}^{(0)}({\bf k}, {\bf q})
= 
\sum_{g'} \mu(g')
\sum_{\ell_1, \ell_2} c_{\ell_1}(k') c_{\ell_2}(r')
 \widetilde{w}(1/2) 
 \frac{\phi(k'k_0 r' r_0)}{2 k'k_0 r' r_0}
\sum_{\substack{(a,b) = 1 \\ ab \leq Y \\ (ab, k_0 r_0) = 1} }
\Big(\frac{N}{ab}\Big)^{1/2} 
\mathcal{A}_1^{(0)} \overline{\mathcal{A}_2^{(0)}}.
\end{equation}

Now we estimate $S_{\leq Y}'({\bf k}, {\bf q})$.
We arrange the expression to most closely resemble \eqref{eq:Sdef}, specifically
\begin{equation}
\label{eq:SgrtYinftyBound}
|S_{\leq Y}'({\bf k}, {\bf q})|
\ll (QkN)^{\varepsilon}
\sum_{g'}
\sum_{\ell_1, \ell_2} |c_{\ell_1}(k') c_{\ell_2}(r')|
\max_{\mathrm{Re}(s) = \varepsilon}
\sum_{\substack{(a,b) = 1 \\ ab \leq Y} }
|\mathcal{A}_1 \mathcal{A}_2|.
\end{equation}
Referring back to \eqref{eq:Delta*Def}, and noting that our new family has varying modulus $q_i'$ of size $\frac{Q}{q_i^{+} q_i^{-} r' r_0 r_1}$, and fixed modulus $q_i^{+} q_i^{-} r' k'$, we see
\begin{equation}
\label{eq:AiBound}
\sum_{g'} \sum_{\substack{(a,b) = 1 \\ ab \leq Y }}
| \mathcal{A}_i |^2 
\ll (QkN)^{\varepsilon}  
\max_{1 \leq Y' \leq Y}
\Delta\Big(\frac{Q}{q_i^{+} q_i^{-} r' r_0 r_1}, q_i^{+} q_i^{-} r' k',
T,  Y' \Big) | \beta_i|^2.
\end{equation}
Using Cauchy's inequality and monotonicity (Lemma \ref{lemma:mononicityNaspect})  leads quickly to \eqref{eq:SgrtYbound}.
\end{proof}

\subsection{Functional equation side}
In this section we will apply the functional equation of Dirichlet $L$-functions to $S_{\infty}({\bf k}, {\bf q})$, picking
up from the expression \eqref{eq:S*decomposed}.
To facilitate this, we first apply M\"{o}bius inversion, in the form
\begin{multline}
\label{eq:Mobius2}
\sum_{\substack{(ab, k_0 r_0) = 1}} w\Big(\frac{ab}{N}\Big) 
\Phi(a \overline{b})
\\
=
\sum_{\substack{g_1 | k_0 \\ g_2 | k_0}}
\sum_{\substack{g_3 | r_0 \\ g_4 | r_0}}
\mu(g_1) \mu(g_2) \mu(g_3) \mu(g_4)
\Phi(g_1 g_3 \overline{g_2 g_4})
\sum_{\substack{ a,b}} w\Big(\frac{ g_1 g_2 g_3 g_4 ab}{N}\Big) 
\Phi(a \overline{b}).
\end{multline}
To continue the theme of concise notation, let ${\bf g} = (g_1, g_2, g_3, g_4)$, $\mu({\bf g}) = \mu(g_1) \mu(g_2) \mu(g_3) \mu(g_4)$,  
$\Phi({\bf g}) = \Phi(g_1 g_3 \overline{g_2 g_4})$, and
$|{\bf g}| = g_1 g_2 g_3 g_4$.  The summation condition on ${\bf g}$ is that
\begin{equation}
\label{eq:gsummationcondition}
g_1 | k_0, \quad g_2 | k_0, \quad g_3 | r_0, \quad g_4 | r_0,
\end{equation}
though we will usually suppress this and only recall it as needed.
Then $S_{\infty}({\bf k}, {\bf q})$ equals
\begin{equation*}
\sum_{\substack{{\bf g} \\ \eqref{eq:gsummationcondition} \text{ holds}}} \mu({\bf g})
  \sum_{\substack{\theta_1', \theta_2' \shortmod{k'} \\ \theta_1' \overline{\theta_2'} \text{ prim.}}}
   \sum_{\substack{q_i',
   \chi_i', \chi_i^{+}, \chi_i^{-}, \psi_i
   \\
   (q_1', q_2') = 1, \thinspace
\psi_1 \overline{\psi_2} \text{ prim.} \\
   (\text{Def } \ref{defi:chiifactorization})
   }} 
 \int_{t_1, t_2} 
 \beta_1 \overline{\beta_2}
 \sum_{\substack{a,b}} w\Big(\frac{ab|{\bf g}|}{N}\Big) 
\Phi({\bf g} a \overline{b})
 dt_1 dt_2.
\end{equation*}
We also have need to decompose the $t_i$-integrals to help pin down the archimedean conductor.  Applying the partition from Definition \ref{defi:dyadicpartition}, we obtain that 
$S_{\infty}({\bf k}, {\bf q})$ equals
\begin{equation}
\label{eq:SinftykqFormulaAfterDyadic}
\sum_{{\bf g}, T'} \mu({\bf g})
  \sum_{\substack{\theta_1', \theta_2' \shortmod{k'} \\ \theta_1' \overline{\theta_2'} \text{ prim.}}}
   \sum_{\substack{q_i',
   \chi_i', \chi_i^{+}, \chi_i^{-}, \psi_i
   \\
   (q_1', q_2') = 1, \thinspace
\psi_1 \overline{\psi_2} \text{ prim.} \\
   (\text{Def } \ref{defi:chiifactorization})
   }} 
 \int_{t_1, t_2} 
 \beta_1 \overline{\beta_2} \omega_{T'}(t_1 - t_2)
 \sum_{\substack{a,b}} w\Big(\frac{ab|{\bf g}|}{N}\Big) 
\Phi({\bf g} a \overline{b})
 dt_1 dt_2.
\end{equation}

Define quantities
\begin{equation}
\label{eq:N*def}
%\mathfrak{q} = q_1' q_2' q_1^{+} q_2^{+} r' k',
Q^* = \frac{Q^2 k T'}{q_1^{-} q_2^{-} r' r_0^2 r_1^2 k_0 k_1}
\qquad
N^* =  \frac{Q^4 k^2 (T')^2 |{\bf g}| (QkTN)^{\varepsilon}}{N (q_1^{-} q_2^{-} r' r_0^2 r_1^2 k_0 k_1)^2} = (QkTN)^{\varepsilon} \frac{(Q^*)^2 |{\bf g}|}{N},
\end{equation}
and note that among the variables of summation, $Q^*$ depends only on the outer variables ${\bf q}$, ${\bf k}$, and $T'$, while $N^*$ depends only on ${\bf q}$, ${\bf k}$, $T'$, and ${\bf g}$.

\begin{myprop}
\label{prop:SleqYbound}
We have a decomposition
\begin{equation}
\label{eq:SinftyTheorem}
S_{\infty}({\bf k}, {\bf q}) = 
S_{\infty}^{(0)}({\bf k}, {\bf q})
+ S_{\infty}'({\bf k}, {\bf q})
+ S_{\infty}^{\text{diag}}({\bf k}, {\bf q})
+ \mathcal{E}_{\infty}
,
\end{equation}
with the following properties.  The term
$S_{\infty}^{(0)}({\bf k}, {\bf q})$ is given by \eqref{eq:Sinfty0Def} below,
and $S_{\infty}'({\bf k}, {\bf q})$ satisfies
\begin{equation}
\label{eq:SleqYbound}
|S_{\infty}'({\bf k}, {\bf q})| \lesssim  
 \sum_{{\bf g}, T'} 
 \frac{N}{Q^* |{\bf g}|} 
 \prod_{i=1}^{2}
 \overline{\Delta}\Big(\frac{Q}{q_i^{+} q_i^{-} r' r_0 r_1}, q_i^{+} q_i^{-} r' k', T', N^* \Big)^{1/2} 
 | \beta_i|.
\end{equation}
The diagonal term satisfies the bound
\begin{equation}
\label{eq:SleqYdiagonal}
 \sum_{{\bf k}, {\bf q}} |S_{\infty  }^{\text{diag}}({\bf k}, {\bf q})| \lesssim N  | \beta |^2,
\end{equation}
and the term $\mathcal{E}_{\infty}$ is negligibly small.
\end{myprop}

\begin{proof}[Proof of Proposition \ref{prop:SleqYbound}]
Applying the Mellin inversion formula to $w$ and writing the sum over $a$ and $b$ as a product of Dirichlet $L$-functions 
in \eqref{eq:SinftykqFormulaAfterDyadic} 
gives
\begin{multline*}
S_{\infty}({\bf k}, {\bf q})
= 
\sum_{{\bf g}, T'} \mu({\bf g})
  \sum_{\substack{\theta_1', \theta_2' \shortmod{k'} \\ \theta_1' \overline{\theta_2'} \text{ prim.}}}
   \sum_{\substack{q_i',
   \chi_i', \chi_i^{+}, \chi_i^{-}, \psi_i
   \\
   (q_1', q_2') = 1, \thinspace
\psi_1 \overline{\psi_2} \text{ prim.} \\
   (\text{Def } \ref{defi:chiifactorization})
   }} 
   \\
 \int_{t_1, t_2} \omega_{T'}(t_1 - t_2)
 \Phi({\bf g}) \beta_1 \overline{\beta_2}
  \int_{(2)} \Big(\frac{N}{|{\bf g}|}\Big)^s \widetilde{w}(s) 
L(s, \Phi) L(s, \overline{\Phi}) \frac{ds}{2 \pi i}
 dt_1 dt_2.
\end{multline*}
We shift contours to the line $-\varepsilon$, crossing a pair of poles at $s=1 \pm i(t_1 -t_2)$, which exist only when $\Phi$ is trivial, and 
let $S_{\infty}'({\bf k}, {\bf q})$ be the new integral on the line $-\varepsilon$.
Recall that the finite part of $\Phi$ is primitive of modulus 
\begin{equation}
\mathfrak{q} = q_1' q_2' q_1^{+} q_2^{+} r' k'.
\end{equation}
In particular, $\Phi$ being trivial forces $q_1' = q_2' = q_1^{+} = q_2^{+} = q_1^{-} = q_2^{-} = r' = k' = 1$, and the rapid decay of $\widetilde{w}(s)$ practically forces $|t_1 - t_2| \ll T^{\varepsilon}$.  
It is easy to see that the contribution of this diagonal polar term is consistent with \eqref{eq:SleqYdiagonal}.

On the line $-\varepsilon$ we change variables $s \rightarrow 1-s$.
Note that $L(s, \Phi) L(s, \overline{\Phi})$ satisfies the asymmetric functional equation
\begin{equation}
 L(1-s, \Phi) L(1-s, \overline{\Phi})
 = \mathfrak{q}^{2s-1} 
\gamma_{s} 
 L(s, \Phi) L(s, \overline{\Phi}),
\end{equation}
where $\gamma_{s} = \gamma_s(t_1 - t_2)$ (recall \eqref{eq:fdef} for the definition),
which is holomorphic for $\mathrm{Re}(s) > 0$.  
Recall that the parity of the $\chi_i$ and $\theta_i$ was assumed to be fixed, so that $\chi_1 \overline{\chi_2} \theta_1 \overline{\theta_2}$ is even, and hence the gamma factor is as stated in \eqref{eq:fdef}.
For later use, note that $\gamma_s \vert_{s=1/2} = 1$.  In addition, recall the bound \eqref{eq:fderivatives}, 
which in the present context means $\gamma_s(r) \ll (T')^{2\sigma -1}$.
We then obtain
\begin{multline*}
%\label{eq:SkqgDEF}
S_{\infty}'({\bf k}, {\bf q})
= 
\sum_{\substack{{\bf g}, T'} } \mu({\bf g})
  \sum_{\substack{\theta_1', \theta_2' \shortmod{k'} \\ \theta_1' \overline{\theta_2'} \text{ prim.}}}
   \sum_{\substack{q_i',
   \chi_i', \chi_i^{+}, \chi_i^{-}, \psi_i
   \\
   (q_1', q_2') = 1, \thinspace
\psi_1 \overline{\psi_2} \text{ prim.} \\
   (\text{Def } \ref{defi:chiifactorization})
   }} 
 \int_{t_1, t_2} \omega_{T'}(t_1 - t_2)
\Phi({\bf g}) 
 \beta_1 \overline{\beta_2}
 \\
  \int_{(1+\varepsilon)}
 \widetilde{w}(1-s) 
 \Big(\frac{N}{|{\bf g}|}\Big)^{1-s}
 \mathfrak{q}^{2s-1} \gamma_s
 L(s, \Phi) L(s, \overline{\Phi})
 \frac{ds}{2 \pi i}
 dt_1 dt_2.
\end{multline*}
Next we will re-open the Dirichlet series expansions of the Dirichlet $L$-functions.  A small modification is that we write
\begin{equation*}
L(s, \Phi) = \rho_{\Phi, k_0 r_0} \sum_{(a,k_0 r_0) =1} a^{-s} \Phi(a), \qquad \text{where} \qquad \rho_{\Phi, k_0 r_0} = \prod_{p|k_0 r_0} (1- \Phi(p) p^{-s})^{-1},
\end{equation*}
and likewise for $L(s, \overline{\Phi})$.
This gives
\begin{multline*}
%\label{eq:SkqgDEF}
S_{\infty}'({\bf k}, {\bf q})
= 
\sum_{\substack{{\bf g}, T'} } \mu({\bf g})
  \sum_{\substack{\theta_1', \theta_2' \shortmod{k'} \\ \theta_1' \overline{\theta_2'} \text{ prim.}}}
   \sum_{\substack{q_i',
   \chi_i', \chi_i^{+}, \chi_i^{-}, \psi_i
   \\
   (q_1', q_2') = 1, \thinspace
\psi_1 \overline{\psi_2} \text{ prim.} \\
   (\text{Def } \ref{defi:chiifactorization})
   }} 
 \int_{t_1, t_2} 
 \omega_{T'}(t_1 - t_2)
\Phi({\bf g}) 
 \beta_1 \overline{\beta_2}
 \\
 \frac{N}{|{\bf g}|\mathfrak{q}}
\sum_{(ab, k_0 r_0) = 1}   
  \int_{(1+\varepsilon)}
 \widetilde{w}(1-s) 
 \Big(\frac{\mathfrak{q}^2 |{\bf g}|}{Nab}\Big)^{s}
 \gamma_{s}
 \Phi(a \overline{b})
 \rho_{\Phi, k_0 r_0}
 \rho_{\overline{\Phi}, k_0 r_0}
 \frac{ds}{2 \pi i}
 dt_1 dt_2.
\end{multline*}
We then factor out the gcd of $a$ and $b$, by writing $g' = (a,b)$ and changing variables $a \rightarrow g' a$ and $b \rightarrow g'b$.  The sum over $g'$ forms a Dirichlet $L$-function of principal character of modulus $\mathfrak{q} k_0 r_0$, which is given by 
%a variant on 
\eqref{eq:LfunctionTrivialCharacter}.
%(with $\rho_{k_0} \rho_{r_0}$ omitted).  
Then $S_{\infty}'({\bf k}, {\bf q})$ equals
\begin{multline*}
%\label{eq:SkqgDEF}
%S_{\infty}'({\bf k}, {\bf q})
%= 
\sum_{\substack{{\bf g}, T'} } \mu({\bf g})
  \sum_{\substack{\theta_1', \theta_2' \shortmod{k'} \\ \theta_1' \overline{\theta_2'} \text{ prim.}}}
   \sum_{\substack{q_i',
   \chi_i', \chi_i^{+}, \chi_i^{-}, \psi_i
   \\
   (q_1', q_2') = 1, \thinspace
\psi_1 \overline{\psi_2} \text{ prim.} \\
   (\text{Def } \ref{defi:chiifactorization})
   }} 
 \int_{t_1, t_2} \omega_{T'}(t_1 - t_2)
\Phi({\bf g}) 
 \beta_1 \overline{\beta_2}
 \int_{(1+\varepsilon)}
 \widetilde{w}(1-s) 
 \\
 \frac{N}{|{\bf g}|\mathfrak{q}}
\sum_{\substack{(a,b)=1 \\ (ab, k_0 r_0) = 1}}     
 \Big(\frac{\mathfrak{q}^2 |{\bf g}|}{Nab}\Big)^{s}
\zeta(2s) \rho_{q_1'} \rho_{q_2'} \rho_{q_1^{+}} \rho_{q_2^{+}} \rho_{k' r' k_0 r_0}
\rho_{\Phi, k_0 r_0}
 \rho_{\overline{\Phi}, k_0 r_0}
  \gamma_{s}
 \Phi(a \overline{b})
 \frac{ds}{2 \pi i}
 dt_1 dt_2.
\end{multline*}

Shifting the integral far to the right shows that the portion of the sum with $ab \gg \frac{\mathfrak{q}^2 (T')^2 |{\bf g}|}{N} (QkTN)^{\varepsilon}$ is very small.  Note
\begin{equation}
\label{eq:mathfrakqdef}
\mathfrak{q} = 
\frac{q_1' q_1^{+} q_1^{-} r' r_0 r_1 }{q_1^{-}  \sqrt{r'} r_0 r_1}
\frac{q_2' q_2^{+} q_2^{-} r' r_0 r_1}{q_2^{-}  \sqrt{r'} r_0 r_1}
\frac{k' k_0 k_1}{k_0 k_1}
\asymp \frac{Q^2 k}{q_1^{-} q_2^{-} r' r_0^2 r_1^2 k_0 k_1} = \frac{Q^*}{T'},
\end{equation}
and hence
\begin{equation*}
\frac{\mathfrak{q}^2 |{\bf g}| (T')^2}{N} \asymp 
\frac{(Q^*)^2 |{\bf g}|}{N}.
%\frac{Q^4 k^2 Y}{N (q_1^{-} q_2^{-} r' r_0^2 r_1^2 k_0 k_1)^2}.
\end{equation*}
Thus we can truncate the sum at $ab \leq N^*$.
Let $S_{\infty}''({\bf k}, {\bf q})$ denote the contribution to $S_{\infty}'({\bf k}, {\bf q})$ from the terms with $ab \leq N^*$.
 Let $\mathfrak{q} = \mathfrak{q}_1 \mathfrak{q}_2$, where $\mathfrak{q}_i = q_i' q_i^{+} q_i^{-} \sqrt{r' k'}$.

Next we apply Lemma \ref{lemma:detectingprimitivecharacters} to detect the condition that $\theta_1' \overline{\theta_2'}$ is primitive of modulus $k'$, and likewise for $\psi_1 \overline{\psi_2}$ of modulus $r'$. 
We also apply M\"{o}bius inversion to detect $(q_1', q_2') = 1$, as preceding \eqref{eq:SkqgSievedSeparatedVariables}.  Our final arithmetical separation of variables step is to write
\begin{equation*}
\rho_{\Phi, k_0 r_0} = 
\sum_{d_1 | (k_0 r_0)^{\infty}} d^{-s} \Phi_1 \overline{\Phi_2}(d_1),
\end{equation*}
and likewise for $\rho_{\overline{\Phi}, k_0 r_0}$ (indexing the sum with the letter $d_2$).  
We need an archimedean separation of variables as well, and this is provided by Corollary \ref{coro:archimedeanseparationofvariables}.
% 
% , whereby for $T' \gg T^{\varepsilon}$ we have
% \begin{equation}
% \omega_{T'}(t_1 - t_2) \gamma_s(t_1 - t_2) = \intR \eta_{T'}(u) e(u(t_1 -t_2)) du.
% \end{equation}
With this, and rearranging, we then obtain
\begin{multline*}
%\label{eq:SkqgDEF}
S_{\infty}''({\bf k}, {\bf q})
= 
\sum_{\substack{{\bf g}, T', g'  \\ |j_1 - j_2| \leq 1} } \mu({\bf g}) \mu(g')
  \sum_{\ell_1, \ell_2} c_{\ell_1}(k') c_{\ell_2}(r')
  \sum_{d_1, d_2 | (k_0 r_0)^{\infty}}
  \intR \eta_{T'}(u) e(uT'(j_1 - j_2))
  \\
\sum_{\substack{(a,b)=1 \\ ab \leq N^* \\ (ab, k_0 r_0) = 1}}   
  \int_{(1+\varepsilon)}
 \frac{\widetilde{w}(1-s)}{(ab d_1 d_2)^s}
\Big(\frac{N}{|{\bf g}|}\Big)^{1-s}
  \zeta(2s)  
 \rho_{k'r' k_0 r_0} 
 \mathcal{B}_1 \overline{\mathcal{B}_2} \frac{ds}{2 \pi i} du,
\end{multline*}
where
\begin{equation*}
\mathcal{B}_1 = \mathcal{B}_{1,s} =
\int_{U}^{2U}
\sum_{\substack{q_1', \chi_1', \chi_1^{+}, \chi_1^{-}, \psi_1, \theta_1' \\
q_1' \equiv 0 \shortmod{g'}
\\
   (\text{Def } \ref{defi:chiifactorization})
   }} 
    \beta_{1, j_1} \theta_1'(\ell_1) \psi_1(\ell_2) \Phi_1({\bf g} d_1 \overline{d_2})
    \mathfrak{q}_1^{2s-1} \rho_{q_1'} \rho_{q_1^{+}} \Phi_1(a \overline{b}) e(ut_1) dt_1,
\end{equation*}
with $\beta_{1, j_1}$ taking the form $\beta_{*, T-T'/2 + T' j_1 + t_1}$ (i.e., with a linear change of variables as in Corollary \ref{coro:archimedeanseparationofvariables}),
and $\mathcal{B}_2$ is given by a similar definition.

We next shift the contour of integration back to the line $\mathrm{Re}(s) = \varepsilon$, crossing a pole at $s=1/2$ only.  Let $S_{\infty}^{(0)}({\bf k}, {\bf q})$ denote this polar term, and let $S_{\infty}'''({\bf k}, {\bf q})$ be the new integral.  We record the polar term:
\begin{multline}
\label{eq:Sinfty0Def}
S_{\infty}^{(0)}({\bf k}, {\bf q})
= 
\sum_{\substack{{\bf g}, T', g' \\ |j_1 - j_2| \leq 1} } \mu({\bf g}) \mu(g')
  \sum_{\ell_1, \ell_2} c_{\ell_1}(k') c_{\ell_2}(r')
\sum_{d_1, d_2 | (k_0 r_0)^{\infty}}
\frac{\widetilde{w}(1/2)}{\sqrt{d_1 d_2}}
\frac{\phi(k' r' k_0 r_0)}{2 k' r' k_0 r_0}  
 \\
\intR \eta_{T'}(u) e(uT'(j_1 - j_2))
  \sum_{\substack{(a,b)=1 \\ (ab, k_0 r_0) = 1 \\ ab \leq N^*}}   
\Big(\frac{N}{ab|{\bf g}|}\Big)^{1/2}
 \mathcal{B}_1^{(0)} \overline{\mathcal{B}_2^{(0)}} du,
\end{multline}
where $\mathcal{B}_i^{(0)} = \mathcal{B}_i \vert_{s=1/2}$ is given by
\begin{equation}
\label{eq:Bi0def}
\mathcal{B}_i^{(0)} = \int_{t_i}
\sum_{\substack{q_i',
   \chi_i', \chi_i^{+}, \chi_i^{-}, \psi_i, \theta_i'
   \\
   q_i' \equiv 0 \shortmod{g'} \\
   (\text{Def } \ref{defi:chiifactorization})
   }} 
   \beta_{i, j_i} 
   \frac{\phi(q_i' q_i^{+})}{q_i' q_i^{+}}
   \theta_i'(\ell_1) \psi_i(\ell_2)
   \Phi_i({\bf g} d_1 \overline{d_2}  a \overline{b}) e(ut_i) dt_i.
\end{equation}

Now we turn to $S_{\infty}'''({\bf k}, {\bf q})$.  
By the triangle inequality, 
and using \eqref{eq:etabound} to bound the $L^1$ norm of $\eta_{T'}$, 
we obtain
\begin{equation}
\label{eq:Sinfty'''bound}
 |S_{\infty}'''({\bf k}, {\bf q})| \lesssim 
 \sum_{\substack{{\bf g}, T', g' \\ |j_1 - j_2| \leq 1} }
\frac{N}{|{\bf g}| Q^*} 
\max_{\substack{\mathrm{Re}(s) = \varepsilon \\ u \in \mr \\ \ell_1, \ell_2}}
\sum_{\substack{(a,b)=1 \\ ab \leq N^*}} |\mathfrak{q}_1^{-2s+1} \mathcal{B}_{1,s}| \thinspace |\mathfrak{q}_2^{-2s+1} \mathcal{B}_{2,s}|.
\end{equation}

Analogously to \eqref{eq:AiBound}, on the line $\mathrm{Re}(s) = \varepsilon$, we obtain the bound
\begin{equation}
\label{eq:mathcalBisBound}
 \sum_{\substack{(a,b)=1 \\ ab \leq N^*}} |\mathfrak{q}_i^{-2s+1} \mathcal{B}_{i,s}|^2
 \lesssim
\overline{\Delta}\Big(\frac{Q}{q_i^{+} q_i^{-} r' r_0}, q_i^{+} q_i^{-} r' k', 2T', N^* \Big) | \beta_{i, j_i}|^2.
\end{equation}
We note that $\sum_{j_1} |\beta_{1,j_1}|^2 = |\beta_1|^2$, since this simply re-assembles the integral to all of $[T/2, T]$
(also, for each $j_1$, the number of $j_2$ with $|j_1 - j_2| \leq 1$ is at most $3$).
Applying \eqref{eq:mathcalBisBound} to \eqref{eq:Sinfty'''bound} via
Cauchy's inequality and using \eqref{eq:betainorm} (and the previous sentence to handle the sum over the $j_i$)
completes the proof of Proposition \ref{prop:SleqYbound}.
\end{proof}

\subsection{Conclusion}
Now we use
Propositions \ref{prop:SgrtYstatement} and \ref{prop:SleqYbound}
to prove Theorem \ref{thm:recursivethmFE}.
We have a decomposition
\begin{equation}
 S({\bf k}, {\bf q}) = 
 S_{\infty}^{\text{diag}}({\bf k}, {\bf q})
 +
 S_{\infty}'({\bf k}, {\bf q})
 -
 S_{\leq Y}'({\bf k}, {\bf q})
+ (S_{\infty}^{(0)}({\bf k}, {\bf q})-S_{\leq Y}^{(0)}({\bf k}, {\bf q})) + \mathcal{E}_{\infty}.
 \end{equation}
The diagonal term is acceptable for Theorem \ref{thm:recursivethmFE}, as is the small error term $\mathcal{E}_{\infty}$.  

Next we turn to the terms $S_{*}'({\bf k}, {\bf q})$, where $*$ refers to $\leq Y$ or $\infty$.  We choose
\begin{equation}
Y = (QkTN)^{\varepsilon} \frac{Q^4 k^2 T^2}{N},
\end{equation}
with the same value of $\varepsilon$ as in the definition of $N^*$ (see \eqref{eq:N*def}).
First consider $S_{\leq Y}'$, where Cauchy's inequality implies
\begin{equation*}
\sum_{{\bf k}, {\bf q}} |S_{\leq Y}'({\bf k}, {\bf q})|
\lesssim
\prod_{i=1}^{2} 
\Big(
\sum_{{\bf k}, {\bf q}} 
 \overline{\Delta}\Big(\frac{Q}{q_i^{+} q_i^{-} r' r_0 r_1}, q_i^{+} q_i^{-} r' k', T, Y\Big) 
 | \beta_i|^2 \Big)^{1/2}.
\end{equation*}
Recall from \eqref{eq:betainorm} that $\sum_{{\bf k}, {\bf q}} |\beta_i|^2 \ll (kQ)^{\varepsilon} | \beta|^2$.  Hence
\begin{equation*}
\sum_{{\bf k}, {\bf q}} |S_{\leq Y}'({\bf k}, {\bf q})|
\lesssim
\prod_{i=1}^{2} 
\Big(
\max_{{\bf k}, {\bf q}} 
 \overline{\Delta}\Big(\frac{Q}{q_i^{+} q_i^{-} r' r_0 r_1}, q_i^{+} q_i^{-} r' k',T, Y\Big) 
 \Big)^{1/2}
|\beta|^2 
 .
\end{equation*}
Recalling the definition \eqref{eq:Delta''def}, it is easy to see that
\begin{equation*}
\max_{{\bf k}, {\bf q}} \overline{\Delta}\Big(\frac{Q}{q_i^{+} q_i^{-} r' r_0 r_1}, q_i^{+} q_i^{-} r' k',T, Y\Big) 
 \leq \overline{\Delta'}(Q, k, T, Y).
\end{equation*}
In summary, we have shown
\begin{equation*}
\sum_{{\bf k}, {\bf q}} |S_{\leq Y}'({\bf k}, {\bf q})|
\lesssim 
\overline{\Delta'}\Big(Q, k, T, \frac{Q^4 k^2 T^2}{N}\Big) |\beta|^2,
\end{equation*}
which is consistent with Theorem \ref{thm:recursivethmFE}.

The case of $S_{\infty}'$ is fairly similar to that of $S_{\leq Y}'$, though the details are more complicated.  Following similar steps as the case of $S_{\leq Y}'$, and using the AM-GM inequality, we derive
\begin{equation*}
\sum_{{\bf k}, {\bf q}} |S_{\infty }'({\bf k}, {\bf q})|
\lesssim 
|\beta|^2 \max_{{\bf k}, {\bf q}, {\bf g}, T'} \frac{N}{Q^* |{\bf g}|} 
\overline{\Delta}\Big(\frac{Q}{q_1^{+} q_1^{-} r' r_0 r_1}, q_1^{+} q_1^{-} r' k', T', N^* \Big),
\end{equation*}
plus a similar term with the $i=2$ variables ($q_2^{+}$, $q_2^{-}$, etc.).  By symmetry, this latter term will give the same bound as the displayed one.
Substituting the values of $Q^*$ and $N^*$ from \eqref{eq:N*def}, we obtain
\begin{multline}
\label{eq:Sinfty'kqBoundwrtDeltabar}
\sum_{{\bf k}, {\bf q}} |S_{\infty }'({\bf k}, {\bf q})|
\lesssim 
\frac{N}{Q^2 k T}
|\beta|^2 
\\
\times \max_{{\bf k}, {\bf q}, {\bf g}, T'} 
 \frac{q_1^{-} q_2^{-} r' r_0^2 r_1^2 k_0 k_1 T}{|{\bf g}| T'}
\overline{\Delta}\Big(\frac{Q}{q_1^{+} q_1^{-} r' r_0 r_1}, q_1^{+} q_1^{-} r' k', T', \frac{Q^4 k^2 (T')^2|{\bf g}| (QkN)^{\varepsilon}}{N (q_1^{-} q_2^{-} r' r_0^2 r_1^2 k_0 k_1)^2} \Big).
\end{multline}
A bit of checking, recalling $q_2^{-} \leq q_1^{+}$, shows this is consistent with Theorem \ref{thm:recursivethmFE}.

Finally, we consider the polar terms from $s=1/2$, namely $S_{\infty}^{(0)}({\bf k}, {\bf q})-S_{\leq Y}^{(0)}({\bf k}, {\bf q})$.  We need to show there is substantial cancellation between these two terms.  To aid in this, we first simplify $S_{\infty}^{(0)}({\bf k}, {\bf q})$ which recall is defined in \eqref{eq:Sinfty0Def}.
Observe that 
\begin{equation}
N^* = Y \frac{|{\bf g}| (T')^2}{(q_1^{-} q_2^{-} r' r_0^2 r_1^2 k_0 k_1)^2 T^2},
\end{equation}
and since $|{\bf g}|$ divides $k_0^2 r_0^2$ (recalling \eqref{eq:gsummationcondition}), then $N^* \leq Y$.  Then in the definition of $S_{\infty}^{(0)}$, we extend the sum over $ab \leq N^*$ to $ab \leq Y$, and subtract back the terms between $N^*$ and $Y$.  Write $S_{\infty, Y}^{(0)}$ for the terms with $ab \leq Y$, and let $S_{\infty, Y^*}^{(0)} = S_{\infty, Y}^{(0)} - S_{\infty}^{(0)}$ (which represents the terms with $N^* < ab \leq Y$). We claim that $S_{\infty, Y}^{(0)} = S_{\leq Y}^{(0)}$.  To see this, we sum over ${\bf g}$ and $d_1$ and $d_2$ in \eqref{eq:Sinfty0Def} (though modified to read $ab \leq Y$ in place of $ab \leq N^*$).  The sum over ${\bf g}$ is not constrained, and we have
\begin{equation*}
\sum_{\bf g} \frac{\mu({\bf g}) \Phi({\bf g})}{\sqrt{|{\bf g}|}}
= \prod_{p | k_0 r_0} 
\Big(1 - \frac{\Phi(p)}{\sqrt{p}}\Big)
\Big(1 - \frac{\overline{\Phi}(p)}{\sqrt{p}}\Big).
\end{equation*}
For $d_1$ and $d_2$, we have
\begin{equation*}
\sum_{d_1, d_2 | (k_0 r_0)^{\infty}} \frac{\Phi(d_1 \overline{d_2})}{\sqrt{d_1 d_2}}
= \prod_{p | k_0 r_0} \Big(1 - \frac{\Phi(p)}{\sqrt{p}}\Big)^{-1}
\Big(1 - \frac{\overline{\Phi}(p)}{\sqrt{p}}\Big)^{-1}.
\end{equation*}
Therefore, these two evaluations perfectly cancel.  
The sums over $j_1$ and $j_2$ can be simplified by using Lemma \ref{lemma:archimedeanCoset} in the reverse order.
Moreover, since $\gamma_s(t_1 - t_2) = 1$ at $s=1/2$, we can write 
$\sum_{T'} \omega_{T'}(t_1 - t_2) = 1$.  
Hence, the partition of unity is fully re-assembled.

Comparing \eqref{eq:Ai0def} and \eqref{eq:Bi0def}, it is not hard to see that $\mathcal{B}_i^{(0)}$ agrees with $\mathcal{A}_i^{(0)}$ after removal of $\Phi_i({\bf g} d_1 \overline{d_2}) e(u t_i)$.  This shows the claim that $S_{\infty, Y}^{(0)} = S_{\leq Y}^{(0)}$.  Hence $S_{\infty}^{(0)} - S_{\leq Y}^{(0)} = - S_{\infty, Y^*}^{(0)}$, which for ease of reference we write directly as follows:
\begin{multline*}
S_{\infty, Y^*}^{(0)}({\bf k}, {\bf q})
= 
\sum_{\substack{{\bf g}, T',  g' \\ |j_1 - j_2| \leq 1} } \mu({\bf g}) \mu(g')
\sum_{\ell_1, \ell_2} c_{\ell_1}(k') c_{\ell_2}(r')
\sum_{d_1, d_2 | (k_0 r_0)^{\infty}}
\frac{\widetilde{w}(1/2)}{\sqrt{d_1 d_2}}
 \frac{\phi(k' r' k_0 r_0)}{2 k' r' k_0 r_0}
\\
\intR \eta_{T'}(u) e(u T'(j_1 - j_2))
\sum_{\substack{(a,b) = 1 \\ N^* < ab \leq Y \\ (ab, k_0 r_0) = 1} }
\Big(\frac{N}{ab  |{\bf g}|}\Big)^{1/2} 
\mathcal{B}_1^{(0)} \overline{\mathcal{B}_2^{(0)}} du.
\end{multline*}
Now the estimations are similar to those of $S_{\leq Y}'$ and $S_{\infty}'$, though the details are a little different.  Following the same initial steps as in $S_{\leq Y}'$, we obtain
\begin{equation}
\label{eq:SinftyYstarBound}
\sum_{{\bf k}, {\bf q}} |S_{\infty, Y^*}^{(0)}({\bf k}, {\bf q})|
\lesssim  |\beta|^2 
\max_{{\bf k}, {\bf q}, {\bf g}, T'}
\max_{N^* \ll M \ll Y} \frac{N^{1/2}}{(|{\bf g}| M)^{1/2}} 
\Delta\Big(\frac{Q}{q_1^{+} q_1^{-} r' r_0 r_1}, q_1^{+} q_1^{-} r' k', T',M\Big).
\end{equation}
We claim this is bounded consistently with Theorem \ref{thm:recursivethmFE}.  To see this, first note $\frac{N^{1/2}}{Y^{1/2}} \leq \frac{N}{Q^2 kT}$.  Then the condition ``$XR^2 \ell U \leq Q^2 kT$" from \eqref{eq:Delta''def} is deduced from
\begin{equation*}
\frac{Q^2 k T}{N} \frac{N^{1/2}}{(|{\bf g}| M)^{1/2}} \Big(\frac{Q^2 k' T'}{q_1^{+} q_1^{-} r' r_0^2 r_1^2} \Big)
\leq \frac{Q^2 k T}{|{\bf g}|} \leq Q^2 kT.
\end{equation*}
The condition ``$X \leq C$" from \eqref{eq:Delta''def} is easy to check, by setting $M = Y/C$.  This completes the proof of Theorem \ref{thm:recursivethmFE}.

\section{Proof of Theorem \ref{thm:familyavgthm}}
\label{section:FamilyAvg}
\subsection{Miscellany}
Here we present a couple tools with self-contained proofs.
\begin{mylemma}
\label{lemma:Zproperties}
Let $c,d$ be positive integers, and define the Dirichlet series
\begin{equation}
Z_{c,d}(s) = \sum_{\substack{(n, cd) = 1 \\ m|c^{\infty}}} \frac{n}{\varphi(n)} \frac{1}{(mn)^s}, \qquad \mathrm{Re}(s) >1.
\end{equation}
Then
\begin{equation}
\label{eq:Zformula1}
 Z_{c,d}(s) = Z_{1,1}(s) \nu_c(s) \delta_d(s),
\end{equation}
where $Z_{1,1}(s)$ has meromorphic continuation to $\mathrm{Re}(s) > 0$ with a simple pole at $s=1$ only, and where
\begin{equation*}
\nu_c(s) = \prod_{p |c } \Big(1+\frac{p^{-s-1}}{1-p^{-1}}\Big)^{-1},
\qquad
\delta_d(s) = 
\prod_{p|d} \Big(1+(1-p^{-1}) \frac{p^{-s}}{1-p^{-s}}\Big)^{-1}.
\end{equation*}
\end{mylemma}
\begin{proof}
A routine calculation gives
\begin{equation*}
Z_{c,d}(s) =
 \prod_{p | c} (1-p^{-s})^{-1} 
 \prod_{p \nmid cd} \Big(1 + (1-p^{-1})^{-1} \frac{p^{-s}}{1-p^{-s}}\Big),
\end{equation*}
from which the lemma follows with a bit of calculation.
\end{proof}

\begin{mylemma}[Separation of variables]
\label{lemma:separationofvariablesDeterminant}
 Let $\omega = \omega_V$ be a smooth, even, function supported on $[-2V,2V]$, where $V >0$, satisfying $\omega_V^{(j)}(x) \ll V^{-j}$, for all $j=0,1,\dots$.  Let $w(x,y,z,w)$ be smooth of compact support on $\mr_{>0}^4$.  Let $g$ be a Schwartz-class function.  Define $F: \mr_{>0}^4$ by
 \begin{equation*}
  F(x_1, y_1, x_2, y_2) = 
  \omega_V(x_1 y_2 - x_2 y_1)
  g\Big(T \log \frac{x_1 y_2}{x_2 y_1} \Big)
  w\Big(\frac{x_1}{X}, \frac{y_1}{Y}, \frac{x_2}{X}, \frac{y_2}{Y}\Big),
 \end{equation*}
where $T$, $X$, $Y$ are positive parameters.  
Let $R = \frac{V}{XY}$, and set $U = \max(T, R^{-1})$.
Then
\begin{equation*}
 F(x_1, y_1, x_2, y_2) = 
\int_{\mr^4} G(u_1, u_2, u_3, t) \Big(\frac{x_1 y_2}{x_2 y_1}\Big)^{it} 
 \frac{du_1 du_2 du_3}{y_1^{iu_1} y_2^{iu_2} x_2^{iu_3}} dt,
\end{equation*}
where $G$ (depending on $T,V, X,Y$) satisfies the bound for any $A > 0$
\begin{equation}
\label{eq:GboundFamilyAvgSide}
 |G(u_1, u_2, u_3, t)| \ll_A U^{-1} \Big(1+ \frac{|t|}{U}\Big)^{-A} \prod_{i=1}^{3} (1+|u_i|)^{-A}.
\end{equation}
\end{mylemma}
Remark.  If $s \in \mathbb{C}$ and $\omega(x,s) = x^{s-1} \omega_{V}(x)$, then one may apply the lemma to $\omega(x,s)$, giving rise to a family of functions $G = G_s$.  The proof shows that $G_s$ satisfies \eqref{eq:GboundFamilyAvgSide} with an implied constant depending polynomially on $s$.  

\begin{proof}
By Mellin inversion, 
\begin{equation}
\label{eq:FintermsofFtilde}
 F(x_1, y_1, x_2, y_2) 
 = \int \widetilde{F}(s_1, u_1, s_2, u_2) 
 x_1^{-s_1} y_1^{-u_1} x_2^{-s_2} y_2^{-u_2}
 \frac{ds_1 du_1 ds_2 du_2}{(2\pi i)^4},
\end{equation}
where $\widetilde{F}(s_1, u_1, s_2, u_2)$ is defined by
\begin{equation}
\label{eq:FtildeDef}
  \int_{\mr_{>0}^4} \omega_V(x_1 y_2 - x_2 y_1)
  g\Big(T \log \frac{x_1 y_2}{x_2 y_1} \Big)
  w\Big(\frac{x_1}{X}, \frac{y_1}{Y}, \frac{x_2}{X}, \frac{y_2}{Y}\Big) x_1^{s_1} y_1^{u_1} x_2^{s_2} y_2^{u_2} \frac{dx_1 dy_1 dx_2 dy_2}{x_1 y_1 x_2 y_2}.
\end{equation}
In \eqref{eq:FtildeDef}, change variables $x_1 \rightarrow \frac{x_2 y_1}{y_2} x_1$, giving that $\widetilde{F}(s_1, s_2, s_2, s_4)$ equals
\begin{multline*}
  \int_{\mr_{>0}^4} \omega_V \Big(\frac{x_2 y_1}{XY} \frac{(x_1 -1)}{R/V}\Big)
  g(T \log x_1 )
  w\Big(\frac{x_1 x_2 y_1}{X y_2}, \frac{y_1}{Y},  \frac{x_2}{X}, \frac{y_2}{Y}\Big) 
\\  
  x_1^{s_1} y_1^{s_1 + u_1} x_2^{s_1 + s_2} y_2^{-s_1+u_2} \frac{dx_1 dy_1 dx_2 dy_2}{x_1 y_1 x_2 y_2}.
\end{multline*}
Now in \eqref{eq:FintermsofFtilde}, change variables $u_1 \rightarrow u_1 - s_1$, $s_2 \rightarrow s_2 - s_1$, and $u_2 \rightarrow u_2 + s_1$, giving
\begin{equation}
\label{eq:FintermsofFtildeVer2}
 F(x_1, y_1, x_2, y_2) 
 = \int \widetilde{F}(s_1, u_1-s_1, s_2-s_1, u_2+s_1) 
 \Big(\frac{x_1 y_2}{x_2 y_1}\Big)^{-s_1}
 \frac{ds_1 du_1 ds_2 du_2}{y_1^{u_1 } x_2^{s_2} y_2^{u_2} (2 \pi i)^4},
\end{equation}
where now $\widetilde{F}(s_1, u_1-s_1, s_2-s_1, u_2+s_1)$ takes the form of
$\widetilde{H}(s_1, u_1, s_2, u_2)$, where
\begin{equation*}
 H(x_1, y_1, x_2, y_2) = \omega_V\Big(\frac{x_2 y_1}{XY} \frac{(x_1 -1)}{R/V}\Big)
  g(T \log x_1 )
  w\Big(\frac{x_1 x_2 y_1}{X y_2}, \frac{y_1}{Y},  \frac{x_2}{X}, \frac{y_2}{Y}\Big).
\end{equation*}
It is easy to check that
\begin{equation*}
 H^{(j_1, k_1, j_2, k_2)}(x_1, y_1, x_2, y_2) \ll 
 U^{j_1} 
 X^{-j_2} Y^{-k_1 - k_2},
\end{equation*}
and that $x_1$ is concentrated on $x_1 = 1 + O(\min(R, T^{-1}))$,
from whence integration by parts gives
\begin{equation*}
 \widetilde{H}(-it, u_1, u_3, u_2)
 \ll_{A} U^{-1} \Big(1 + \frac{|t|}{U}\Big)^{-A}
 Y^{\mathrm{Re}(u_1 + u_2)} X^{\mathrm{Re}(u_3)}
 \prod_{j=1}^{3} (1+|u_j|)^{-A}.
\end{equation*}
Taking $\mathrm{Re}(u_i) = 0$ and defining $G$ on $\mr^4$ appropriately completes the proof.
\end{proof}

\subsection{Preparation}
It is convenient to work with a couple modified norms that are closely related to \eqref{eq:DeltaDef}. Define
\begin{equation}
\label{eq:Delta1def}
 \Delta_1(Q, k, T, N)
 = \max_{|{\bf \alpha}| = 1} 
 \int_{T/2 \leq t \leq T}
\sum_{\substack{Q/2 < q \leq Q \\ (q,k) = 1}}
\thinspace
\sumstar_{\chi \shortmod{q}}
\thinspace
\sumstar_{\theta \shortmod{k}}
\Big|
\sum_{\substack{N/2 < ab \leq N \\ (a,b) = 1}} \alpha_{a,b} \lambda_{\chi \theta,t}(a,b) 
\Big|^2 dt.
\end{equation}
Clearly, $\Delta_1(Q,k,T,N) \leq \Delta(Q, k, T, N)$, and in the other direction, we have
\begin{equation*}
 \Delta(Q,k,T,N) \leq \sum_{j|k} \Delta_1(Q,j,T,N).
\end{equation*}
Secondly, define
\begin{equation}
\label{eq:Delta2def}
 \Delta_2(Q, k, T, N)
 = \max_{|{\bf \alpha}| = 1} 
 \int_{T/2 \leq t \leq T}
\sum_{\substack{Q/2 < q \leq Q}}
\thinspace
\sumstar_{\psi \shortmod{qk}}
\Big|
\sum_{\substack{N/2 < ab \leq N \\ (a,b) = 1}} \alpha_{a,b} \lambda_{\psi ,t}(a,b) 
\Big|^2 dt.
\end{equation}
It is easy to see that $\Delta_1(Q,k,T,N) \leq \Delta_2(Q,k,T,N)$, since 
when $(q,k) = 1$, the map $(\chi, \theta) \mapsto \chi \theta$ is a bijection onto the set of primitive characters modulo $qk$.  After having done this, we then arrive at \eqref{eq:Delta2def} by dropping the condition $(q,k) = 1$, by positivity.  
For the proof of Theorem \ref{thm:familyavgthm}, we will bound the norm $\Delta_2$.  Indeed, we can deduce Theorem \ref{thm:familyavgthm} from the bound
\begin{equation}
\label{eq:Delta2Bound}
\Delta_2(Q, k, T, N) \lesssim 
Q^2 kT + \frac{Q^2 k T}{N} \overline{\Delta'}\Big(\frac{N}{kQT}, k, T, N\Big)
.
\end{equation}

Let $w$ be a nonnegative smooth weight function with $w(x) \geq 1$ for $1/2 \leq x \leq 1$, and $w(x) = 0$ for $x < 1/4$ and for $x \geq 2$.  Then $\Delta_2(Q,k,T,N) \leq \max_{|\alpha| = 1} S$, where
\begin{equation*}
S = \intR w\Big(\frac{t}{T}\Big) 
\sum_{q} w\Big(\frac{q}{Q}\Big) 
\sumstar_{\psi \shortmod{qk}} \frac{qk}{\varphi(qk)}
\Big|
\sum_{\substack{(a,b) = 1 \\ N/2 < ab \leq N}} \alpha_{a,b} 
\psi(a \overline{b}) (a/b)^{it}
\Big|^2 dt.
\end{equation*}
We will assume that $\alpha_{a,b}$ is supported on 
\begin{equation}
\label{eq:alphasupport}
N/2 < ab \leq N, \qquad (ab,k) = 1, \quad \text{and } \quad (a,b) = 1. 
\end{equation}
A simple argument with a dyadic partition of unity and Cauchy's inequality shows that
\begin{equation*}
 \Big| \sum_{a,b} \alpha_{a,b} \Big|^2 = \Big|\sum_{\substack{N_1 N_2 \asymp N \\ \text{dyadic}}} 
 \sum_{\substack{a \asymp N_1 \\ b \asymp N_2}} \alpha_{a,b} \Big|^2
 \ll \log{N} \sum_{\substack{N_1 N_2 \asymp N \\ \text{dyadic}}} 
 \Big| \sum_{\substack{a \asymp N_1 \\ b \asymp N_2}} \alpha_{a,b} \Big|^2.
\end{equation*}
Hence, in the proof of Theorem \ref{thm:familyavgthm}, we may assume that $a$ and $b$ are each supported in dyadic ranges, say $a \asymp N_1$ and $b \asymp N_2$.  

Let $1 \leq Y \leq \frac{Q}{100}$ be a parameter to be chosen later.  
For $\psi \pmod{qk}$, say $qk = q' (dk)$ where $d|k^{\infty}$ and $(q', k) = 1$, and write $\psi = \psi_{k} \psi'$ where $\psi_k$ has modulus $dk$ and $\psi'$ has modulus $q'$.  Let $m_k(\psi) = dk$  denote the modulus of the $k$-part of $\psi$, and $\cond_{q'}(\psi)$ denote the conductor of $\psi'$, i.e., the coprime to $k$ part of $\psi$.
Then $S \leq S_{>Y}$, where
\begin{equation*}
S_{>Y} = 
 \intR w\Big(\frac{t}{T}\Big) 
\sum_{q} w\Big(\frac{q}{Q}\Big) 
\sum_{\substack{\psi \shortmod{qk} \\ \cond_{q'}(\psi) m_k(\psi) > Yk}} \frac{qk}{\varphi(qk)}
\Big|
\sum_{\substack{a,b}} \alpha_{a,b} 
\psi(a \overline{b}) (a/b)^{it}
\Big|^2 dt,
\end{equation*}
by positivity, since if $\psi$ is primitive modulo $qk$, then $\cond_{q'}(\psi) m_k(\psi) = \cond(\psi) = qk$. 
This uses that the condition $qk>Yk$ is redundant to the support of $w(q/Q)$.

By inclusion-exclusion, we have $S_{>Y} = S_{\leq \infty} - S_{\leq Y}$, where for $* \in \{Y, \infty \}$, $S_{\leq *}$ corresponds to the sum over $\cond_{q'}(\psi) m_k(\psi)/k \leq *$.  We will write $S_{\infty}$ as an alias for $S_{ \leq \infty}$.    

We begin with some arithmetic manipulations that are in common between $S_{\infty}$ and $S_{\leq Y}$.  Opening the square, we have
\begin{multline*}
 S_{\leq *} = 
 \intR w\Big(\frac{t}{T}\Big) 
\sum_{q} w\Big(\frac{q}{Q}\Big) 
\sum_{\substack{\psi \shortmod{qk} \\ \cond_{q'}(\psi) m_k(\psi)/k \leq * }} \frac{qk}{\varphi(qk)}
\\
\sum_{\substack{a_1, b_1 \\ a_2, b_2}}
\alpha_{a_1,b_1} \overline{\alpha_{a_2, b_2}} 
\psi(a_1 b_2 \overline{b_1 a_2}) \Big(\frac{a_1 b_2}{b_1 a_2}\Big)^{it}
 dt.
\end{multline*}
Define
\begin{equation}
\label{eq:gioriginalDef}
 g_1 = (a_1, a_2), \qquad
 g_2 = (b_1, b_2), \qquad
 g_3 = (a_1, b_2), \qquad
 g_4 = (b_1, a_2),
\end{equation}
and note that the $g_i$ are pairwise coprime since $(a_1, b_1) = (a_2, b_2) = 1$ by the support of $\alpha$ (recall \eqref{eq:alphasupport}).  Then change variables
\begin{equation}
\label{eq:aiChangeofVars}
\begin{split}
 &a_1 \rightarrow g_1 g_3 h_{11} h_{13} a_1, \qquad \text{where} \quad (a_1, g_1 g_3) = 1 \\
 &a_2 \rightarrow g_1 g_4 h_{21} h_{24} a_2, \qquad \text{where} \quad (a_2, g_1 g_4) = 1 \\
 &b_1 \rightarrow g_2 g_4 h_{32} h_{34} b_1, 
 \qquad \text{where} \quad (b_1, g_2 g_4) = 1
 \\
 &b_2 \rightarrow g_2 g_3 h_{42} h_{43} b_2,
 \qquad \text{where} \quad (b_2, g_2 g_3) = 1
\end{split}
\end{equation}
and where 
\begin{equation}
 \label{eq:hijproperties}
 h_{ij} | g_j^{\infty} \quad \text{for all } i,j,
 \qquad \text{and} \qquad
 (h_{ij}, h_{kj}) = 1 \quad \text{for } i \neq k.
\end{equation}
The conditions 
\eqref{eq:gioriginalDef} translate into
\begin{equation*}
 (a_1 b_1, a_2 b_2) = 1.
\end{equation*}
Moreover, the conditions $(a_1, g_1 g_3) = 1, \dots, (b_2, g_2 g_3) = 1$ in \eqref{eq:aiChangeofVars} may be expressed succintly as $(a_1 a_2 b_1 b_2, g_1 g_2 g_3 g_4) = 1$, since prior to \eqref{eq:aiChangeofVars} we had $(a_i, b_i) = 1$ from \eqref{eq:alphasupport}.
Let
\begin{equation}
 {\bf g} = (g_1, g_2, g_3, g_4, h_{11}, h_{13}, h_{21}, h_{24}, h_{32}, h_{34}, h_{42}, h_{43}),
\end{equation}
where the $h_{ij}$ satisfy \eqref{eq:hijproperties}. 
In addition, let 
\begin{equation*}
 \beta_{13} = g_1 g_3 h_{11} h_{13}, 
 \qquad
 \beta_{14} = g_1 g_4 h_{21} h_{24}, 
 \qquad
 \beta_{24} = g_2 g_4 h_{32} h_{34},
 \qquad
 \beta_{23} = g_2 g_3 h_{42} h_{43}, 
\end{equation*}
and
\begin{equation*}
\gamma_1 = g_3^2 h_{11} h_{42} h_{13} h_{43} 
	= \frac{\beta_{13} \beta_{23}}{g_1 g_2}
\qquad 
\text{and} 
\qquad 
\gamma_2 =  g_4^2 h_{21} h_{32} h_{24} h_{34}
	  = \frac{\beta_{14} \beta_{24}}{g_1 g_2}.
\end{equation*}
Observe that $(\gamma_1, \gamma_2) = 1$ since the $g_i$ are pairwise coprime, and using \eqref{eq:hijproperties}.
With these substitutions, we obtain
\begin{multline}
\label{eq:Sleq*FinalCommonFormula}
 S_{\leq *} = 
 \sum_{\bf g}
 \intR w\Big(\frac{t}{T}\Big) 
\sum_{(q, g_1 g_2) = 1} w\Big(\frac{q}{Q}\Big) 
\sum_{\substack{\psi \shortmod{qk} \\ \cond_{q'}(\psi) m_k(\psi)/k \leq * }} \frac{qk}{\varphi(qk)}
\\
\sum_{\substack{(a_1 b_1, a_2 b_2) = 1 \\ ({\bf a}, {\bf g}) = 1}}
\alpha_{a_1,b_1}^{(1, {\bf g})} \overline{\alpha}_{a_2, b_2}^{(2, {\bf g})}
\psi\Big(\frac{\gamma_1 a_1 b_2}{\gamma_2 b_1 a_2}\Big) \Big(\frac{\gamma_1 a_1 b_2}{\gamma_2 b_1 a_2}\Big)^{it}
 dt,
\end{multline}
where
\begin{equation} 
\label{eq:alphabfgDef}
\alpha_{a_1,b_1}^{(1, {\bf g})} \overline{\alpha}_{a_2, b_2}^{(2, {\bf g})} =  \alpha_{\beta_{13} a_1, \beta_{24} b_1}
 \overline{\alpha}_{\beta_{14} a_2, \beta_{23} b_2},
\end{equation}
and where the condition $({\bf a}, {\bf g}) = 1$ is shorthand for $(a_1 a_2 b_1 b_2, g_1 g_2 g_3 g_4) = 1$.
There are additional conditions that are implicitly enforced by \eqref{eq:alphasupport}, which we will recall only as needed.
For later use, note
\begin{equation}
\label{eq:gammaiaibjsizes}
 \gamma_1 a_1 b_2 \asymp \gamma_2 a_2 b_1 \asymp \frac{N}{g_1 g_2}.
\end{equation}
Moreover, we claim that 
\begin{equation}
\label{eq:alphagvsalphanogbound}
\sum_{{\bf g}, a_1, b_1} |\alpha_{a_1, b_1}^{(1, \bf g)}|^2 \lesssim |\alpha|^2,
\end{equation}
and similarly for $\alpha_{a_2, b_2}^{(2, {\bf g})}$.  To see this, note that the variables $g_1, g_2, g_3, g_4$ appear as divisors of $\beta_{13}$ or $\beta_{24}$, and similarly for half of the $h_{ij}$ variables (namely, $h_{11}$, $h_{13}$, $h_{32}$, and $h_{34}$).  For the remaining $h_{ij}$ variables, we recall from \eqref{eq:hijproperties} that $h_{12} | g_2$, etc., so these variables range over a set of cardinality $\ll N^{\varepsilon}$.  Then \eqref{eq:alphagvsalphanogbound} follows easily.

\subsection{Direct method}
In this section we estimate $S_{\leq Y}$ by reducing to an instance of the original norm, but with smaller parameters.
\begin{myprop}
\label{prop:SleqYFamilyAvg}
We have $S_{\leq Y} = S_{\leq Y}^{(0)} + S_{\leq Y}'$, where
$S_{\leq Y}^{(0)}$ is given by \eqref{eq:SleqY0formulaFamilyAvg} below,
and where
\begin{equation}
\label{eq:SleqY'boundFamilyAvgSide}
S_{\leq Y}' \lesssim 
\max_{\substack{Y' \leq Y \\ r_k|k^{\infty}}}
\Delta(Y'/r_k, r_k k, 2T, N) |\alpha|^2.
\end{equation}
\end{myprop}
\begin{proof}
We pick up from \eqref{eq:Sleq*FinalCommonFormula}. Write $q=r_k q'$ where $r_k|k^{\infty}$ and $(q',k) = 1$, and write $\psi = \chi \theta$ where $\theta$ runs modulo $r_k k$ and $\chi$ runs modulo $q'$.  Then
\begin{multline}
\label{eq:SleqYv1InProofOfProp}
 S_{\leq Y} = 
 \sum_{\bf g}
 \sum_{r_k|k^{\infty}}
 \intR w\Big(\frac{t}{T}\Big) 
 \sum_{(q', k g_1 g_2) = 1} w\Big(\frac{q'r_k}{Q}\Big) 
\sum_{\theta \shortmod{r_k k}}
\sum_{\substack{\chi \shortmod{q'} \\ \cond(\chi) \leq Y/r_k }} \frac{q'k}{\varphi(q') \varphi(k)}
\\
\sum_{\substack{(a_1 b_1, a_2 b_2) = 1 \\ ({\bf a}, {\bf g}) = 1}}
\alpha_{a_1,b_1}^{(1, {\bf g})} \overline{\alpha}_{a_2, b_2}^{(2, {\bf g})}
 \chi \theta\Big(\frac{\gamma_1 a_1 b_2}{\gamma_2 b_1 a_2}\Big) \Big(\frac{\gamma_1 a_1 b_2}{\gamma_2 b_1 a_2}\Big)^{it}
 dt.
\end{multline}

We next replace $q'$ by 
$q' q_0 q_1$ where $q'$ is the conductor of $\chi$, $(q_0, q') = 1$, and $q_1 | (q')^{\infty}$, and correspondingly write $\chi = \chi' \chi_0$ where $\chi'$ is primitive modulo $q'$, and $\chi_0$ is trivial modulo $q_0$.  Applying this substitution in \eqref{eq:SleqYv1InProofOfProp}, we obtain
\begin{multline*}
 S_{\leq Y} = 
 \sum_{\bf g}
 \sum_{r_k|k^{\infty}}
 \intR w\Big(\frac{t}{T}\Big) 
 \sum_{\substack{(q', k g_1 g_2) = 1 \\ q' \leq Y/r_k}}
\sum_{\theta \shortmod{r_k k}}
\thinspace
\sumstar_{\substack{\chi \shortmod{q'}  }} \frac{q' k}{\varphi(q')  \varphi(k)}
\\
\sum_{\substack{(q_0,q' k {\bf g}) = 1 \\ q_1 | (q')^{\infty}}} w\Big(\frac{q' q_0 q_1 r_k}{Q}\Big) 
\frac{q_0 }{ \varphi(q_0)}
\sum_{\substack{(a_1 b_1, a_2 b_2) = 1 \\ ({\bf a}, q_0 {\bf g}) = 1}}
\alpha_{a_1,b_1}^{(1, {\bf g})} \overline{\alpha}_{a_2, b_2}^{(2, {\bf g})}
\chi' \theta \Big(\frac{\gamma_1 a_1 b_2}{\gamma_2 b_1 a_2}\Big) \Big(\frac{\gamma_1 a_1 b_2}{\gamma_2 b_1 a_2}\Big)^{it}
 dt.
\end{multline*}
By Mellin inversion, and evaluating the sums over $q_0$ and $q_1$ with Lemma \ref{lemma:Zproperties}, the second line above equals
\begin{equation*}
\sum_{\substack{(a_1 b_1, a_2 b_2) = 1 \\ ({\bf a}, {\bf g}) = 1}}
\alpha_{a_1,b_1}^{({\bf g})} \overline{\alpha}_{a_2, b_2}^{(2, {\bf g})}
\chi'  \theta\Big(\frac{\gamma_1 a_1 b_2}{\gamma_2 b_1 a_2}\Big) \Big( \frac{\gamma_1 a_1 b_2}{\gamma_2 b_1 a_2}\Big)^{it}
\frac{1}{2 \pi i} \int_{(2)} 
\Big(\frac{Q}{r_k q'}\Big)^s 
\widetilde{w}(s) 
Z_{q',k{\bf g} {\bf a}}(s) ds dt.
\end{equation*}
Since $k, g_1, g_2, g_3, g_4, a_1, a_2, b_1, b_2$ are pairwise coprime, we have 
\begin{equation}
\label{eq:Zq'kgaformula}
Z_{q', k{\bf g} {\bf a}}(s)
= Z_{1,1}(s) \nu_{q'}(s) \delta_{k {\bf g}}(s) \delta_{a_1 b_1}(s) \delta_{a_2 b_2}(s),
\end{equation}
which is an important separation of variables.  

Using the meromorphic continuation of $Z_{1,1}(s)$ provided by Lemma \ref{lemma:Zproperties}, we shift the contour of integration to the line $\mathrm{Re}(s) = \varepsilon$, passing a pole at $s=1$.  Let $S_{\leq Y}^{(0)}$ denote the residue term, which is given by
\begin{multline}
\label{eq:SleqY0formulaFamilyAvg}
S_{\leq Y}^{(0)} =  
\sum_{r_k|k^{\infty}}
\sum_{\bf g} \delta_{k {\bf g}}
 \intR w\Big(\frac{t}{T}\Big) 
\sum_{\theta \shortmod{r_k k}} 
\sum_{\substack{(q',kg_1 g_2) = 1 \\ q' \leq Y/r_k}} 
\thinspace
\sumstar_{\substack{\chi' \shortmod{q'}}} \frac{Q \widetilde{w}(1) Z_{1,1} \nu_{q'} k}{r_k \varphi(q') \varphi(k)}
\\
\sum_{\substack{(a_1 b_1, a_2 b_2) = 1 \\ ({\bf a}, {\bf g}  ) = 1}}
\delta_{a_1 b_1} 
\alpha_{a_1,b_1}^{(1, {\bf g})} 
\delta_{a_2 b_2}
\overline{\alpha}_{a_2, b_2}^{(2, {\bf g})}
\chi'  \theta\Big(\frac{\gamma_1 a_1 b_2}{\gamma_2 b_1 a_2}\Big) \Big(\frac{\gamma_1 a_1 b_2}{\gamma_2 b_1 a_2}\Big)^{it}
   dt,
\end{multline}
where $Z_{1,1}$ denotes $\mathrm{Res}_{s=1} Z_{1,1}(s)$, $\nu_{q'}$ denotes $\nu_{q'}(1)$, and $\delta_{n} = \delta_{n}(1)$.

By the triangle inequality, and some simple bounds, we have
\begin{multline}
\label{eq:SleqY'boundV1FamilyAvgSide}
|S_{\leq Y}'| \lesssim
\sum_{r_k|k^{\infty}} r_k^{-\varepsilon}
\max_{\mathrm{Re}(s) = \varepsilon}
\sum_{\bf g}
\int_{-2T}^{2T} 
\sum_{\theta \shortmod{r_k k}} 
\sum_{\substack{(q',kg_1 g_2) = 1 \\ q' \leq Y/r_k}} 
\thinspace
\sumstar_{\substack{\chi' \shortmod{q'}}} 
\\
\Big|
\sum_{\substack{(a_1 b_1, a_2 b_2) = 1 \\ ({\bf a}, {\bf g}) = 1}}
\delta_{a_1 b_1}(s)
\alpha_{a_1,b_1}^{(1, {\bf g})} 
\delta_{a_2 b_2}(s)
\overline{\alpha}_{a_2, b_2}^{(2, {\bf g})}
\chi'  \theta\Big(\frac{\gamma_1 a_1 b_2}{\gamma_2 b_1 a_2}\Big) \Big( \frac{\gamma_1 a_1 b_2}{\gamma_2 b_1 a_2}\Big)^{it}
 \Big| dt.
\end{multline}
Note $|\chi' \theta(\gamma_1 \overline{\gamma_2}) (\gamma_1/\gamma_2)^{it} | \leq 1$, which may be used to simplify this bound.  To show the desired bound \eqref{eq:SleqY'boundFamilyAvgSide}, we state and prove Lemma \ref{lemma:FamilyAvgBilinearBoundViaMobius} just below, as it will be useful later as well.
\end{proof}

\begin{mylemma}
 \label{lemma:FamilyAvgBilinearBoundViaMobius}
Let $\gamma_{a,b}^{(1)}$ and $\gamma_{a,b}^{(2)}$ be sequences of complex numbers supported on $ab \asymp M$, $(a,b) = 1$.  
 Consider an expression of the form $\mathcal{S}_{\gamma}(Q, k, T, M)$ defined by
 \begin{equation*}
\int_{-T}^{T} 
\sum_{\theta \shortmod{k}} 
\sum_{\substack{(q,k) = 1 \\Q/2 < q \leq Q}} 
\thinspace
\sumstar_{\substack{\chi \shortmod{q}}} 
\Big|
\sum_{\substack{(a_1 b_1, a_2 b_2) = 1}}
\gamma_{a_1,b_1}^{(1)}
\overline{\gamma}_{a_2, b_2}^{(2)}
\chi  \theta\Big(\frac{a_1 b_2}{b_1 a_2}\Big) \Big( \frac{a_1 b_2}{ b_1 a_2}\Big)^{it}
 \Big| dt.
\end{equation*} 
 Then
 \begin{equation*}
  \mathcal{S}_{\gamma}(Q,k,T,M)  \lesssim  \overline{\Delta}(Q,k,T,M)
  \max_{i=1,2} |\gamma^{(i)}|^2.
 \end{equation*}
\end{mylemma}

\begin{proof}
To separate the inner variables, we use M\"{o}bius inversion in the form
\begin{equation}
\label{eq:SomeRandomMobiusStepFamilyAvg}
\delta((a_1 b_1, a_2 b_2) = 1) = 
\sum_{e_1 | (a_1, a_2)}
\sum_{e_2 | (a_1, b_2)}
\sum_{e_3 | (b_1, a_2)}
\sum_{e_4 | (b_1, b_2)} 
\mu(e_1) 
\mu(e_2) 
\mu(e_3) 
\mu(e_4). 
\end{equation}
The $e_i$ are pairwise coprime, by the support of $\gamma$.  Thus
\begin{equation*}
\mathcal{S}_{\gamma}(Q,k,T,M) \ll \sum_{e_1, e_2, e_3, e_4}
\int_{-T}^{T}  
\sum_{\theta \shortmod{k}} 
\sum_{\substack{(q,k) = 1 \\Q/2 < q \leq Q}} 
\thinspace
\sumstar_{\substack{\chi \shortmod{q}}} 
|\mathcal{A}_1 \mathcal{A}_2| dt,
\end{equation*}
where
\begin{equation*}
 \mathcal{A}_1 = \sum_{\substack{ a_1 \equiv 0 \shortmod{e_1 e_2} \\ b_1 \equiv 0 \shortmod{e_3 e_4}}}
\gamma_{a_1,b_1} 
\chi \theta(a_1 \overline{b_1}) \Big(\frac{a_1}{b_1}\Big)^{it},
\end{equation*}
and $\mathcal{A}_2$ has a similar definition.
Lemma \ref{lemma:FamilyAvgBilinearBoundViaMobius} follows 
by using $|\mathcal{A}_1 \mathcal{A}_2| \ll |\mathcal{A}_1|^2 + |\mathcal{A}_2|^2$ and monotonicity (Lemma \ref{lemma:mononicityNaspect}).
\end{proof}

\subsection{Divisor switching method}
\begin{myprop}
\label{prop:SinftyFamilyAvg}
We have a decomposition
\begin{equation*}
S_{\infty} = S_{\infty}^{(0)} + S_{\infty}' + S_{\infty}^{\text{diag}} + \mathcal{E}_{\infty},
\end{equation*}
with the following properties.  The term $S_{\infty}^{(0)}$ is given by
\eqref{eq:Sinfty0FamilyAvgSide} below,
and $S_{\infty}'$ satisfies the bound
\begin{equation}
\label{eq:Sinfty'boundFamilyAvgSide}
|S_{\infty}'| \lesssim \frac{Q^2 k T}{N}
\overline{\Delta'}\Big(\frac{N}{kQT}, k, T, N\Big) 
|\alpha|^2.
\end{equation}
The diagonal term satisfies the bound
\begin{equation}
\label{eq:SinftydiagboundFamilyAvgCase}
|S_{\infty}^{\text{diag}}| \ll Q^2 k T |\alpha|^2,
\end{equation}
and the term $\mathcal{E}_{\infty}$ is negligibly small.
\end{myprop}

\begin{proof}
We carry on with \eqref{eq:Sleq*FinalCommonFormula} and apply orthogonality of characters to the sum over $\psi$. This picks out the congruence
$\gamma_1 a_1 b_2 \equiv \gamma_2 a_2 b_1 \pmod{kq}$, however with a side condition $(\gamma_1 \gamma_2 a_1 a_2 b_1 b_2, kq) = 1$.  This side condition can be dropped, since the congruence $\gamma_1 a_1 b_2 \equiv \gamma_2 a_2 b_1 \pmod{kq}$, combined with $(\gamma_1 a_1 b_2, \gamma_2 a_2 b_1) = 1$, implies that $(\gamma_1 \gamma_2 a_1 b_2 a_2 b_1, kq) = 1$.
Additionally evaluating the $t$-integral, in all we obtain
\begin{equation*}
 S_{\infty} = Q k T
 \sum_{\bf g}
\sum_{(q,g_1 g_2) = 1} w_1\Big(\frac{q}{Q}\Big) 
 \\
\sum_{\substack{(a_1 b_1, a_2 b_2) = ({\bf a}, {\bf g}) = 1
\\
\gamma_1 a_1 b_2 \equiv \gamma_2 a_2 b_1 \shortmod{kq}
}}
\alpha_{a_1,b_1}^{(1, {\bf g})} \overline{\alpha}_{a_2, b_2}^{(2, {\bf g})}
 \widehat{w}\Big(T \log \frac{\gamma_1 a_1 b_2}{\gamma_2 a_2 b_1} \Big),
\end{equation*}
where $w_1(x) = x w(x)$ and $\widehat{w}(x) = \intR w(t) e^{ixt} dt$.

Let $S_{\infty}^{\text{diag}}$ be the contribution to $S_{\infty}$ from the diagonal $\gamma_1 a_1 b_2 = \gamma_2 a_2 b_1$.  Since $(\gamma_1 a_1 b_2, \gamma_2 a_2 b_1) = 1$, this forces $\gamma_i = a_i = b_i = 1$ for $i=1,2$.  Hence recalling \eqref{eq:alphabfgDef}, we obtain
\begin{equation}
\label{eq:SinftyDiagBound}
S_{\infty}^{\text{diag}} \ll Q^2 k T \sum_{g_1, g_2} |\alpha_{g_1, g_2}|^2 = Q^2 k T |\alpha|^2.
\end{equation}

Let $S_{\infty}'' = S_{\infty} - S_{\infty}^{\text{diag}}$ be the non-diagonal portion of $S_{\infty}$.  Write $\gamma_1 a_1 b_2 = \gamma_2 a_2 b_1 + qkr$, where $r \neq 0$.  Additionally, we detect the condition $(q,  g_1 g_2) = 1$ by M\"obius inversion in the form
%\begin{equation}
 $\sum_{\substack{d|(q, g_1 g_2)}} \mu(d)$,
%\end{equation}
and substitute $q=de$.  This gives
\begin{multline*}
 S_{\infty}'' = Q k T
 \sum_{\bf g}
 \sum_{d| g_1 g_2  } \mu(d)
 \sum_{e}
 w_1\Big(\frac{de}{Q}\Big) 
 \sum_{\substack{r \in \mz \setminus \{0\} }}
 \\
\sum_{\substack{(a_1 b_1, a_2 b_2) =  ({\bf a}, {\bf g}) = 1
\\
\gamma_1 a_1 b_2 - \gamma_2 a_2 b_1 = dekr
}}
\alpha_{a_1,b_1}^{(1, {\bf g})} \overline{\alpha}_{a_2, b_2}^{(2, {\bf g})}
 \widehat{w}\Big(T \log \frac{\gamma_1 a_1 b_2}{\gamma_2 a_2 b_1} \Big).
\end{multline*}

Now we perform the divisor switch: re-write $\gamma_1 a_1 b_2 - \gamma_2 a_2 b_1 = dekr$ as
\begin{equation}
\label{eq:divisorswitchedcongruence}
 \gamma_1 a_1 b_2 \equiv \gamma_2 a_2 b_1 \pmod{dk|r|}, 
 \qquad e = \frac{\gamma_1 a_1 b_2 - \gamma_2 a_2 b_1}{dk r}.
\end{equation}
It is convenient to record that the side condition 
\begin{equation}
 \label{eq:sideconditionCoprimality}
 (\gamma_1 \gamma_2 a_1 a_2 b_1 b_2, dk r) = 1
\end{equation}
follows from the congruence \eqref{eq:divisorswitchedcongruence} together with the coprimality $(\gamma_1 a_1 b_2, \gamma_2 a_2 b_1) = 1$.  We also factor $r$ as
\begin{equation*}
 r = r_0 r_1, \qquad r_0 | (k g_1 g_2)^{\infty}, \qquad (r_1, k g_1 g_2) = 1.
\end{equation*}
With these substitutions, we obtain
\begin{multline*}
 S_{\infty}'' = Q k T
 \sum_{\bf g}
 \sum_{\substack{ d|g_1 g_2 \\ r_0 | (k g_1 g_2)^{\infty}}} \mu(d)
 \sum_{\substack{r_1 \in \mz \setminus \{0\}  \\ (r_1, k g_1 g_2) = 1}}
 w_1\Big(\frac{\gamma_1 a_1 b_2 - \gamma_2 a_2 b_1}{kr_0 r_1 Q}\Big) 
  \\
\sum_{\substack{(a_1 b_1, a_2 b_2) = ({\bf a}, {\bf g}) = 1
\\
\gamma_1 a_1 b_2 \equiv \gamma_2 a_2 b_1 \shortmod{dkr_0} \\
\gamma_1 a_1 b_2 \equiv \gamma_2 a_2 b_1 \shortmod{|r_1|}
}}
\alpha_{a_1,b_1}^{(1, {\bf g})} \overline{\alpha}_{a_2, b_2}^{(2, {\bf g})}
 \widehat{w}\Big(T \log \frac{\gamma_1 a_1 b_2}{\gamma_2 a_2 b_1} \Big).
\end{multline*}

Next we express the congruences using Dirichlet characters modulo $dkr_0$ and $|r_1|$; this is enabled by the side condition \eqref{eq:sideconditionCoprimality}.  This leads to
\begin{multline*}
 S_{\infty}'' = Q k T
 \sum_{\bf g}
  \sum_{\substack{ d|g_1 g_2 \\ r_0 | (k g_1 g_2)^{\infty}}} 
 \frac{\mu(d)}{\varphi(dk r_0)} \sum_{\theta \shortmod{dkr_0}}
 \sum_{\substack{r_1 \in \mz \setminus \{0\}  \\ (r_1, k g_1 g_2) = 1}}
\frac{1}{\varphi(|r_1|)} \sum_{\chi \shortmod{|r_1|}}
\\
  \sum_{\substack{(a_1 b_1, a_2 b_2) = 1 \\ ({\bf a}, {\bf g}) = 1
}}
\alpha_{a_1,b_1}^{(1, {\bf g})} \overline{\alpha}_{a_2, b_2}^{(2, {\bf g})}
 \chi \theta \Big(\frac{\gamma_1 a_1 b_2}{\gamma_2 a_2 b_1} \Big)
 w_1\Big(\frac{\gamma_1 a_1 b_2 - \gamma_2 a_2 b_1}{kr_0 r_1 Q}\Big) 
 \widehat{w}\Big(T \log \frac{\gamma_1 a_1 b_2}{\gamma_2 a_2 b_1} \Big).
\end{multline*}
The characters of varying modulus need to be primitive, so we substitute
\begin{equation*}
 r_1 \rightarrow r_1 r_2 q', 
 \qquad
 \chi \rightarrow \chi_0 \chi,
\end{equation*}
where $r_1 | (q')^{\infty}$, $(r_2, q') = 1$,
$\chi$ is primitive of modulus $|q'|$, and $\chi_0$ is trivial modulo $r_2$.  With this, we obtain
\begin{multline*}
 S_{\infty}'' = Q k T
 \sum_{\bf g}
 \sum_{\substack{ d|g_1 g_2 \\ r_0 | (k g_1 g_2)^{\infty}}} 
 \frac{\mu(d)}{\varphi(dk r_0)} 
 \sum_{\theta \shortmod{dkr_0}}
 \sum_{\substack{q' \neq 0 \\ (q', k g_1 g_2) = 1}}
 \sum_{\substack{r_1 | (q')^{\infty} \\ (r_2, q' k {\bf g}) = 1}} 
  \\
\sumstar_{\chi \shortmod{|q'|}}
  \sum_{\substack{(a_1 b_1, a_2 b_2) = 1 \\ (a_1 a_2 b_1 b_2, {\bf g} r_2) = 1
}}
\alpha_{a_1,b_1}^{(1, {\bf g})} \overline{\alpha}_{a_2, b_2}^{(2, {\bf g})}
 \chi \theta\Big(\frac{\gamma_1 a_1 b_2}{\gamma_2 a_2 b_1} \Big)
 \frac{w_1\Big(\frac{\gamma_1 a_1 b_2 - \gamma_2 a_2 b_1}{kr_0 r_1 r_2 q' Q}\Big) }{\varphi(r_1 r_2 |q'|)} 
 \widehat{w}\Big(T \log \frac{\gamma_1 a_1 b_2}{\gamma_2 a_2 b_1} \Big).
\end{multline*}
Let $w_1(x) = x^{-1} w_2(x)$, so $w_2(x) = x^2 w(x)$, and $\widetilde{w_2}(-s) = \widetilde{w}(2-s)$.  In addition, apply the Mellin inversion formula to $w_2$.  Then we obtain that $S_{\infty}''$ equals
\begin{multline*}
 Q^2 k T \sum_{\bf g}
 \sum_{\substack{ d|g_1 g_2 \\ r_0 | (k g_1 g_2)^{\infty}}} 
 \frac{\mu(d) k r_0}{\varphi(dk r_0)} \sum_{\theta \shortmod{dkr_0}}
 \sum_{\substack{ (q', k g_1 g_2) = 1}}
 \sum_{\substack{r_1 | (q')^{\infty} \\ (r_2, q' k {\bf g}) = 1}} 
\frac{   r_2 |q'|  }{\varphi(r_2) \varphi(|q'|)} 
\sumstar_{\chi \shortmod{|q'|}}
  \int_{(2)} \widetilde{w}(2-s)
  \\
    \sum_{\substack{(a_1 b_1, a_2 b_2) = 1 \\ (a_1 a_2 b_1 b_2, {\bf g} r_2) = 1
}}
\Big(\frac{\gamma_1 a_1 b_2 - \gamma_2 a_2 b_1}{kr_0 r_1 r_2 q' Q} \Big)^{s}
\frac{(\sgn) \alpha_{a_1,b_1}^{(1, {\bf g})} \overline{\alpha}_{a_2, b_2}^{(2, {\bf g})}}{|\gamma_1 a_1 b_2 - \gamma_2 a_2 b_1|}
 \chi \theta\Big(\frac{\gamma_1 a_1 b_2}{\gamma_2 a_2 b_1} \Big)
  \widehat{w}\Big(T \log \frac{\gamma_1 a_1 b_2}{\gamma_2 a_2 b_1} \Big) \frac{ds}{2 \pi i},
\end{multline*}
where the summand $(\sgn)$ is shorthand for the indicator function that
\begin{equation}
\label{eq:sgncondition}
\sgn(q') = \sgn(\gamma_1 a_1 b_2 - \gamma_2 a_2 b_1).
\end{equation}
Prior to the Mellin inversion formula, \eqref{eq:sgncondition} was enforced by the support of $w_2$.

The sums over $r_1$ and $r_2$ evaluate exactly as in \eqref{eq:Zq'kgaformula}.
Thus
\begin{multline*}
 S_{\infty}'' = Q^2 k T
 \sum_{\bf g}
 \sum_{\substack{d|g_1 g_2 \\ r_0 | (k g_1 g_2)^{\infty}}} 
 \frac{\mu(d)  k r_0}{\varphi(dk r_0)} \sum_{\theta \shortmod{dkr_0}}
 \sum_{\substack{q' \neq 0 \\ (q', k g_1 g_2) = 1}}
\frac{   |q'|  }{\varphi(|q'|)} 
\sumstar_{\chi \shortmod{| q'|}}
\\
   \int_{(2)} \widetilde{w}(2-s)
  Z_{1,1}(s) \nu_{q'}(s) \delta_{{\bf g} k}(s)
    \sum_{\substack{(a_1 b_1, a_2 b_2) = 1 \\ ({\bf a}, {\bf g}) = 1
}}  
\Big(\frac{\gamma_1 a_1 b_2 - \gamma_2 a_2 b_1}{kr_0 q' Q} \Big)^{s}
\\
\frac{(\sgn) \delta_{a_1 b_1}(s) \delta_{a_2 b_2}(s) \alpha_{a_1,b_1}^{(1, {\bf g})} \overline{\alpha}_{a_2, b_2}^{(2, {\bf g})}}{|\gamma_1 a_1 b_2 - \gamma_2 a_2 b_1|}
 \chi \theta\Big(\frac{\gamma_1 a_1 b_2}{\gamma_2 a_2 b_1} \Big)
  \widehat{w}\Big(T \log \frac{\gamma_1 a_1 b_2}{\gamma_2 a_2 b_1} \Big)  \frac{ds}{2 \pi i}.
\end{multline*}

Now we apply a dyadic partition of unity of the form
\begin{equation*}
1 = \sum_{\substack{V \text{ dyadic} }} 
\omega\Big(\frac{\gamma_1 a_1 b_2 - \gamma_2 a_2 b_1}{V}\Big)
\end{equation*}
where $\omega$ is smooth, even, and supported on $[1,2] \cup [-2,-1]$.
By the rapid decay of $\widehat{w}$, 
and recalling \eqref{eq:gammaiaibjsizes},
note that
\begin{equation*}
 \widehat{w}\Big(T \log \frac{\gamma_1 a_1 b_2}{\gamma_2 a_2 b_1} \Big)
 \ll \Big(1 + T \frac{|\gamma_1 a_1 b_2 - \gamma_2 a_2 b_1|}{\gamma_2 a_2 b_1}\Big)^{-A} \ll
 \Big(1 + T \frac{|\gamma_1 a_1 b_2 - \gamma_2 a_2 b_1|}{N/(g_1 g_2)}\Big)^{-A}.
\end{equation*}
Therefore, we may assume that
\begin{equation}
\label{eq:VmaxDef}
 1 \ll V \leq V_{\text{max}} = \frac{N}{g_1 g_2 T} (QkTN)^{\varepsilon},
\end{equation}
absorbing $V > V_{\text{max}}$ into the error term $\mathcal{E}_{\infty}$.

With this partition, we obtain
\begin{multline*}
 S_{\infty}'' = Q^2 k T
  \sum_{\bf g}
 \sum_{1 \ll V \leq V_{\text{max}}} V^{-1}
 \sum_{\substack{d|g_1 g_2 \\ r_0 | (k g_1 g_2)^{\infty}}} 
 \frac{\mu(d) k r_0}{\varphi(dk r_0)} \sum_{\theta \shortmod{dkr_0}}
 \sum_{\substack{q' \neq 0 \\ (q', k g_1 g_2) = 1}}
\frac{   |q'|  }{\varphi(|q'|)} 
\\
\sumstar_{\chi \shortmod{| q'|}}
  \frac{1}{2 \pi i} \int_{(2)} 
  \Big(\frac{V}{kr_0 |q'| Q}\Big)^s
  \widetilde{w}(2-s)
  Z_{1,1}(s) \nu_{q'}(s) \delta_{{\bf g} k}(s)
    \sum_{\substack{(a_1 b_1, a_2 b_2) = 1 \\ ({\bf a}, {\bf g}) = 1
}}  
\\
(\sgn)
\delta_{a_1 b_1}(s) \delta_{a_2 b_2}(s) \alpha_{a_1,b_1}^{(1, {\bf g})} \overline{\alpha}_{a_2, b_2}^{(2, {\bf g})}
 \chi \theta\Big(\frac{\gamma_1 a_1 b_2}{\gamma_2 a_2 b_1} \Big)
 \omega_s\Big(\frac{\gamma_1 a_1 b_2 - \gamma_2 a_2 b_1}{V}\Big)
 \widehat{w}\Big(T \log \frac{\gamma_1 a_1 b_2}{\gamma_2 a_2 b_1} \Big) ds,
\end{multline*}
where $\omega_s(x) = x^{s-1} \omega(x)$.  By shifting the contour far to the right, $q'$ may be truncated at
\begin{equation}
\label{eq:Q*defFamilyAvg}
|q'| \leq Q^* := \frac{V}{k r_0 Q} (QkTN)^{\varepsilon}.
\end{equation}

We next want to apply Lemma \ref{lemma:separationofvariablesDeterminant}.
Note that
\begin{equation*}
 \gamma_1 a_1 b_2 = 
 \frac{\beta_{13} a_1}{g_1} \frac{\beta_{23} b_2}{g_2},
 \qquad
 \gamma_2 a_2 b_1
 =
  \frac{\beta_{14} a_2}{g_1} \frac{\beta_{24} b_1}{g_2},
\end{equation*}
where recall the support of $\alpha$ implies $\beta_{13} a_1 \asymp \beta_{14} a_2 \asymp N_1$ and $\beta_{23} b_2 \asymp \beta_{24} b_1 \asymp N_2$.  We may then freely attach a redundant weight function of the form
\begin{equation*}
 w\Big(\frac{\beta_{13} a_1}{N_1}, \frac{\beta_{24} b_1}{N_2}, 
 \frac{\beta_{14} a_2}{N_1}, \frac{\beta_{23} b_2}{N_2} \Big).
\end{equation*}
Now this is set up to apply Lemma \ref{lemma:separationofvariablesDeterminant} with 
$x_1 = g_1^{-1} \beta_{13} a_1$, $y_1 = g_1^{-1} \beta_{24} b_1$, $x_2 = g_1^{-1} \beta_{14} a_2$, $y_2 = g_2^{-1} \beta_{23} b_2$,
$X = \frac{N_1}{g_1}$ and $Y = \frac{N_2}{g_2}$, 
and with $\omega = \omega_s$.  Observe that with this substitution, then $\gamma_1 a_1 b_2 - \gamma_2 a_2 b_1 = x_1 y_2 - x_2 y_1$, as desired.  This gives
\begin{multline*}
 S_{\infty}'' = Q^2 k T
  \sum_{\bf g}
 \sum_{1 \ll V \leq V_{\text{max}}} V^{-1}
 \int_{\mr^4} G_s(u_1, u_2, u_3, t) 
 \sum_{\substack{d| g_1 g_2 \\ r_0 | (k g_1 g_2)^{\infty}}} 
 \frac{\mu(d) k r_0}{\varphi(dk r_0)} \sum_{\theta \shortmod{dkr_0}}
 \\
 \sum_{\substack{|q'| \leq Q^* \\ (q', k g_1 g_2) = 1}}
\frac{   |q'|  }{\varphi(|q'|)} 
\sumstar_{\chi \shortmod{| q'|}}
  \frac{1}{2 \pi i} \int_{(2)} 
  \Big(\frac{V}{kr_0 |q'| Q}\Big)^s
  \widetilde{w}(2-s)
  Z_{1,1}(s) \nu_{q'}(s) \delta_{k {\bf g}}(s)
\\    
    \sum_{\substack{(a_1 b_1, a_2 b_2) = 1 \\ ({\bf a}, {\bf g}) = 1 \\ (\sgn)
}}  
 \delta_{a_1 b_1}(s) \delta_{a_2 b_2}(s) \alpha_{a_1,b_1}^{(1, {\bf g})} \overline{\alpha}_{a_2, b_2}^{(2, {\bf g})}
 \chi \theta\Big(\frac{\gamma_1 a_1 b_2}{\gamma_2 a_2 b_1} \Big)
 \Big(\frac{\gamma_1 a_1 b_2}{\gamma_2 a_2 b_1}\Big)^{it} dt
\frac{du_1 du_2 du_3}{y_1^{iu_1} y_2^{iu_2} x_2^{iu_3} }
  ds,
\end{multline*}
plus a small error term.  Here $G = G_s$ depends on $s$, via $\omega_s(x) = x^{s-1} \omega(x)$.  We also record
\begin{equation}
\label{eq:RandUdefsFamilyAvgSide}
R = \frac{V g_1 g_2}{N}, \qquad \text{and} \qquad U = \frac{N}{g_1 g_2 V} (QkTN)^{o(1)}.
\end{equation}

Now we shift the contour to the line $\mathrm{Re}(s) = \varepsilon$.  In doing so we cross a pole at $s=1$, and we denote its residue by $S_{\infty}^{(0)}$.  
There is a small but convenient simplification with the sign condition \eqref{eq:sgncondition}, namely that all the summands are independent of $\sgn(\gamma_1 a_1 b_2 - \gamma_2 a_2 b_1)$ and $\sgn(q')$, except for the indicator function that these signs agree.  We may therefore take $q' > 0$.  We also make a small modification by factoring $r_0 = r_g r_k$ where $r_g | (g_1 g_2)^{\infty}$ and $r_k | k^{\infty}$.  
With this simplification and others, we obtain
\begin{multline}
\label{eq:Sinfty0FamilyAvgSide}
 S_{\infty}^{(0)} = Q k T
 \sum_{r_k | k^{\infty}}
  \sum_{\bf g}
 \sum_{1 \ll V \leq V_{\text{max}}} 
\int_{\mr^4} G_1(u_1, u_2, u_3, t) 
\sum_{\substack{d|g_1 g_2 \\ r_g | (g_1 g_2)^{\infty}}} 
 \frac{\mu(d)}{\varphi(dk r_g r_k)} 
 \\
 \sum_{\theta \shortmod{dkr_g r_k}}
 \sum_{\substack{q' \leq Q^* \\ (q', k g_1 g_2) = 1}}
\frac{   \widetilde{w}(1)
  Z_{1,1} \nu_{q'} \delta_{k{\bf g}} }{ \varphi(q')} 
\sumstar_{\chi \shortmod{ q'}}
\\
    \sum_{\substack{(a_1 b_1, a_2 b_2) = 1 \\ ({\bf a}, {\bf g}) = 1
}}  
 \delta_{a_1 b_1} \delta_{a_2 b_2} \alpha_{a_1,b_1}^{(1, {\bf g})} \overline{\alpha}_{a_2, b_2}^{(2, {\bf g})}
 \chi \theta\Big(\frac{\gamma_1 a_1 b_2}{\gamma_2 a_2 b_1} \Big)
 \Big(\frac{\gamma_1 a_1 b_2}{\gamma_2 a_2 b_1}\Big)^{it} dt
\frac{du_1 du_2 du_3}{y_1^{iu_1} y_2^{iu_2} x_2^{iu_3} }.
\end{multline}

Let $S_{\infty}'$ denote the remaining contour integral along $\mathrm{Re}(s) = \varepsilon$.  
Here we obtain
\begin{multline*}
|S_{\infty}'| \lesssim 
Q^2 k T
\int_{(\varepsilon)} |\widetilde{w}(2-s)|
  \sum_{\bf g}
 \sum_{1 \ll V \leq V_{\text{max}}} V^{-1}
 \sum_{\substack{d|g_1 g_2 \\ r_0 | (k g_1 g_2)^{\infty}}} 
 d^{-1}
 \\
\int_{\mr^4} |G_s(u_1, u_2, u_3, t)|
du_1 du_2 du_3 
 \sum_{\theta \shortmod{dkr_0}}
 \sum_{\substack{q' \leq Q^* \\ (q', k g_1 g_2) = 1}}
\sumstar_{\chi \shortmod{ q'}}
\\
\Big| 
\sum_{\substack{(a_1 b_1, a_2 b_2) = 1 \\ ({\bf a}, {\bf g}) = 1
}}  
\delta_{a_1 b_1}(s) \alpha_{a_1,b_1}^{(1, {\bf g})} 
\chi \theta(a_1 \overline{b_1}) \Big(\frac{a_1}{b_1}\Big)^{it}
\cdot
 \delta_{a_2 b_2}(s)  \overline{\alpha}_{a_2, b_2}^{(2, {\bf g})}
 \chi \theta(b_2 \overline{a_2}) \Big(\frac{b_2}{a_2}\Big)^{it}
 \Big|
 dt
  |ds|.
\end{multline*}
A small issue here concerns the dependence of $G_s$ on $s$.  
By the rapid decay of $|\widetilde{w}(2-s)|$, we may truncate the $s$-integral at $|s| \lesssim 1$.
The remark following Lemma \ref{lemma:separationofvariablesDeterminant} shows that the family of functions $G_s$ have a good uniform bound.
We may then truncate the $t$-integral at $U (QkTN)^{o(1)}$.
Lemma \ref{lemma:FamilyAvgBilinearBoundViaMobius} allows us to essentially remove the coprimality condition $(a_1 b_1, a_2 b_2) = 1$; we apply this lemma with $M \ll \frac{N}{g_1 g_2}$ and $\gamma_{a,b}^{(i)} = \delta_{ab}(s) \alpha_{a,b}^{(i, {\bf g})}$.
With these steps, we may then estimate $S_{\infty}'$ in terms of the original norm \eqref{eq:DeltaDef}, giving
\begin{multline}
\label{eq:Sinfty'boundPenultimateFamilyAvg}
|S_{\infty}'| \lesssim 
Q^2 k T
  \sum_{\bf g}
 \sum_{1 \ll V \leq V_{\text{max}}} V^{-1}
 \sum_{\substack{d| g_1 g_2 \\ r_0 | (k g_1 g_2)^{\infty}}}  d^{-1} 
 \\
 U^{-1} 
\overline{\Delta}\Big(Q^*, dkr_0,U, \frac{N}{g_1 g_2}\Big) | \alpha_{a_1, b_1}^{(1, {\bf g})} \alpha_{a_2, b_2}^{(2, {\bf g})}|,
\end{multline}
where recall $U$ is given by \eqref{eq:RandUdefsFamilyAvgSide} and $Q^*$ was defined by 
\eqref{eq:Q*defFamilyAvg}.  Note $UV = \frac{N}{g_1 g_2} (QkTN)^{o(1)}$.
It is convenient to write $V = V_{\text{max}}/P$, where $1 \ll P \ll V_{\text{max}}$, in which case \eqref{eq:Sinfty'boundPenultimateFamilyAvg} simplifies as
\begin{multline}
\label{eq:Sinfty'SemiFinalBoundFamilyAvgSide}
|S_{\infty}'| \lesssim
\frac{Q^2 k T}{N}
  \sum_{\bf g}
 \sum_{1 \ll P \ll V_{\text{max}}} 
 \sum_{\substack{d| g_1 g_2 \\ r_0 | (k g_1 g_2)^{\infty}}} 
 \frac{1}{d}
 \\
 g_1 g_2
\overline{\Delta}\Big(\frac{N}{QkT g_1 g_2 r_g  r_k P}, dkr_g r_k, P T , \frac{N}{g_1 g_2} \Big)  | \alpha_{a_1, b_1}^{(1, {\bf g})} \alpha_{a_2, b_2}^{(2, {\bf g})}|.
\end{multline}
Recalling the definition \eqref{eq:Delta''def}, this completes the proof.
\end{proof}

\subsection{Conclusion}
Now we use Propositions \ref{prop:SleqYFamilyAvg} and \ref{prop:SinftyFamilyAvg} to prove Theorem \ref{thm:familyavgthm}.  Recall that we need to show that $S_{>Y}$ satisfies \eqref{eq:Delta2Bound}, that is
\begin{equation*}
 S_{>Y} \lesssim
 |\alpha|^2 \Big(
 Q^2 kT + \frac{Q^2 k T}{N} \overline{\Delta'}\Big(\frac{N}{kQT}, k, T, N\Big)
 \Big) ,
\end{equation*}
where for convenience to the reader we recall the definition \eqref{eq:Delta''def}:
\begin{equation*}
 \Delta'(Q,k,T,N) = 
 \max_{\substack{X, R, U, C \in \mr_{\geq 1}, \ell \in \mz_{>0} \\  X R^2 \ell U \leq Q^2 k T \\ X \leq C }} 
  X
 \Delta\Big(R, \ell, U, \frac{N}{C} \Big).
\end{equation*}

We have a decomposition
\begin{equation*}
S_{>Y} = S_{\infty} - S_{\leq Y} = S_{\infty}^{\text{diag}} + S_{\infty}' - S_{\leq Y}' + (S_{\infty}^{(0)} - S_{\leq Y}^{(0)}) + \mathcal{E}_{\infty}.
\end{equation*}
The diagonal term $S_{\infty}^{\text{diag}}$ is acceptable for Theorem \ref{thm:familyavgthm}, as is $\mathcal{E}_{\infty}$.

Now we turn to the terms $S_{\leq *}'$.  Recall the definitions \eqref{eq:Q*defFamilyAvg} and \eqref{eq:VmaxDef}.
We choose 
\begin{equation}
Y = (QkTN)^{\varepsilon} \frac{N}{QkT},
\end{equation}
with a value of $\varepsilon$ so that
% \begin{equation}
% \label{eq:Q*simplifiedFamilyAvg}
% Q^* = \frac{V}{Qk r_g  r_k} (QkTN)^{\varepsilon} 
% %\leq \frac{N (QkTN)^{\varepsilon}}{QkT} \frac{1}{g_1 g_2 r_g r_k}
% \leq \frac{N (QkTN)^{\varepsilon}}{QkT} \frac{1}{g_1 g_2 r_g  r_k}
% .
% \end{equation}
% We choose 
%\eqref{eq:Q*simplifiedFamilyAvg}, so that 
when $V = V_{\text{max}}$, then
$Q^* = \frac{Y}{g_1 g_2 r_g  r_k}$.
Using the assumption $Q^2 k T \gg N^{1-\varepsilon}$, it is easy to check that \eqref{eq:SleqY'boundFamilyAvgSide} is acceptable for Theorem \ref{thm:familyavgthm}, and also that $Y \leq Q/100$, so this is a valid choice of $Y$.  Moreover, \eqref{eq:Sinfty'boundFamilyAvgSide} directly shows that $S_{\infty}'$ is bounded in accord with the theorem.

Finally, consider the polar terms from $s=1$, namely $S_{\infty}^{(0)}$ and $S_{\leq Y}^{(0)}$ given by \eqref{eq:Sinfty0FamilyAvgSide} and \eqref{eq:SleqY0formulaFamilyAvg}.  We simplify $S_{\infty}^{(0)}$, continuing with \eqref{eq:Sinfty0FamilyAvgSide}.  
We reverse the orders of summation between $V$ and $q'$; 
the condition $q' \leq Q^* = \frac{CV}{Q k r_g r_k}$ (where $C$ here is shorthand for $(QkTN)^{\varepsilon}$) becomes instead $V > C^{-1} q'Q k r_k r_g$ (on the inside) and $q' \leq \frac{Y}{r_g r_k g_1 g_2}$ (on the outside).  We then write $S_{\infty}^{(0)} = S_{\infty, 1}^{(0)} - S_{\infty, 2}^{(0)}$, where $S_{\infty, 1}^{(0)}$ has $V$ unconstrained, and $S_{\infty, 2}^{(0)}$ has $V \leq C^{-1} q' Q k r_k r_g$.  A pleasant feature of $S_{\infty, 1}^{(0)}$ is that the sum over $V$ re-assembles the partition of unity, since $G_1$ corresponds to $\omega_{s}(x) \vert_{s=1} = \omega(x)$.  We also re-open the definition of $\widehat{w}$.  Together, these steps give
\begin{multline}
\label{eq:Sinfty10formula}
 S_{\infty,1}^{(0)} = Q k 
 \sum_{r_k | k^{\infty}}
  \sum_{\bf g}
  \intR w\Big(\frac{t}{T}\Big)
\sum_{\substack{d|g_1 g_2 \\ r_g | (g_1 g_2)^{\infty}}} 
 \frac{\mu(d)}{\varphi(dk r_g r_k)} \sum_{\theta \shortmod{dkr_g r_k}}
 \\
 \sum_{\substack{q' \leq \frac{Y}{r_k r_g g_1 g_2} \\ (q', k g_1 g_2) = 1}}
\frac{   1 }{ \varphi(q')} 
\sumstar_{\chi \shortmod{ q'}}
  \widetilde{w}(1)
  Z_{1,1} \nu_{q'} \delta_{k{\bf g}}
    \sum_{\substack{(a_1 b_1, a_2 b_2) = 1 \\ ({\bf a}, {\bf g}) = 1
}}  
\\
 \delta_{a_1 b_1} \delta_{a_2 b_2} \alpha_{a_1,b_1}^{(1, {\bf g})} \overline{\alpha}_{a_2, b_2}^{(2, {\bf g})}
 \chi \theta\Big(\frac{\gamma_1 a_1 b_2}{\gamma_2 a_2 b_1} \Big)
 \Big(\frac{\gamma_1 a_1 b_2}{\gamma_2 a_2 b_1}\Big)^{it} dt.
\end{multline}
Next we further cut this sum into four pieces, via
\begin{equation}
\label{eq:SinftyFourPieces}
 \sum_{q' \leq \frac{Y}{r_k r_g g_1 g_2}} 
 = \sum_{q' \leq \frac{Y}{r_k}} - \sum_{\frac{Y}{r_k g_1 g_2} < q' \leq \frac{Y}{r_k}}
 - 
 \sum_{\frac{Y}{r_k r_g d g_1 g_2} < q' \leq \frac{Y}{r_k g_1 g_ 2}}
 +
 \sum_{\frac{Y}{r_k r_g d g_1 g_2} \leq q' \leq \frac{Y}{r_k r_g g_1 g_2}}.
\end{equation}
Call the corresponding sums $S_{i}$, for $i=1,2,3,4$.  There is a pleasant simplification available for $S_1$, $S_2$, and $S_3$.  In these three sums, both the summation conditions in \eqref{eq:SinftyFourPieces}, as well as all the summands in \eqref{eq:Sinfty10formula}, depend only on the \emph{product} $d r_g=D$ (say), with the exception of the presence of $\mu(d)$.  
M\"obius inversion means that the sum over $d|D$ detects $D=1$.  This immediately implies $S_3 = 0$.  Moreover, we see that $S_1 = S_{\leq Y}^{(0)}$, which is a crucial cancellation.  The sum $S_2$ becomes
\begin{multline*}
 S_{2} = - Q k 
 \sum_{r_k | k^{\infty}}
  \sum_{\bf g}
  \intR w\Big(\frac{t}{T}\Big)
 \frac{1}{\varphi(kr_k)} \sum_{\theta \shortmod{kr_k}}
 \\
 \sum_{\substack{\frac{Y}{r_k g_1 g_2} < q' \leq \frac{Y}{r_k} \\ (q', k g_1 g_2) = 1}}
\frac{   1 }{ \varphi(q')} 
\sumstar_{\chi \shortmod{ q'}}
  \widetilde{w}(1)
  Z_{1,1} \nu_{q'} \delta_{k{\bf g}}
    \sum_{\substack{(a_1 b_1, a_2 b_2) = 1 \\ ({\bf a}, {\bf g}) = 1
}}  
\\
 \delta_{a_1 b_1} \delta_{a_2 b_2} \alpha_{a_1,b_1}^{(1, {\bf g})} \overline{\alpha}_{a_2, b_2}^{(2, {\bf g})}
 \chi \theta\Big(\frac{\gamma_1 a_1 b_2}{\gamma_2 a_2 b_1} \Big)
 \Big(\frac{\gamma_1 a_1 b_2}{\gamma_2 a_2 b_1}\Big)^{it} dt.
\end{multline*}
Similarly to the estimation of $S_{\infty}'$, using Lemma \ref{lemma:FamilyAvgBilinearBoundViaMobius}  we obtain
\begin{equation*}
|S_2| \lesssim |\alpha|^2 Qk  
\max_{\substack{{\bf g}, r_k | k^{\infty}  }} \frac{1}{k r_k}
 \max_{\frac{Y}{g_1 g_2 r_k} \leq Q' \leq \frac{Y}{r_k}} \frac{1}{Q'} 
 \overline{\Delta}\Big(Q', k r_k, T, \frac{N}{g_1 g_2}\Big).
\end{equation*}
Write $Q' = \frac{Y}{r_k P}$, where $1 \leq P \leq g_1 g_2$, giving 
\begin{equation*}
S_2 \lesssim |\alpha|^2 \frac{Q^2 k T}{N}
\max_{\substack{{\bf g}, r_k | k^{\infty}  \\ 1 \leq P \leq g_1 g_2}} P
 \overline{\Delta}\Big(\frac{N}{QkT r_k P}, k r_k, T, \frac{N}{g_1 g_2}\Big).
\end{equation*}
This is consistent with Theorem \ref{thm:familyavgthm}.
The sum $S_4$ is similar in shape, and we obtain
\begin{equation*}
 S_4 \lesssim |\alpha|^2 \frac{Q^2 k T}{N} 
 \max_{\substack{{\bf g}, r_k | k^{\infty} \\ r_g | (g_1 g_2)^{\infty}  }} 
 \max_{1 \leq P \leq d} \frac{Pg_1 g_2}{d}
 \overline{\Delta}\Big(\frac{N}{Qk T r_k r_g g_1 g_2 P}, d k  r_k r_g, T, \frac{N}{g_1 g_2}\Big),
\end{equation*}
which is acceptable.

Next we turn to estimating $S_{\infty,2}^{(0)}$.  Our expression for this is identical to \eqref{eq:Sinfty10formula}, except we have an additional weight function of the form
\begin{equation}
\Omega(x) := \sum_{1 \ll V \lesssim q' Q k r_k r_g} \omega\Big(\frac{x}{V}\Big),
\qquad \text{with} 
\qquad x=\gamma_1 a_1 b_2 - \gamma_2 a_2 b_1.
\end{equation}
The function $\Omega(x)$ is identically $1$ for $1 \leq |x| \lesssim q' Q k r_k r_g$, but it vanishes at $x=0$.  Let $\Omega_0(x) = 1- \Omega(x)$ for $|x| \leq 1$, and such that $\Omega_0(x) =0$ for $|x| \geq 1$. Let $S_{\infty, 2}'$ denote the same expression as $S_{\infty, 2}^{(0)}$ but with $\Omega$ replaced by $\Omega_1 := \Omega + \Omega_0$, and let $S_{\infty,2}^{\text{diag}} = S_{\infty, 2}' - S_{\infty, 2}^{(0)}$.  Indeed, $S_{\infty, 2}^{\text{diag}}$ is supported only on $\gamma_1 a_1 b_2 = \gamma_2 a_2 b_1$.  By similar reasoning as in \eqref{eq:SinftyDiagBound}, we obtain
\begin{equation*}
 |S_{\infty, 2}^{\text{diag}}| \lesssim QkT Y |\alpha|^2 \lesssim N |\alpha|^2.
\end{equation*}
Since $N \lesssim Q^2 k T$, this is no worse than \eqref{eq:SinftyDiagBound}.

Finally, consider $S_{\infty, 2}'$.  The function $\Omega_1$ meets the conditions of Lemma \ref{lemma:separationofvariablesDeterminant}, with $V$ taking the value $C^{-1} q' Q k r_k r_g$.  
Hence we obtain an expression of the form
\begin{multline*}
 S_{\infty,2}' = Q k T
 \sum_{r_k | k^{\infty}}
  \sum_{\bf g}
  \int_{\mr^4} 
\sum_{\substack{d|g_1 g_2 \\ r_g | (g_1 g_2)^{\infty}}} 
 \frac{\mu(d)}{\varphi(dk r_g r_k)} \sum_{\theta \shortmod{dkr_g r_k}}
 \\
 \sum_{\substack{q' \leq \frac{Y}{r_k r_g g_1 g_2} \\ (q', k g_1 g_2) = 1}}
\frac{   1 }{ \varphi(q')} 
\sumstar_{\chi \shortmod{ q'}}
  \widetilde{w}(1)
  Z_{1,1} \nu_{q'} \delta_{k{\bf g}}
    \sum_{\substack{(a_1 b_1, a_2 b_2) = 1 \\ ({\bf a}, {\bf g}) = 1
}}  
G(t, u_1, u_2, u_3)
\\
 \delta_{a_1 b_1} \delta_{a_2 b_2} \alpha_{a_1,b_1}^{(1, {\bf g})} \overline{\alpha}_{a_2, b_2}^{(2, {\bf g})}
 \chi \theta\Big(\frac{\gamma_1 a_1 b_2}{\gamma_2 a_2 b_1} \Big)
 \Big(\frac{\gamma_1 a_1 b_2}{\gamma_2 a_2 b_1}\Big)^{it} dt \frac{du_1 du_2 du_3}{y_1^{u_1} y_2^{u_2} x_2^{u_3}}.
\end{multline*}
The bound on $G$ is given by \eqref{eq:GboundFamilyAvgSide}, with now
\begin{equation*}
U = \frac{N}{q' Q k r_k r_g g_1 g_2 } (QkTN)^{o(1)}.
\end{equation*}
The estimations are similar to those of $S_{\infty}'$, $S_2$, and $S_4$, and 
we obtain
\begin{equation*}
|S_{\infty,2}'| \lesssim 
|\alpha|^2 Qk T 
\max_{\substack{{\bf g}, r_k | k^{\infty}  \\ r_g | (g_1 g_2)^{\infty} }} \frac{1}{d k r_g r_k}
 \max_{ Q' \leq \frac{Y}{r_k r_g g_1 g_2}} \frac{1}{UQ'} 
 \overline{\Delta}\Big(Q', d k r_k r_g, U, \frac{N}{g_1 g_2}\Big).
\end{equation*}
This simplifies as
\begin{equation*}
|S_{\infty,2}'| \lesssim 
|\alpha|^2 \frac{Q^2 k T}{N} 
\max_{\substack{{\bf g}, r_k | k^{\infty}  }} 
 \max_{ 1 \ll P \lesssim \frac{N}{Q k T r_k r_g g_1 g_2 }} \frac{   g_1 g_2}{d} 
 \overline{\Delta}\Big(\frac{N}{QkT r_k r_g g_1 g_2 P}, d k r_k r_g, TP, \frac{N}{g_1 g_2}\Big).
\end{equation*}
One checks this is consistent with Theorem \ref{thm:familyavgthm}, which completes its proof.

\end{document}